\documentclass[12pt]{amsart}

\usepackage{extarrows}
\usepackage{amsaddr}
\usepackage{tikz} \usepackage{tikz-cd}
\usepackage{amssymb}
\usepackage{amsmath}
\usepackage{mathtools}
\usepackage{mathrsfs}  
\usepackage{comment}
\usepackage{setspace}
\usepackage{stmaryrd}
\usepackage{hyperref}
\hypersetup{
    breaklinks=true,
    colorlinks=true,
    citecolor=red,
    linkcolor=blue,
    filecolor=magenta,      
    urlcolor=cyan,
}

\usepackage[top=2.6cm,
bottom=2.6cm,
left=2.3cm,
right=2.3cm]{geometry}

\usepackage{collcell}

\usepackage{tkz-euclide}
\usetikzlibrary {decorations.pathmorphing,shadows}

\newcommand{\shiftdown}[1]{\smash{\raisebox{-.5\normalbaselineskip}{#1}}}
\newcommand{\BoHo}{H^{\BM}}

\newcommand{\Sh}[1]{\mathcal{#1}}

\newcolumntype{C}{>{\collectcell\shiftdown}c<{\endcollectcell}}

\DeclareSymbolFont{extraup}{U}{zavm}{m}{n}
\DeclareMathSymbol{\varheart}{\mathalpha}{extraup}{86}
\DeclareMathSymbol{\vardiamond}{\mathalpha}{extraup}{87}

\DeclareMathOperator\GL{GL}

\DeclareMathOperator\Char{Char}
\DeclareMathOperator\weight{weight}
\DeclareMathOperator\C{\mathbb C}
\DeclareMathOperator{\Cs}{\mathbb{C}^*}

\DeclareMathOperator\N{\mathbb N}

\DeclareMathOperator\h{\hbar}
\DeclareMathOperator\Hom{Hom}
\DeclareMathOperator\End{End}
\DeclareMathOperator\Fl{Fl}

\DeclareMathOperator\Chamb{\mathfrak{C}}
\DeclareMathOperator\Stab{Stab}
\DeclareMathOperator\uStab{\underline{Stab}}
\DeclareMathOperator\Rep{Rep}
\DeclareMathOperator\ch{ch}
\DeclareMathOperator\gl{\mathfrak{gl}}
\DeclareMathOperator\dd{\mathbf{d}}
\DeclareMathOperator\ww{\mathbf{w}}
\DeclareMathOperator\fg{\mathfrak{g}}
\DeclareMathOperator\FM{\mathfrak{M}}
\DeclareMathOperator\FN{\mathfrak{N}}

\DeclareMathOperator\BM{BM}
\DeclareMathOperator\codim{codim}

\DeclareMathOperator\mC{\mC}
\DeclareMathOperator\Pic{Pic}
\DeclareMathOperator{\Base}{\mathscr{B}}
\DeclareMathOperator{\rk}{rk}

\DeclareMathOperator{\liet}{\mathfrak{t}}
\DeclareMathOperator{\lien}{\mathfrak{n}}
\DeclareMathOperator{\lieg}{\mathfrak{g}}
\DeclareMathOperator{\vmult}{\mathsf{m}}
\DeclareMathOperator{\Ell}{Ell}
\DeclareMathOperator{\ext}{ext}
\DeclareMathOperator\Kthy{\mathfrak{K}}
\newcommand{\ul}{\underline}
\newcommand{\hata}{\hat a}
\newcommand{\wt}{\widetilde}
\newcommand{\ol}{\overline}
\newcommand{\ag}[2]{\hat a \left(\frac{#1}{#2}\right)}

\DeclareMathOperator{\Ht}{\mathcal{H}}

\def\mC{\mathfrak C}
\def\W{W}
\def\Wt{\widetilde{W}}
\DeclareMathOperator\Shuffle{Shuffle}
\def\One{\mathbf{1}}
\def\eu{eu}
\def\Yt{\mathcal Y}

\def\R{\mathbb{R}}
\def\Z{\mathbb{Z}}
\def\th{\vartheta}
\def\thh{\th\!}
\def\r{r}
\def\c{c}
\DeclareMathOperator\BCT{BCT}
\DeclareMathOperator{\Pe}{\mathbb{P}}

\newcommand\bs{{\color{blue} \char`\\}}
\newcommand\fs{{\color{red}/}}
\def\ttt#1{\text{\tt #1}}
\DeclareMathOperator\Zb{\mathcal Z}
\newcommand{\Ab}{\mathcal A}
\newcommand{\Xb}{\mathcal X}
\def\NS5{N\!S5}
\def\Ch{X}
\def\DD{\mathcal D}
\def\ch{ch}
\def\id{id}

\usepackage{mathabx,epsfig}

\DeclareMathOperator\MM{\mathbb{M}}
\DeclareMathOperator\NN{\mathbb{N}}
\DeclareMathOperator\Attr{Att}
\DeclareMathOperator\At{\mathbb A}
\DeclareMathOperator\Tt{\mathbb T}
\newcommand{\Att}[1]{\Attr_{\mathfrak{#1}}}
\def\Cs{\C^{\times}}

\DeclareMathOperator\TsP{T^*\!\Pe}
\DeclareMathOperator\TP{T\!\Pe}
\DeclareMathOperator\TsGr{T^*Gr}
\DeclareMathOperator\pt{\ast}
\DeclareMathOperator{\Chern}{ch}
\DeclareMathOperator{\Corr}{\mathfrak{U}}
\DeclareMathOperator{\Spec}{Spec}
\DeclareMathOperator{\chern}{ch}
\def\X{X}
\def\NS5{\text{NS5}}
\def\D5{\text{D5}}
\def\no{}

\makeatletter
\renewcommand{\oast}{\mathbin{\mathpalette\make@circled\ast}}
\newcommand{\make@circled}[2]{%
  \ooalign{$\m@th#1\smallbigcirc{#1}$\cr\hidewidth$\m@th#1#2$\hidewidth\cr}%
}
\newcommand{\smallbigcirc}[1]{%
  \vcenter{\hbox{\scalebox{0.77778}{$\m@th#1\bigcirc$}}}%
}
\makeatother

\newtheorem{fact}{Fact}[section]
\newtheorem{lemma}[fact]{Lemma}
\newtheorem{theorem}[fact]{Theorem}
\newtheorem{definition}[fact]{Definition}
\newtheorem{eexample}[fact]{Example}
\newtheorem{ccounterexample}[fact]{Counterexample}
\newtheorem{rremark}[fact]{Remark}
\newenvironment{remark}{\begin{rremark} \rm}{\end{rremark}}
\newenvironment{example}{\begin{eexample} \rm}{\end{eexample}}

\newtheorem{proposition}[fact]{Proposition}
\newtheorem{corollary}[fact]{Corollary}

\title[]{Hall algebra multiplication for \\stable envelopes on bow varieties}

\author{T. M. Botta$^\diamond$, R. Rim\'anyi$^\star$}

\email{tommaso.botta@columbia.edu}

\email{rimanyi@email.unc.edu}

\address{$^\diamond$ Department of Mathematics, Columbia University, New York, USA \\
$^\star$ Department of Mathematics, University of North Carolina at Chapel Hill, USA}

\begin{document}
\onehalfspacing

\begin{abstract} 
Elliptic stable envelopes are fundamental components in the geometric realization of quantum group representations. We present a formula for elliptic stable envelopes on type~A Cherkis bow varieties, as a product of simple basic objects in an elliptic cohomology Hall algebra. Combined with the 3d mirror symmetry property of elliptic stable envelopes, our result implies theta function identities for any pair of 01-matrices sharing the same row and column sums.
\end{abstract}

\maketitle

\setcounter{tocdepth}{1} 
\tableofcontents

\section{Introduction}

For certain holomorphic symplectic manifolds $X$ equipped with a torus action, Aganagic and Okounkov introduced the notion of {\em elliptic stable envelopes} \cite{aganagic2016elliptic, okounkov2020inductiveI}. These elliptic stable envelopes form a collection of elliptic cohomology classes $\Stab(f)$ associated with the torus-fixed points $f$ of $X$.

Stable envelopes have proven to be extremely useful in geometric representation theory and enumerative geometry. In essence, they geometrize quantum group actions \cite{maulik2012quantum, aganagic2016elliptic,Felder2018}, and their substitutions give rise to the monodromy matrix of physically relevant qKZ difference equations, which govern the curve-counting vertex functions \cite{Okounkov_lectures, aganagic2016elliptic, Aganagic:2017gsx}. Furthermore, the theory of stable envelopes overlaps with that of $\h$-deformed characteristic classes of singularities \cite{RW_elliptic}.

In this paper, we present explicit formulas for elliptic stable envelopes for a broad class of ambient spaces $X$, known as Cherkis bow varieties of type A. Before discussing bow varieties or the nature of our formulas, let us discuss our motivation.

Our primary focus is the role of elliptic stable envelopes in 3d mirror symmetry. Certain three-dimensional quantum field theories with $N=4$ supersymmetries naturally appear in pairs, suggesting that their mathematical incarnations should also be related \cite{WebsterYoo}. These conjectured relationships often involve abstract structures, such as 2-categories. However, stable envelopes are an exception: as cohomology classes, they are {\em formulas}. Consequently, their 3d mirror symmetry property---established in \cite{BR23} building on earlier works \cite{Rimanyi_2019full, RSVZ_Gr3d, RRAW_Langlands}---is an explicit identity between formulas.

More precisely, 3d mirror symmetry for elliptic stable envelopes on bow varieties has the following form. Let $X$ and $X^!$ be mirror dual bow varieties, and let $f,g$ be torus fixed points on $X$ with corresponding torus fixed points $f^!, g^!$ on $X^!$. Then 
\begin{equation}\label{eq:mirror_in_intro}
\frac{\Stab(f)|_g}{\Stab(g)|_g} \leftrightarrow \pm \frac{\Stab(g^!)|_{f^!}}{\Stab(f^!)|_{f^!}}
\end{equation}
where $\leftrightarrow$ indicates equality after an exchange of variables, specifically $a\leftrightarrow z$, $\h\leftrightarrow \h^{-1}$.  
Another, related, manifestation of 3d mirror symmetry relates the equivariant K-theoretic counts of quasimaps \cite{ciocan_quasimaps, Okounkov_lectures} from $\mathbb{P}^1$ to mirror dual bow varieties $X$ and $X^!$. Explicitly, it states that suitably normalized generating functions $V$ and $V^!$ of the curve counts satisfy
\begin{equation}
    \label{eq: mirror_vertex_in_intro}
    V^!|_{f^!} \leftrightarrow \pm\frac{\Stab(f)|_g}{\Stab(g)|_g} V|_g.
\end{equation}

The proof of \eqref{eq:mirror_in_intro} in \cite{BR23} and \eqref{eq: mirror_vertex_in_intro} in \cite{BDBowVertex} (in preparation) are based on representation-theoretic geometry rather than a direct comparison of explicit formulas. Therefore, if explicit formulas for stable envelopes were available, equations \eqref{eq:mirror_in_intro} and \eqref{eq: mirror_vertex_in_intro} would translate into concrete identities among these formulas.

Our main result, Theorem~\ref{thm:main}, provides an explicit formula for elliptic stable envelopes on bow varieties. This formula is expressed as a sum of products of certain elliptic (theta-) functions and their inverses. Therefore
the 3d mirror symmetry relation \eqref{eq:mirror_in_intro} manifests itself as an identity among theta functions, for any pair of fixed points $f, g$.  

The combinatorial structures encoding fixed points on bow varieties are 01-matrices with prescribed row and column sums. Consequently, equation \eqref{eq:mirror_in_intro} generates a wealth of theta function identities—one for each pair of 01-matrices with the same row and column sums.

The prototype identity for theta functions is the renowed Fay trisecant identity, see \eqref{eq:fay} below. After denominators are cleared, this identity consists of three terms, each expressed as a product of four $\th$ factors. The sizes of the identities we obtain for small bow varieties are illustrated in Figure~\ref{fig:coordinates_intro}; see \S~\ref{sec:mirror} for more details.

\begin{figure}
\[
\begin{tikzpicture}[scale=.4]
   \tkzInit[xmax=17,ymax=9,xmin=0,ymin=0]
   \tkzGrid
   \draw[thick,->] (0,0) to (17.5,0);
      \node at (10,-1.5)  {number of $\th$-factors in each term};
   \draw[thick,->] (0,0) to (0,9.5);
      \node[rotate=90] at (-1.7,5.3) {number of terms};
   \node at (4,3) {$\bullet$};
   \node at (6,4) {$\bullet$};
   \node at (7,4) {$\bullet$};
   \node at (8,4) {$\bullet$};
   \node at (8,5) {$\bullet$};
   \node at (9,5) {$\bullet$};
   \node at (9,6) {$\bullet$};
   \node at (10,5) {$\bullet$};
   \node at (10,6) {$\bullet$};
   \node at (11,6) {$\bullet$};
   \node at (11,7) {$\bullet$};
   \node at (11,8) {$\bullet$};
   \node at (12,7) {$\bullet$};
   \node at (12,8) {$\bullet$};
   \node at (13,7) {$\bullet$};
   \node at (13,8) {$\bullet$};
   \node at (14,8) {$\bullet$};
      \node at (13,9) {$\bullet$};
      \node at (14,9) {$\bullet$};
      \node at (15,9) {$\bullet$};
   \node at( 4,-.5) {$4$}; \node at( 8,-.5) {$8$}; \node at(12,-.5) {$12$}; \node at(16,-.5) {$16$};
   \node at (-.5,2) {$2$}; \node at (-.5,4) {$4$}; \node at (-.5,6) {$6$}; \node at (-.5,8) {$8$};
\draw[thick, red,->] (19,3.7) to [out=-150,in=-20] (4.2,2.8);
\draw (23,4) node [fill=yellow!25] {Fay's trisecant identity};
   \end{tikzpicture}
\]
\caption{The sizes of $\th$ function identities obtained from 3d mirror symmetry for elliptic stable envelopes on `small' bow varieties, cf. \S~\ref{sec:mirror}.} 
\label{fig:coordinates_intro}
\end{figure}
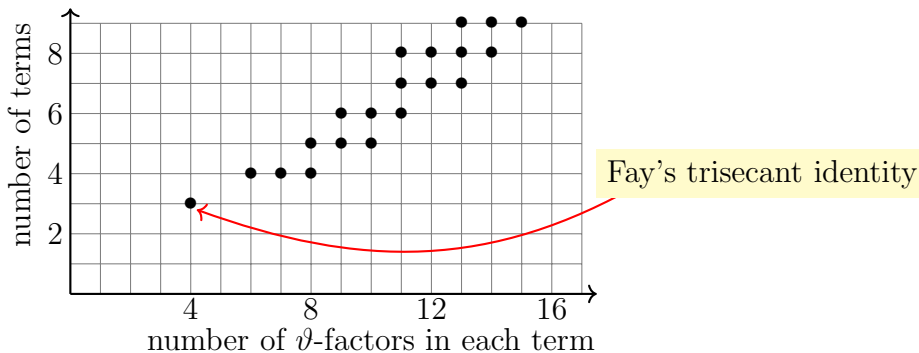

In addition to the fact that bow varieties are closed under 3d mirror symmetry, another key feature of these spaces underpins our main theorem: the computation of their stable envelopes as ``shuffles'' of {\em very simple functions}.

To appreciate this fact, recall that among type A bow varieties, one finds type A Nakajima quiver varieties---though this smaller set is not closed under 3d mirror symmetry. Stable envelope formulas for quiver varieties are already known \cite{smirnov2018elliptic, Dinkins}, and they involve ``shuffling'' some basic formulas. Each of these basic formulas corresponds to a single vertex of the quiver, with a partition assigned to it. As a result, even the basic formulas, from which the recursion begins, exhibit a rich internal structure that reflects the combinatorics of the associated partition.

In contrast, the combinatorial codes for fixed points on bow varieties are tie diagrams that are pictures like this
\begin{equation} \label{eq:tiediagram in Intro}
  \begin{tikzpicture}[scale=.3,baseline=12]
\begin{scope}[yshift=0cm]
\draw [thick,red] (0.5,0) --(1.5,2); 
\draw[thick] (1,1)--(2.5,1)  -- (19,1);
\draw [thick,red](3.5,0) --(4.5,2);  
\draw [thick](4.5,1)--(5.5,1) -- (6.5,1);
\draw [thick,red](6.5,0) -- (7.5,2);  
\draw [thick](7.5,1) --(8.5,1) -- (9.5,1); 
\draw[thick,red] (9.5,0) -- (10.5,2);  
\draw[thick] (10.5,1) --(11.5,1)  -- (12.5,1); 
\draw [thick,blue](13.5,0) -- (12.5,2);   
\draw [thick](13.5,1) --(14.5,1)  -- (15.5,1);
\draw[thick,blue] (16.5,0) -- (15.5,2);  
\draw [thick](16.5,1) --(17.5,1)  -- (18.5,1);  
\draw [thick,blue](19.5,0) -- (18.5,2);  
\draw [dashed, black](1.5,2.2) to [out=45,in=135] (12.5,2.2);
\draw [dashed, black](4.5,2.2) to [out=35,in=145] (15.5,2.2);
\draw [dashed, black](4.5,2.2) to [out=45,in=135] (18.5,2.2);
\draw [dashed, black](7.5,2.2) to [out=25,in=155] (12.5,2.2);
\draw [dashed, black](10.5,2.2) to [out=25,in=155] (15.5,2.2);
\end{scope}
\end{tikzpicture}.
\end{equation}
 Our formula for the corresponding stable envelope is constructed by ``shuffling'' basic formulas associated with the ties (depicted by the dotted lines in the diagram). The basic formula for a single tie is simply a product of theta functions (or, in the case of cohomology, it reduces to an $\h$-power); see \S~\ref{sec:one tie functions} for details. This means that the building blocks of our formulas are trivial, and the {\em entire} structural complexity of the stable envelopes is encoded in the combinatorics of the tie diagram. Ties are more basic building blocks in the combinatorics of bow and quiver varieties than partitions: if partitions are molecules, ties are the atoms. (For more discussion on the relations between partitions and ties, see \cite[Figs.~8,9]{rimanyi2020bow}.)

\smallskip


From the geometric perspective, the ``shuffling" operation corresponds to the multiplication in a suitably framed version of a type A (elliptic) cohomological Hall algebra. The connection between such a multiplication and the stable envelopes of quiver varieties was developed in \cite{botta2021shuffle, botta2023framed, BD}. Hence, a byproduct of the present work is the extension of this framework to bow varieties. This is achieved by linking the stable envelopes of a general bow variety to those of a partial flag variety, where the analogous structure is known to exist. In the language of representation theory, this approach corresponds to relating the R-matrices\footnote{R-matrices arise as transition matrices between stable envelopes associated with different chambers, i.e., inequivalent generic $\Cs$-actions} of the fundamental representations of $\mathfrak{sl}_n$ to those of the vector representation, via the so-called fusion procedure. On the Hall algebra side, this fusion procedure manifests as a modification of the torus weights assigned to the framing edges of the associated quiver.

\medskip

After introducing the necessary geometry and algebra, we state our formula for stable envelopes in \S~\ref{sec:shuffle_formula_for Stab}, and prove it in \S~\ref{sec:proof}. For completeness, we develop the theory and formulas not only for elliptic, but also for K theoretic and cohomological stable envelopes. In \S~\ref{sec:mirror}, we derive the theta function identities that arise from our main theorem and 3d mirror symmetry.

\bigskip

\noindent{\bf Acknowledgments.} 
The first author was supported by grant P500PT-222219 of the Swiss National Science Foundation (SNSF) and a Postdoctoral Fellowship at Columbia University. The second author was supported by NSF grant 2200867. We are grateful to H. Dinkins and S. Jindal for useful discussions on 3d mirror symmetry of bow varieties.

\section{Bow varieties}

\subsection{Separated brane diagrams}
\label{sec:brane diagrams}

Vectors of non-negative integers of the types
\[
\begin{array}{lcrcl}
d^+&= &     & (d_0, & d_1,\ \ d_2,\ \ d_3, \ldots)\\
d^-&=  & (\ldots, d_{-3},\ \ d_{-2},\ \ d_{-1}, & d_0) & \\
d  &=&   (\ldots, d_{-3},\ \ d_{-2},\ \ d_{-1}, & d_0, & d_1,\ \ d_2,\ \ d_3, \ldots),
\end{array}
\]
will be called `right', `left', and `full' dimension vectors, if  
\[
d_{i-1} \leq d_{i} \ \text{for } i\leq 0, \qquad
d_{i} \geq d_{i+1} \ \text{for } i\geq 0,
\]
as well as $d_{i}=d_{-i}=0$ for $i\gg 0$. In essence, both left and right dimensions vectors are partitions, and a full dimension vector is the two partitions glued in the middle. The differences $c_i=d_{i-1}-d_{i}$ for $i\geq 1$ are called D5 charges, and the differences $r_i=d_{-i+1}-d_{-i}$ for $i\geq 1$ are called NS5 charges ($ch$).

We can also encode a full dimension vector by a combinatorial object called  (type A, separated) brane diagram. For example, the brane diagram   
\begin{equation}
    \label{eq: example brane diagram}
    \DD=
\begin{tikzpicture}[baseline=7,scale=.4]
\begin{scope}[yshift=0cm]
\draw [thick,red] (0.5,0) --(1.5,2); 
\draw[thick] (1,1)--(2.5,1) node [above] {$1$} -- (19,1);
\draw [thick,red](3.5,0) --(4.5,2);  
\draw [thick](4.5,1)--(5.5,1) node [above] {$3$} -- (6.5,1);
\draw [thick,red](6.5,0) -- (7.5,2);  
\draw [thick](7.5,1) --(8.5,1) node [above] {$4$} -- (9.5,1); 
\draw[thick,red] (9.5,0) -- (10.5,2);  
\draw[thick] (10.5,1) --(11.5,1) node [above] {$5$} -- (12.5,1); 
\draw [thick,blue](13.5,0) -- (12.5,2);   
\draw [thick](13.5,1) --(14.5,1) node [above] {$3$} -- (15.5,1);
\draw[thick,blue] (16.5,0) -- (15.5,2);  
\draw [thick](16.5,1) --(17.5,1) node [above] {$1$} -- (18.5,1);  
\draw [thick,blue](19.5,0) -- (18.5,2);  
\node at (.5,-1) {$\Zb_4$};
\node at (3.5,-1) {$\Zb_3$};
\node at (6.5,-1) {$\Zb_2$};
\node at (9.5,-1) {$\Zb_1$};
\node at (13.5,-1) {$\Ab_1$};
\node at (16.5,-1) {$\Ab_2$};
\node at (19.5,-1) {$\Ab_3$};
\end{scope}
\end{tikzpicture}
\end{equation}
is associated to the dimension vector $(\ldots,0,1,3,4,5,3,1,0,\ldots)$ (with the necessary $d_0=5$). The red forward-leaning lines are called NS5 branes, denoted by $\Zb$. The blue backward-leaning lines are called D5 branes, denoted by $\Ab$. The black lines between 5-branes are called D3 branes, denoted by $\Xb$ (or by $\ldots, \Xb_{-2}, \Xb_{-1}, \Xb_{0}, \Xb_{1}, \Xb_{2}, \ldots$), and the associated $d_i$ is also called the multiplicity of that D3 brane. The charge $r_i$ is now associated to the NS5 brane $\Zb_i$, and the charge $c_i$ is associated to the D5 brane $\Ab_i$. For example $ch(\Zb_2)=1$, $ch(\Ab_2)=2$.
For brevity, we can  use the more compact notation
$
\DD=\ttt{\fs $1$\fs $3$\fs $4$\fs $5$\bs $3$\bs $1$\bs }.
$

The combinatorial objects `dimension vector $d$', `pair of charge vectors $r,c$', and `brane diagram $\DD$' all encode the same information and will be used interchangeably. In what follows, we will associate a variety $\Ch$ with these combinatorial structures, which can therefore be referred to by any of these names. For example,
\begin{multline*}
\Ch(1,3,4,d_0=5,3,1)
=
\Ch(d^-=(5,4,3,1), d^+=(5,3,1))
\\
=
\Ch(r=(1,1,2,1),c=(2,2,1))
=
\Ch(\DD).
\end{multline*}
Note the convention of suppressing unnecessary 0's, as well as numbering $\Zb_i$'s and listing the components of $d^-$ and $r$ {\em from right to left}.

\smallskip

\noindent{\em Assumption 1.} In the whole paper we will assume that the charges of fivebranes, except those located between two 0-multiplicity D3 branes, are positive. This assumption is not restrictive, as it is demonstrated in \cite[\S2.8]{BR23} that any bow variety initially violating this condition can also be presented as a bow variety satisfying it.

\subsection{Bow varieties}
\label{sec:def of bow variety}
In this section we recall the quiver description of bow varieties introduced by Nakajima and Takayama in \cite{Nakajima_Takayama}. We refer to [loc. cit.] for more details. Let $\DD$ be a brane diagram. To a D3 brane $\Xb$ we associate a complex vector space $W_{\Xb}$ of dimension $d_{\chi}$. To a D5 brane $\Ab$ we associate a one-dimensional space $\C_{\Ab}$ with the standard $\GL(\C_{\Ab})$ action and the ``three-way part''
\begin{align}
\begin{split}
\label{eq: three way part}
\MM_{\Ab}= &\Hom(W_{{\Ab}^+},W_{{\Ab}^-}) \oplus
\h\Hom(W_{{\Ab}^+},\C_{\Ab}) \oplus \Hom(\C_{\Ab},W_{{\Ab}^-}) \\
&  \oplus\h\End(W_{{\Ab}^-}) \oplus \h\End(W_{{\Ab}^+}),
\end{split}
\end{align} 
with elements denoted $(A_{\Ab}, b_{\Ab}, a_{\Ab}, B_{\Ab}, B'_{\Ab})$, and $\NN_{\Ab}=\h\Hom(W_{{\Ab}^+},W_{{\Ab}^-})$.  Here, $\hbar$ indicates the action of an additional torus $\Cs_{\hbar}$.  To an NS5 brane $\Zb$ we associate the ``two-way part''
\begin{equation}
    \label{eq: two way part}
    \MM_{\Zb}= \h\Hom(W_{{\Zb}^+},W_{{\Zb}^-}) \oplus
\Hom(W_{{\Zb}^-},W_{{\Zb}^+}),
\end{equation}
whose elements will  be denoted by $(C_{\Zb}, D_{\Zb})$. To a D3 brane $\Xb$ we associate $\NN_{\Xb}=\h\End(W_{\Xb})$. In these formulas, the $\h$ factor means an action of an extra $\Cs$ factor called $\Cs_{\h}$.
Let 
\[
\MM=\bigoplus_{\Ab} \MM_{\Ab} \oplus \bigoplus_{\Zb} \MM_{\Zb}, \qquad \NN=\bigoplus_{\Ab} \NN_{\Ab} \oplus \bigoplus_{\Xb} \NN_{\Xb}.
\]
We define a map $\mu:\MM \to \NN$ componentwise as follows.
\begin{itemize}
\item 
The $\NN_{\Ab}$-component of $\mu$ is 
$B_{\Ab} A_{\Ab} -A_{\Ab} B'_{\Ab}+a_{\Ab} b_{\Ab}$.
\item
The $\NN_{\Xb}$-components of $\mu$ depend on the diagram:
\begin{itemize}
\item[\ttt{\bs -\bs}] If $\Xb$ is in between two D5 branes then it is $B'_{\Xb^-}-B_{\Xb^+}$. 
\item[\ttt{{\fs}-{\fs}}] If $\Xb$ is in between two NS5 branes then it is $C_{\Xb^+}D_{\Xb^+}$ $-D_{\Xb^-}C_{\Xb^-}$.
\item[\ttt{{\fs}-\bs}] If $\Xb^-$ is an NS5 brane and $\Xb^+$ is a D5 brane then it is $-D_{\Xb^-}C_{\Xb^-}$ $-B_{\Xb^+}$.
\item[\ttt{\bs -{\fs}}] If $\Xb^-$ is a  D5 brane and $\Xb^-$ is an NS5 brane then it is $C_{\Xb^+}D_{\Xb^+}+B'_{\Xb^-}$. (This case does not happen for our separated brane diagrams.)
\end{itemize}
\end{itemize}
consist of points of $\mu^{-1}(0)\subset \MM$ for which the stability conditions  
\begin{itemize}
\item[(S1)]  if $S\leq W_{{\Ab}^+}$ is a subspace with $B'_{\Ab}(S)\subset S$, $A_{\Ab}(S)=0$, $b_{\Ab}(S)=0$ then $S=0$,
\item[(S2)]  if $T \leq W_{{\Ab}^-}$ is a subspace with $B_{\Ab}(T)\subset T$, $Im(A_{\Ab})+Im(a_{\Ab})\subset T$ then $T=W_{\Ab^-}$
\end{itemize}
hold. As shown by Takayama \cite[Prop.~2.9]{takayama_2016}, the open subset $\widetilde{\mathcal M}\subset \mu^{-1}(0)$ is affine. Let $G=\prod_{\Xb} \GL(W_\Xb)$ and consider the character 
\begin{equation}\label{eq:character}
\chi: G\ \to \C^{\times}, 
\qquad\qquad
(g_\Xb)_{\Xb} \mapsto \prod_{\Xb'} \det(g_{\Xb'}),
\end{equation}
where the product runs for D3 branes $\Xb'$ such that $(\Xb')^-$ is an NS5 brane (in picture \ttt{{\fs}X'}). The Bow variety $\Ch(\DD)$ is defined as the GIT quotient
\[
\Ch(\DD):=\widetilde{\mathcal{M}}\sslash^{\chi} G.
\]

Let $G$ act on $\widetilde{\mathcal M} \times \C$ by $g.(m,x)=(gm,\chi^{-1}(g)x)$. We say that $m\in \widetilde{\mathcal M}$ is semistable (notation $m\in \widetilde{\mathcal M}^{ss}$) if the closure of the orbit $G(m, x)$ in $\widetilde{\mathcal{M}}\times \C$ does not intersect with $\widetilde{\mathcal{M}}\times\{0\}$ for any (every) $x\neq 0$. We say that $m$ is stable (notation $m\in \widetilde{\mathcal M}^{s}$) if $G(m,x)$ is closed and the stabilizer of $(m,x)$ is finite for $x\neq 0$. Although a priori different, semistability and stability coincide for the chosen character $\chi$. This fact can be proved by transposing Nakajima's original argument \cite[\S~3]{Nakajimaquiver} in the setting of bow varieties. Since the stabilizers are trivial \cite[Lemma 2.10]{Nakajima_Takayama}, the GIT quotient map $\widetilde{\mathcal{M}}^{ss}\to \Ch(\DD)$ is a $G$-torsor. Consequently, $\Ch(\DD)$  can be identified with the orbit space $\widetilde{\mathcal{M}}^{ss}/G=\widetilde{\mathcal{M}}^{s}/G$. In particular, $\Ch(\DD)$ is smooth. 

\subsection{Affinization and handsaw variety} \label{sec:Affinization}

Although the stability conditions (S1) and (S2) are open, the variety $\widetilde{\mathcal M}$ is affine  \cite[\S~2]{takayama_2016}. As a consequence, bow varieties come with a projective morphism
\[
\pi:X(\DD)\to X_0(\DD):=\text{Spec}(\C[\widetilde{\mathcal M}]^{G})
\]
to an affine variety. By neglecting the two-way parts of the diagram, we get a map $  \mu_{HS}:(\oplus_{\Ab} \MM_{\Ab})\to (\oplus_{\Ab} \NN_{\Ab}) \oplus (\oplus_{\Xb} \NN_{\Xb}) $.
The subscript refers to the shape of the associated quiver, which resembles a handsaw. Set $G_{HS}=\prod_{\Ab}\GL(W_{\Ab_-})$. Let $\widetilde{ \mathcal{M}}_{HS}$ denote the subvariety of $\mu_{HS}^{-1}(0)$ satisfying the conditions (S1) and (S2). It is also affine and its quotient 
\[
HS(\DD):=\text{Spec}(\C[\widetilde{\mathcal{M}}_{HS}]^{G_{HS}})
\]
is called a handsaw variety. As shown in \cite[Prop. 2.9 and Cor. 2.21]{takayama_2016}, the $G_{\DD}$ action on $\widetilde{\mathcal M}$ is free and all its orbits are closed, so $HS(\DD)$ is just the orbit space $\widetilde{\mathcal{M}}_{HS}/G_{HS}$. The natural projection $\widetilde{\mathcal{M}}\to \widetilde{\mathcal{M}}_{HS}$ descends to an affine morphism $\rho: X_0(\DD)\to HS(\DD)$. Altogether, we have morphisms 
\begin{equation}
    \label{maps from bow to affine bow and handsaw}
    X(\DD)\xrightarrow[]{\pi} X_0(\DD) \xrightarrow[]{\rho} HS(\DD).
\end{equation}

\subsection{Bow varieties and partial flag varieties}
\label{subsec: bow = partial flag} 
Let $m$ be the numer of NS5 branes in the diagram $\DD$. If now we neglect the D5 part of the brane diagram, we get a map $\mu_{A}:(\oplus_{\Zb} \MM_{\Zb})\to \bigoplus_{\Zb} \NN_{\Zb_-}$. The subscript now refers to the shape of the associated quiver, which is of type $A$. Set $G_{A}=\prod_{i=1}^m \GL(W_{\Zb_-})$. 
Set $\dd=(0=d_{-m},d_{-m+1}, \dots, d_{-2}, d_{-1})\in \NN^{m}$ and $\ww =(0,\dots , 0, d_{0})\in \NN^{m}$. We have an isomorphism
\begin{equation}
    \label{eq: quiver rep and ND5 bow}\oplus_{\Zb} \MM_{\Zb}=T^*\Rep_{A_{m}}(\dd, \ww),
\end{equation}
where $\Rep_{A_{m}}(\dd, \ww)$ is the space of representations of the quiver $A_{m}$ with dimension $\dd$ and framing $\ww$. Moreover, this isomorphism is $G_A$-equivariant, and identifies $\mu_A$ with the moment map for the action of $G_A$ on $T^*\Rep_{A_m}(\dd, \ww)$.  Notice that $G_A=\prod_{i=1}^m\GL(d_{-i})$ only acts on the vertices of the principal quiver and not on the framing.

We restrict the stability condition $\chi:G\to \Cs$ to $G_{A}\subset G$ and consider the symplectic quotient 
\[
\mu_A^{-1}(0)\sslash^{\chi} G_A
\]
By \eqref{eq: quiver rep and ND5 bow} and the comparison of the moment maps, it follows that we have a canonical isomorphism
\[
\mu_A^{-1}(0)\sslash^{\chi} G_A \cong X_{A_{m}}(\dd, \ww),
\]
where $\mathcal{M}(\dd, \ww)$ is a Nakajima quiver variety. More precisely, because of our specific choice of stability condition and dimension vectors, we have $X_{A_{m}}(\dd, \ww)\cong T^*\Fl(\dd, d_0)$, where $\Fl(\dd, d_0)$ is the variety parametrizing flags of quotients 
\[
0={\C}^{d_{-m}}\twoheadleftarrow \C^{d_{-1}} \twoheadleftarrow\dots \twoheadleftarrow \C^{d_{-1}} \twoheadleftarrow \C^{d_{0}}
\]
\begin{example}
    Consider the bow variety $\DD=\ttt{\fs $1$\fs $3$\fs $4$\fs $5$\fs $7$\bs $4$\bs $2$\bs }$.
Then $\oplus_{\Zb} \MM_{\Zb}$ is the space of framed representation of the quiver $A_5$

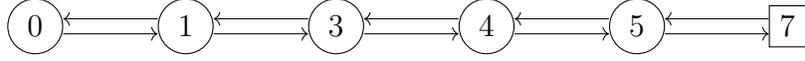
\begin{center}
\begin{tikzpicture}
    \node[circle, draw] (0) at (0,0) {$0$};
    \node[circle, draw] (1) at (2,0) {$1$};
    \node[circle, draw] (2) at (4,0) {$3$};
    \node[circle, draw] (3) at (6,0) {$4$};
    \node[circle, draw] (4) at (8,0) {$5$};
    \node[rectangle, draw] (5) at (10,0) {$7$};

    \draw[->] ([yshift=-0.1cm] 0.east) -- ([yshift=-0.1cm] 1.west);
    \draw[->] ([yshift=-0.1cm] 1.east) -- ([yshift=-0.1cm] 2.west);
    \draw[->] ([yshift=-0.1cm] 2.east) -- ([yshift=-0.1cm] 3.west);
    \draw[->] ([yshift=-0.1cm] 3.east) -- ([yshift=-0.1cm] 4.west);
    \draw[->] ([yshift=-0.1cm] 4.east) -- ([yshift=-0.1cm] 5.west);

    \draw[->] ([yshift=0.1cm] 1.west) -- ([yshift=0.1cm] 0.east);
    \draw[->] ([yshift=0.1cm] 2.west) -- ([yshift=0.1cm] 1.east);
    \draw[->] ([yshift=0.1cm] 3.west) -- ([yshift=0.1cm] 2.east);
    \draw[->] ([yshift=0.1cm] 4.west) -- ([yshift=0.1cm] 3.east);
    \draw[->] ([yshift=0.1cm] 5.west) -- ([yshift=0.1cm] 4.east);
\end{tikzpicture}.
\end{center}
Here, we follow the common practice of denoting the framing vertices with a square. Circles denote the vertices of the principal quiver $A_5$. 
\end{example}

\begin{proposition}[{\cite{BR23,ji2024bow}}]
\label{prop: charge one bows}
    Assume that all the D5 charges $c_i$ are equal to one. 
\begin{itemize}
    \item There is an isomorphism $T^*\Fl(\dd, d_0)\cong  X(\DD)$.
    \item $ HS(\DD)$ is isomorphic to $\gl(d_0)$.
    \item The map $T^*\Fl(\dd, d_0)\cong  X(\DD) \to HS(\DD)\cong \gl(d_0)$ is the moment map for the natural action of $\GL(d_0)$ on $T^*\Fl(\dd, d_0)$.
\end{itemize} 
\end{proposition}

In \S~\ref{subsec: D5 res} we will show that, for arbitrary bow varieties, we have a regular embedding $X(\DD)\hookrightarrow T^*\Fl(\dd, d_0)$.
\subsection{Torus action}
\label{sec: torus action}
Let $n$ be the number of D5 branes in $\DD$. Beyond the complex structure, the bow variety $\Ch(\DD)$ carries a symplectic structure\footnote{We do not explicitly make use of the symplectic structure in this article.} as well as the action of a $n+1$ dimensional torus. The torus $\Tt$ acting on $\Ch(\DD)$ can be decomposed as $\Tt=\At\times \C_{\h}^\times$, where $\At=\prod_{i=1}^n \GL(\C_{\Ab_i})$ is a rank $n$ torus rescaling the spaces $\C_{\Ab_i}$ in $\MM_{\Ab_i}$, cf. \eqref{eq: three way part}, and $\C_{\h}^\times$ rescales the arrows of \eqref{eq: three way part} and \eqref{eq: two way part} as prescribed by the weight $\hbar$. Since both actions preserve the zero locus $\mu^{-1}(0)$, they descend to the GIT quotient $\Ch(\DD)$.

\subsection{Torus fixed points}
\label{sec:fixed points}
Let $\DD$ be a brane diagram. A {\em tie diagram} of $\DD$ is a set of pairs $(\Zb,\Ab)$ where $\Zb$ is an NS5 brane and $\Ab$ is a D5 brane. We indicate such a pair by a curve (`tie') connecting the two fivebranes. We also require that {\em each D3 brane is covered by as many ties as its multiplicity}. Two of the 12 tie diagrams of the brane diagram \eqref{eq: example brane diagram} are below, and another one is \eqref{eq:tiediagram in Intro}.
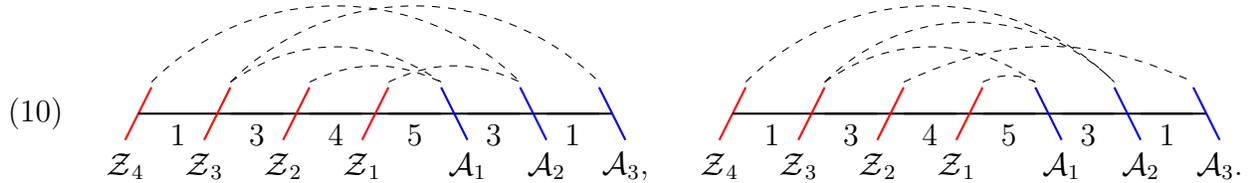
\begin{equation}
   \label{eq:tiediagram}
\begin{tikzpicture}[baseline=7,scale=.35]
\begin{scope}[yshift=0cm]
\draw [thick,red] (0.5,0) --(1.5,2); 
\draw[thick] (1,1)--(2.5,1) node [below] {$1$} -- (19,1);
\draw [thick,red](3.5,0) --(4.5,2);  
\draw [thick](4.5,1)--(5.5,1) node [below] {$3$} -- (6.5,1);
\draw [thick,red](6.5,0) -- (7.5,2);  
\draw [thick](7.5,1) --(8.5,1) node [below] {$4$} -- (9.5,1); 
\draw[thick,red] (9.5,0) -- (10.5,2);  
\draw[thick] (10.5,1) --(11.5,1) node [below] {$5$} -- (12.5,1); 
\draw [thick,blue](13.5,0) -- (12.5,2);   
\draw [thick](13.5,1) --(14.5,1) node [below] {$3$} -- (15.5,1);
\draw[thick,blue] (16.5,0) -- (15.5,2);  
\draw [thick](16.5,1) --(17.5,1) node [below] {$1$} -- (18.5,1);  
\draw [thick,blue](19.5,0) -- (18.5,2);  
\node at (.5,-1) {$\Zb_4$};
\node at (3.5,-1) {$\Zb_3$};
\node at (6.5,-1) {$\Zb_2$};
\node at (9.5,-1) {$\Zb_1$};
\node at (13.5,-1) {$\Ab_1$};
\node at (16.5,-1) {$\Ab_2$};
\node at (19.5,-1) {$\Ab_3$,};
\draw [dashed, black](7.5,2.2) to [out=25,in=155] (12.5,2.2);
\draw [dashed, black](10.5,2.2) to [out=25,in=155] (15.5,2.2);
\draw [dashed, black](4.5,2.2) to [out=35,in=145] (12.5,2.2);
\draw [dashed, black](4.5,2.2) to [out=45,in=135] (18.5,2.2);
\draw [dashed, black](1.5,2.2) to [out=45,in=135] (15.5,2.2);
\end{scope}
\end{tikzpicture}
\quad
\begin{tikzpicture}[baseline=7,scale=.35]
\begin{scope}[yshift=0cm]
\draw [thick,red] (0.5,0) --(1.5,2); 
\draw[thick] (1,1)--(2.5,1) node [below] {$1$} -- (19,1);
\draw [thick,red](3.5,0) --(4.5,2);  
\draw [thick](4.5,1)--(5.5,1) node [below] {$3$} -- (6.5,1);
\draw [thick,red](6.5,0) -- (7.5,2);  
\draw [thick](7.5,1) --(8.5,1) node [below] {$4$} -- (9.5,1); 
\draw[thick,red] (9.5,0) -- (10.5,2);  
\draw[thick] (10.5,1) --(11.5,1) node [below] {$5$} -- (12.5,1); 
\draw [thick,blue](13.5,0) -- (12.5,2);   
\draw [thick](13.5,1) --(14.5,1) node [below] {$3$} -- (15.5,1);
\draw[thick,blue] (16.5,0) -- (15.5,2);  
\draw [thick](16.5,1) --(17.5,1) node [below] {$1$} -- (18.5,1);  
\draw [thick,blue](19.5,0) -- (18.5,2);  
\node at (.5,-1) {$\Zb_4$};
\node at (3.5,-1) {$\Zb_3$};
\node at (6.5,-1) {$\Zb_2$};
\node at (9.5,-1) {$\Zb_1$};
\node at (13.5,-1) {$\Ab_1$};
\node at (16.5,-1) {$\Ab_2$};
\node at (19.5,-1) {$\Ab_3$.};
\draw [dashed, black](7.5,2.2) to [out=25,in=155] (18.5,2.2);
\draw [dashed, black](10.5,2.2) to [out=25,in=155] (12.5,2.2);
\draw [dashed, black](4.5,2.2) to [out=35,in=145] (12.5,2.2);
\draw [dashed, black](4.5,2.2) to [out=45,in=135] (15.5,2.2);
\draw [dashed, black](1.5,2.2) to [out=45,in=135] (15.5,2.2);
\end{scope}
\end{tikzpicture}
\end{equation}
The number of ties adjacent to a fivebrane is necessarily its charge. 

We can encode tie diagrams with 01-matrices of size $m\times n$, called binary contigency tables (BCTs), that indicate which fivebranes are tied. The BCTs of the two tie diagrams above are
\[
\begin{tikzpicture}
        \node at (0,0) {$\Zb_4$};
    \node at (.75,0) {$0$};
    \node at (1.5,0) {$1$};
    \node at (2.25,0) {$0$};
        \node at (0,0.75) {$\Zb_3$};
    \node at (.75,0.75) {$1$};
    \node at (1.5,0.75) {$0$};
    \node at (2.25,0.75) {$1$};
        \node at (0,1.5) {$\Zb_2$};
    \node at (.75,1.5) {$1$};
    \node at (1.5,1.5) {$0$};
    \node at (2.25,1.5) {$0$};
        \node at (0,2.25) {$\Zb_1$};
    \node at (.75,2.25) {$0$};
    \node at (1.5,2.25) {$1$};
    \node at (2.25,2.25) {$0$};
    \node at (.75,3) {$\Ab_1$};
    \node at (1.5,3) {$\Ab_2$};
    \node at (2.25,3) {$\Ab_3$};
    \draw[thick] (0.5,-.3) -- (2.5,-.3) -- (2.5,2.5) -- (0.5,2.5) -- (0.5,-.3);
\end{tikzpicture},
\qquad\qquad
\begin{tikzpicture}
        \node at (0,0) {$\Zb_4$};
    \node at (.75,0) {$0$};
    \node at (1.5,0) {$1$};
    \node at (2.25,0) {$0$};
        \node at (0,0.75) {$\Zb_3$};
    \node at (.75,0.75) {$1$};
    \node at (1.5,0.75) {$1$};
    \node at (2.25,0.75) {$0$};
        \node at (0,1.5) {$\Zb_2$};
    \node at (.75,1.5) {$0$};
    \node at (1.5,1.5) {$0$};
    \node at (2.25,1.5) {$1$};
        \node at (0,2.25) {$\Zb_1$};
    \node at (.75,2.25) {$1$};
    \node at (1.5,2.25) {$0$};
    \node at (2.25,2.25) {$0$}; 
    \node at (.75,3) {$\Ab_1$};
    \node at (1.5,3) {$\Ab_2$};
    \node at (2.25,3) {$\Ab_3$};
    \draw[thick] (0.5,-.3) -- (2.5,-.3) -- (2.5,2.5) -- (0.5,2.5) -- (0.5,-.3);
\end{tikzpicture}.
\]
The row and column sums of the BCTs are exactly the charges of the fivebranes. Let $\BCT(r,c)$ denote the collection of BCTs with row and column sums $r$ and $c$.

Our notation for a tie connecting $\Zb_k$ with $\Ab_l$ will be $\One_{kl}$, and we write $\cup$ for their unions. Hence the tie diagram above on the left will be denoted $\One_{12}\cup \One_{21}\cup\One_{31}\cup\One_{33}\cup\One_{42}$. The geometric significance of tie diagrams is the following theorem.

\begin{theorem}
\cite{rimanyi2020bow, BR23}
\label{thm:torus fixed points}
The natural inclusion $X(\DD)^{\Tt}\hookrightarrow X(\DD)^{\At}$ is the identity map. The set $X(\DD)^{\Tt}$ (equivalently, $X(\DD)^{\At}$) is finite. It is in bijection with the set of tie diagrams of $\DD$. It is also in bijection with $\BCT(r,c)$ where $r$ and $c$ are the charge vectors of $\DD$.
\end{theorem}

\noindent{\em Assumption 2.} In the whole paper we assume that $\Ch(\DD)$ has at least one fixed point. Equivalently, that for the charge vectors $\r,\c$ the set $\BCT(\r,\c)$ is not empty. 

\subsection{Tautological bundles, tangent bundle, polarization}
\label{sec:tautological bundles}
The $W_{\Xb_i}$ spaces descend to rank $d_i=d_{\Xb_i}$ ($\Tt$ equivariant) bundles $\xi_i=\xi_{\Xb_i}$ over $\Ch(\DD)$, we call them the tautological bundles. It is a fact that the bundles corresponding to the right dimension vector $(d_0,d_1,\ldots)$ are topologically trivial (with a $\Tt$-action) \cite[\S~3.2]{BR23}. 

Define 
\begin{align} 
T_{\NS5}X(d^-)  = &
\bigoplus_{k\leq 0} 
\h\Hom(\xi_{k}, \xi_{k-1}) \oplus \Hom(\xi_{k-1}, \xi_{k})
\ominus
\left(
\bigoplus_{k<0}
(1+\h)\End(\xi_k)
\right),
\nonumber \\ 
T_{\D5}X(d^+)  = &
\bigoplus_{k> 0} 
\Hom(\C_{a_k}, \xi_{k-1}) \oplus
\h\Hom(\xi_k,\C_{a_k}) \oplus (1-\h)\Hom(\xi_{k},\xi_{k-1})
\nonumber
\\
&  \oplus 
\h\End(\xi_{k-1})\oplus \h\End(\xi_k)
\ominus
\left(
\bigoplus_{k\geq 0}
(1+\h)\End(\xi_k)
\right).
\label{eq: D5 part tagent}
\end{align}
The two expressions depend only on the left $d^-$, respectively, the right $d^+$, part of the dimension vector $d$ as we indicated in the notation. It follows from the construction that for the tangent space of the bow variety $X(d)$ we obtain 
\begin{equation}\label{eq:tangent bundle}
T\Ch(d)=T_{\NS5}\Ch(d^-)\oplus T_{\D5}\Ch(d^+)  
\in K_{\Tt}(\Ch(d)).
\end{equation}

For future purposes we define the `polarization' class  $\alpha=\alpha_d \in K_{\Tt}(\Ch(d))$ by
\begin{multline}
\label{class alpha for general bow variety}
\alpha=
\h\left( TX_{\h=0} \right)^\vee
= 
\\
\h \left( 
\bigoplus_{k\leq 0} \Hom( \xi_{k-1}, \xi_{k} ) 
\oplus
\bigoplus_{k>0} 
\Hom( \C_{a_k}, \xi_{k-1}) 
\oplus
\Hom( \xi_{k}, \xi_{k-1} )
\ominus 
\bigoplus_{k} \End( \xi_k)
\right)^\vee.
\end{multline}
Calculation shows that, up to (inessential) terms only depending on $\h$, $\alpha$ is a `symplectic half' of the tangent bundle:

\begin{proposition}\cite[\S4.4.2]{Shou}
    There exists a class $\beta\in K_{\Cs_{\h}}(\pt)$ such that $\alpha+\beta$ satisfies 
\begin{equation}
   \label{equation separated alpha is a polarization}
    TX=(\alpha+\beta) +\h (\alpha+\beta)^\vee.
\end{equation}
\end{proposition}

The following lemma will be useful later. 
\begin{lemma}
\label{lma: D5 part tangent is trivial for flags}
    Assume that $c=\ul 1$. Then $T_{\D5}X(d^+)=0$.
\end{lemma}
\begin{proof}
Since the bundles $\xi_i$ are trivial for $i\geq 0$, $T_{\D5}X(d^+)$ is a weighted sum of characters of the torus $\Tt$, which can be explicitly computed to prove the lemma. Alternatively, one can check that $T_{\NS5}X(d^-)$ matches the tangent space of the cotangent bundle of the flag variety $T^*\Fl(\dd,d_0)$, cf. \S~\ref{subsec: bow = partial flag}. 
But then Proposition \ref{prop: charge one bows} and \eqref{eq:tangent bundle} force $T_{\D5}X(d^+)=0$.
\end{proof}

\subsection{Examples}
The prototype bow variety is $\TsP^1=\Ch( \hbox{\ttt{\fs $1$\fs $2$\bs $1$\bs}} )$. Here $\xi_{-1}$ is the tautological line bundle over $\Pe^1=\Pe(\C^2)$, $\xi_0=\C^2=\C_{a_1}\oplus \C_{a_2}$ and $\xi_1=\C_{a_2}$ (cf. $\xi_{\geq 0}$ bundles are topologically trivial). The 2-torus $\At$ acts naturally on the defining $\C^2$, and $\C^{\times}_{\h}$ scales the fibers. 
Equation~\eqref{eq:tangent bundle} yields
\[
T\Ch= \left( \frac{a_1}{\xi_{-1}} + \frac{a_2}{\xi_{-1}} \right) + \h \left( \frac{\xi_{-1}}{a_1} + \frac{\xi_{-1}}{a_2} \right) - 1-\h,
\]
which is---using the well known formula $\TP^1=\Hom(\xi_{-1},\C_{a_1}\oplus \C_{a_2}-\xi_{-1})$---equivalent to the obvious $T(\TsP^1)=\TP^1 \oplus \h \TsP^1$. For the polarization we get $\alpha=a_1/\xi_{-1}+a_2/\xi_{-1}$.

More generally, the diagrams
\begin{multline*}
\ttt{\fs $1$\fs $n$\bs $n-1$\bs \ldots \bs $2$\bs $1$\bs}, \qquad
\ttt{\fs $k$\fs $n$\bs $n-1$\bs \ldots \bs $2$\bs $1$\bs}, 
\\
\ttt{\fs $1$\fs $2$\fs \ldots \fs $n-1$\fs $n$\bs $n-1$\bs \ldots \bs $2$\bs $1$\bs}, \qquad
\ttt{\fs $k_1$\fs $k_2$\fs \ldots \fs $k_m=n$\bs $n-1$\bs \ldots \bs $2$\bs $1$\bs}
\end{multline*}
correspond to $\TsP^{n-1}$, $\TsGr(k,n)$, the cotangent bundle of a full flag variety, and the cotangent bundle of a partial flag variety, respectively. 
The space $\tilde{A_2}=\Ch( \hbox{\ttt{\fs $1$\fs $2$\fs $3$\fs $4$\bs $2$\bs}} )$
is the resolution of the Kleinian singularity $\C^2/\Z_3$. 
The bow variety $\Ch(\DD)$ for $\DD$ in \eqref{eq: example brane diagram} is of dimension 6, and has 12 torus fixed points, cf. Example~\ref{ex:poset}.

\subsection{D5 brane resolutions}
\label{subsec: D5 res}
Fix a brane diagram $\DD$.
Let $\wt \DD$ be the brane diagram obtained by replacing a single D5 brane $\Ab_k$ of charge $c=c(\Ab_k)\geq 2$ in $\DD$ by a pair of consecutive D5 branes $\Ab'_k$ and $\Ab''_k$ of charges $c_k'=c(\Ab'_k)\geq 1$ and $c_k''=c(\Ab''_k)\geq 1$ such that $c_k=c'_k+c''_k$. In other words, the diagram $\wt \DD$ is obtained via the local surgery
\begin{center}
\begin{tikzpicture}[scale=.7]
\draw[thick] (-1,0) -- (1,0);
\draw[thick, blue] (0.2,-0.5) -- (-0.2,0.5);
\node at (-2.5,0.1) {$\DD=$};
\node at (8.1,0.1) {$=\wt \DD$};
\node[label={\small $c'_k+c''_k$}] at (-0.2,0.5) {};

\draw[stealth-stealth, thick] (2,0) -- (3,0);

\draw[thick] (4,0) -- (7,0);
\draw[thick, blue] (5.2,-0.5) -- (4.8,0.5);
\draw[thick, blue] (6.2,-0.5) -- (5.8,0.5);
\node[label={\small $c'_k$}] at (4.8,0.5) {};
\node[label={\small $c''_k$}] at (5.8,0.5) {};

\end{tikzpicture}
\end{center}

We call $\wt \DD$ a D5 resolution of the brane diagram $\DD$, and the branes $\Ab'_k$ and $\Ab''_k$ resolving branes. Let $\wt X$ and $X$ be the bow varieties associated to $\wt \DD$ and $\DD$. We say that $\wt X$ is a D5 resolution of the brane $\Ab_k$ in the bow variety $X$. Let $\Tt=\At\times \Cs_{\hbar}$ (resp. $\wt \Tt= \wt \At \times \Cs_{\hbar}$) be the torus acting on $X$ (resp. $\wt X$). We define a homomorphism $\varphi: \Tt\to \wt \Tt$ to be the identity on most components, except 
\begin{equation}
\label{group homomorphism A-resolution}
    \Cs_{\Ab_k} \times \Cs_{\hbar}  \to  \Cs_{\Ab'_k} \times  \Cs_{\Ab''_k} \times \Cs_{\hbar} \qquad (a,\hbar)\mapsto 
        (a\hbar^{-c''_k}, a, \hbar ) .
\end{equation}
\begin{theorem}[\cite{BR23}]
\label{theorem: embedding resolution D5 branes}
    There exists a regular closed embedding $j: X\hookrightarrow \wt X$. The map $j$ is equivariant along $\varphi:\Tt\to \wt \Tt$.  Moreover, for any fixed point $f\in \X(r,c)^{\At}$ the $\At$-fixed component $F\subseteq X(r,\tilde c )^{\At}$ such that $j(f)\in F$ is isomorphic to the bow variety with brane diagram
\[
\ttt{\fs $1$\fs $2$\fs $3$\fs $\dots$\fs $c'_k+c''_k$\bs $c''_k$\bs}.
\]
\end{theorem}

\begin{definition}
    Let $f\in X^{\At}$ and let 
    $F\subset \wt{X}^{\At}$ 
    be the fixed component containing $j(f)$. The elements 
    $\tilde{f}\in F^{\wt{\At}/\At}$ 
    are called resolutions of $X$.
\end{definition}

Theorem \ref{theorem: embedding resolution D5 branes} implies that any fixed point $f$ admits exactly $\binom{c'_k+c_k''}{c'_k}$ resolutions. In the language of BCT tables, the resolutions of a fixed point $f$ represented by a BCT table $M$ are obtained by replacing the $k$-th column $M_{\bullet k}$ of $M$ with two adjacent columns that sum to $M_{\bullet k}$. It is then clear that there are exactly $\binom{c'_k+c_k''}{c'_k}$ resolutions.

\begin{lemma}
\label{lemma: D5 res bundles pullback}
    Let $\tilde{\xi_i}$ (resp. $\xi_i$) be the tautological bundles on $\wt X$ (resp. $X$). Then $j^*\tilde{ \xi_i}=\xi_i$.
\end{lemma}
\begin{proof}
    The proof follows the explicit description of the map $j$ as in \cite{BR23}. In a little more detail, an arbitrary bow variety is of the form $\wt X= T^*\Fl(N)\times_{\gl(N)} {HS}$ where $\Fl(N)$ is a partial flag variety in the ambient space $\mathbb{C}^N$. Moreover, the tautological bundles are pulled back from the projection $T^*\Fl(N) \times_{\gl(N)} HS\to T^*\Fl(N)$. Therefore, the statement follows from the fact that the embedding $j: X\to \wt X$ is induced by a morphism 
    \[
    T^*\Fl(N)\times HS\to T^*\Fl(N)\times  \wt{HS}
    \]
    which is the identity on the first factor.
\end{proof}

The next lemma relates the fixed point restrictions of $\tilde \xi_i$ to those of $\xi_i$.

\begin{lemma}
\label{restrction Chern roots in A resolution}
    Let $\mathcal \xi$ be a tautological bundle of $\widetilde X$. Then $j^* \xi\Big|_f= \xi\Big|_{\tilde f_\sharp}$ as representations of $\Tt$. As a consequence, the diagram
    \[
    \begin{tikzcd}
        K_{ \widetilde \Tt}(\widetilde X)^{taut}\arrow[d, "j^*\circ \varphi^*"] \arrow[r] &K_{ \widetilde \Tt}(\tilde f_\sharp)\arrow[d, "\varphi^*"]\\
        K_{\Tt}(X)^{taut} \arrow[r] &K_{\Tt}(f)\\
    \end{tikzcd}
    \]
    associated with fixed point localization at $\tilde f_\sharp$ and $f$ and the change of group map $\varphi:  \Tt\to \widetilde \Tt$ is commutative. 
\end{lemma}
\begin{proof}
    The statement follows from the combinatorics of the fixed point restrictions developed in \cite[\S~4.4]{rimanyi2020bow}, cf. \S~\ref{sec:fixed point restrictions}. Alternatively, it follows from the fact that $j(f)$ lies in the attracting set of $\tilde f_{\sharp}$ for the action of $\Cs_{\hbar}$ on $\wt X$ induced by the composition $\Cs_{\hbar}\hookrightarrow\Tt\xrightarrow{\varphi}\wt\Tt$.
\end{proof}

\subsection{Maximal resolutions}
\label{sec: max resolutions}
We say that the resolution of a D5 brane $\Ab$ is maximally resolved if it is replaced with $c(\Ab)$ consecutive D5 branes with D5 charges all equal to one, see
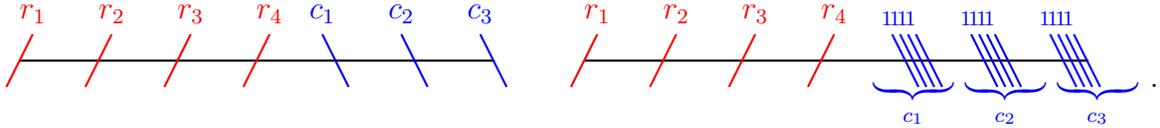
\[
\begin{tikzpicture}[baseline=0pt,scale=.35]
\draw[thick] (4,1)--(22,1) ; 
\draw[thick,red] (3.5,0)  -- (4.5,2) node[above]{$r_1$};
\draw[thick,red] (6.5,0)  -- (7.5,2) node[above]{$r_2$};
\draw[thick,red] (9.5,0) -- (10.5,2)  node[above]{$r_3$};
\draw[thick,red] (12.5,0) -- (13.5,2) node[above]{$r_4$};
\draw[thick,blue] (16.5,0) -- (15.5,2) node[above]{$c_1$};
\draw[thick,blue] (19.5,0) -- (18.5,2) node[above]{$c_2$};
\draw[thick,blue] (22.5,0) -- (21.5,2) node[above]{$c_3$};
\end{tikzpicture}
\qquad  
\begin{tikzpicture}[baseline=0pt,scale=.35]
\draw[thick] (4,1)--(23.15,1) ; 
\draw[thick,red] (3.5,0)  -- (4.5,2) node[above]{$r_1$};
\draw[thick,red] (6.5,0)  -- (7.5,2) node[above]{$r_2$};
\draw[thick,red] (9.5,0) -- (10.5,2)  node[above]{$r_3$};
\draw[thick,red] (12.5,0) -- (13.5,2) node[above]{$r_4$};
\draw[thick,blue] (16.7,0) --(15.7,2);
\draw[thick,blue] (17.0,0) -- (16.0,2);
\draw[thick,blue] (17.3,0) --(16.3,2);
\draw[thick,blue] (17.6,0) --(16.6,2); 
\draw[thick,blue] (19.7,0) --(18.7,2);
\draw[thick,blue] (20.0,0) -- (19.0,2);
\draw[thick,blue] (20.3,0) --(19.3,2);
\draw[thick,blue] (20.6,0) --(19.6,2);
\draw[thick,blue] (22.7,0) --(21.7,2);
\draw[thick,blue] (23.0,0) -- (22.0,2);
\draw[thick,blue] (23.3,0) --(22.3,2);
\draw[thick,blue] (23.6,0) --(22.6,2);
\node[blue] at (16.5,-0.65) {$\underbrace{}_{c_1}$};
\node[blue] at (20.0,-0.65) {$\underbrace{}_{c_2}$};
\node[blue] at (23.5,-0.65) {$\underbrace{}_{c_3}$};
\node[blue] at (15.5,2.6) {\tiny $1$}; \node[blue] at (15.8,2.6) {\tiny $1$}; \node[blue] at (16.1,2.6) {\tiny $1$}; \node[blue] at (16.4,2.6) {\tiny $1$};
\node[blue] at (18.5,2.6) {\tiny $1$}; \node[blue] at (18.8,2.6) {\tiny $1$}; \node[blue] at (19.1,2.6) {\tiny $1$}; \node[blue] at (19.4,2.6) {\tiny $1$};
\node[blue] at (21.5,2.6) {\tiny $1$}; \node[blue] at (21.8,2.6) {\tiny $1$}; \node[blue] at (22.1,2.6) {\tiny $1$}; \node[blue] at (22.4,2.6) {\tiny $1$};
\end{tikzpicture}.
\]
\begin{corollary}
\label{corollary any bow can be embedded in a flag variety} 
Any bow variety $X(\DD)$ can be embedded in the cotangent bundle of a partial flag variety. More precisely, $X(\DD)$ can be embedded in the bow variety $X(\wt \DD)$ whose brane diagram $\wt \DD$ is obtained from $\DD$ by maximally resolving all the D5 branes in $\DD$. 
\end{corollary}

\section{Cohomology, K-theory, Elliptic cohomology classes}

\subsection{Cohomology} Denote the Chern roots of the tautological bundle $\xi_k$ by $t_{k,i}$ for $i=1,\ldots,d_k$. Let $a_j$ be the first Chern class of the $j$-th $\C^\times$ factor of $\At$, and let $\h$ denote the first Chern class of $\C^\times_{\h}$. Polynomial expressions of $t_{ki}$ (symmetric in $i$ for the same $k$), $a_j$, $\h$ are then $\Tt$-equivariant cohomology classes on $\Ch(\DD)$.

\subsection{K theory} Similarly, Laurent polynomial expressions of the $t_{ki}$, $a_j$, $\h$ (symmetric in $t_{ki}$ for the same $k$) are $\Tt$-equivariant K theory classes on $\Ch(\DD)$. Here, each variables are interpreted as K-theoretic Chern roots, Chern classes.

\subsection{Elliptic cohomology classes} Elliptic cohomology is a scheme-valued co-variant theory, see~\cite{aganagic2016elliptic, botta2021shuffle, BR23}. We will be only interested in its $0$-th (covariant) functor
\[
\Ell_G(-):G\text{-spaces}\to \text{Schemes}_{\Ell_G},
\]
where
$E_G:=\Ell_G(\pt)$ is the elliptic cohomology of the one-point space. If $G$ is either a rank $n$ torus $T$ or $\GL(n)$, then $\Ell_T=\text{Cochar}(T)\otimes_\Z E \cong  E^{n}$ 
and $\Ell_{GL(n)}= E^{(n)}$, the symmetrized $r$-{th} Cartesian power of $E$.

In elliptic cohomology, Laurent polynomials are lifted to products of (odd) Jacobi theta functions 
\begin{equation}
    \label{eq: theta function}
    \th(x)=(x^{1/2}-x^{-1/2})\prod_{n>0} (1-q^nx)(1-q^nx^{-1}).
\end{equation}
Formally, theta functions emerge as follows. The Chern class of a rank $r$ vector bundle $V\to X$ is
 a morphism
 \[
c(V):E_G(X)\to E_{\GL(r)}(\pt)=E^{(r)},
\]
called elliptic characteristic class of $V$, see \cite[\S~1.8]{ginzburg1995elliptic} and \cite[\S~5]{Ganter_2014}. Its coordinates in the target are called elliptic Chern roots. Associated with $V$ is an invertible sheaf, the Thom sheaf of $V$
\[
\th(V):=c(V)^*(D),
\]
obtained by pulling back the effective divisor $D=\lbrace0\rbrace+E^{(r-1)}\subset E^{(r)}$. We denote its canonical global section by $\th(V)$. The connection of the section $\th(V)$ with Jacobi theta functions emerge as follows. Consider the elliptic Chern character map 
\[
\chern: \Spec(K_G(X))\to \Ell_G(X).
\]
We have a commutative diagram
\[
\begin{tikzcd}
    \Spec(K_G(X))\arrow[r] \arrow[d, "\chern"]& \Spec(K_{\GL(r)}(\pt))=(\Cs)^{(r)}\arrow[d, "\chern"]
    \\
    \Ell_G(X)\arrow[r, "c(V)"] & \Ell_{\GL(r)}(\pt)=E^{(r)}
\end{tikzcd}
\]
where the vertical right map is induced by the (complex-analytic) covering map $\chern: \Cs\to E=\Cs/q^{\Z}$. We choose a trivialization of $\chern^*\Sh{O}(0)$ that identifies its canonical section with \eqref{eq: theta function}. This trivialization combined with the diagram above induce trivializations for any line bundle of the form $\chern^*\th(V)$. As a result, the pullback of $\th(V)$ by $\chern$ is the function
\[
\prod_{i=1}^r\th(t_i)
\]
in the Chern roots $t_i$ of $V$. In the rest of the article, we will generally not distinguish between sections of $\th(V)$ and their pullbacks along the Chern character.

The Thom sheaf allows defining another important class of line bundles, which depends on an extra parameter $z$ living in a one-dimensional torus $\Cs_z$ acting trivially on $X$. Namely, for every $V$ as above,  we define 
\begin{equation*}
    \Sh{U}(V,z):=\th\left((V-1^{\oplus \rk(V)})(z-1)\right)\in \Pic{}{E_{G\times \Cs_z}(X)}.
\end{equation*}
It admits a canonical meromorphic section, which is the pullback of the canonical section
\begin{equation*}\label{eq:delta_function_def}
\prod_{i=1}^{r}\delta(t_i,z):=\prod_{i=1}^r \frac{\th(t_iz)}{\th(t_i)\th(z)}
\end{equation*}
of the line bundle $\bigotimes_{i=1}^r\th((t_i-1)(z-1))$ on $E_{\GL(r)\times \Cs_{\hbar}}(\pt)$.
Note that the latter plays a central role in $\h$-deformed characteristic classes in elliptic cohomology \cite{rimanyi2019hbardeformed}. A comprehensive list of properties of the line bundle $\Sh{U}$ can be found in \cite[Appendix 1]{okounkov2020inductiveI} and \cite[\S~4.4]{BR23}.

In this language, elliptic cohomology {\em classes} of $X(\DD)$ are sections of certain line bundles over the elliptic cohomology scheme of $\Ch(\DD)$. The coordinates of that scheme are naturally denoted by $t_{ki}, a_j, \h$, and a set of new, auxiliary variables $z_j$ ($j=1,\ldots,m$) associated with the NS5 branes. Hence various (elliptic) functions in these variables (symmetric in $t_{ki}$ for the same $k$) express elliptic cohomology classes on $\Ch(\DD)$.

\subsection{Notation} 

Singular cohomology, K-theory, and elliptic cohomology all come with natural notions of pullback and pushforward for suitable morphisms $f: X\to Y$, see \cite[\S4]{BR23}. We denote them by $f^{\oast}$ and $f_{\oast}$, respectively. Albeit slightly non-standard, our notation is motivated as follows. Elliptic cohomology associates to any equivariant morphisms $f:X\to Y$ a functorial morphism of schemes $\Ell_G(X)\to \Ell_G(Y)$ and hence adjoint functors 
\[
f^*: \text{Qcoh}(E_G(X))\to \text{Qcoh}(E_G(Y)),\qquad f_*: \text{Qcoh}(E_G(Y))\to \text{Qcoh}(E_G(X)).
\]
These are the usual pullback and pushforward functors for coherent sheaves, and should not be confused with $f^{\oast}$ and $f_{\oast}$, which are not functors, but morphisms in the categories $\text{Qcoh}(E_G(X))$ and $\text{Qcoh}(E_G(Y))$, respectively. 

To treat Euler classes of our three cohomology theories on the same ground, we will use the following notation. Consider a $G$-equivariant vector bundle on a variety $X$ with Chern roots $w_1,\dots, w_n$. For $*\in \lbrace H, K, E\rbrace$, we denote by $e^*(V)$ the corresponding Euler classes, given by 
\[
e^*(V)=\prod_{i=1}^n w_i\qquad e^*(V)=\prod_{i=1}^n\hata(w) \qquad e^E(V)=\prod_{i=1}^n\th(w)
\]
where $\hata(w)=(w^{1/2}-w^{-1/2})$ is the $\hat A$-genus. Our choice of using the $\hat A$-genus instead of the more standard $\Lambda^\bullet V^\vee=\prod_{i=1}^n (1-w^{-1})$ follows from the fact that $\hata(x)=\lim_{q\to 0}\th(x)$, so taking limits from elliptic cohomology to K-theory is more straightforward.

\subsection{Fixed point restriction maps}
\label{sec:fixed point restrictions}
Let $f\in \Ch(\DD)^{\Tt}$ be a torus fixed point. The inclusion $\{f\}\to \Ch(\DD)$ induces a {\em restriction map}
\[
\phi_f: H_{\Tt}^*(\Ch(\DD)) \to H_{\Tt}^*(\{f\}). 
\]
In effect this map is obtained by specializing the topological variables $t_{ki}$ to functions of $a_j$ and $\h$. The same holds in K theory and elliptic cohomology. The combinatorics of these specializations is described in \cite[\S{}4.6]{rimanyi2020bow} and is illustrated in Figure~\ref{fig:fixed point restriction}.

Namely, we decorate the ties over every D3 brane. The ones coming out of $\Ab_{j}$ will be decorated by $a_j$ times an $\h$-power. The rules are:
\begin{itemize}
    \item over the D3 brane adjacent to $\Ab_j$ the decorations are $a_j,a_j/h, \ldots,a_j/h^{c_j-1}$ (from longest tie to shortest tie), where $c_j$ is the charge of $\Ab_j$ (the number of ties connected to $\Ab_j$);
    \item on a tie if we move over a D5 brane, the decoration does not change;
    \item on a tie if we move over an NS5 brane to the left, the decoration is divided by $\h$.
\end{itemize}

\begin{theorem} \cite[\S4]{rimanyi2020bow} \label{thm:RestrictionMaps}
The map $\phi_f$ maps the variables $t_{ki}$ to the decorations of the ties above the D3 brane $\Xb_k$.    
\end{theorem}

The caveat is that in cohomology the decorations are read additively, while in K theory and elliptic cohomology the decorations are read multiplicatively. For example, for $f$ in Figure~\ref{fig:fixed point restriction} the cohomological $\phi_f$ maps the Chern roots of $\xi_{-1}$ as 
\[
t_{-11}\mapsto a_1-\h, \qquad
t_{-12}\mapsto a_2-\h, \qquad
t_{-13}\mapsto a_2-2\h, \qquad
t_{-14}\mapsto a_3-\h.
\]

\begin{remark}
    The Chern roots are not elements of the (respective) cohomologies, only their symmetric expressions. Hence, the order in which we substitute them in the $t_{ki}$ variables is irrelevant. However, later, we will agree on a specific order for substituting in the $t_{0i}$ variables, see Definition~\ref{def:610}. 
\end{remark}

For different fixed points $f$ on the same bow variety $\Ch(\DD)$ the maps $\phi_f$ are, of course, different. However, a crucial observation is that the difference is only for the $t_{ki}$ substitutions {\em with negative $k$}. For all the fixed points on the same $\Ch(\DD)$ variety the $t_{ki}$ substitutions for non-negative $k$ {\em are the same}. The topological reason for this is that the bundles $\xi_k$ with $k\geq 0$ are topologically trivial (with a $\Tt$ action), namely 
\begin{equation}\label{eq:tplus subs}
\xi_k = \bigoplus_{j>k} \bigoplus_{i=0}^{c_j-1} a_j\h^{-i} \qquad \text{ for $k\geq 0$,}
\end{equation}
see \cite[Prop.~3.4]{BR23}. This motivates the following convention:

\smallskip

\noindent{\bf Convention:} For $k > 0$, the variables $t_{ki}$ will be identified with the corresponding expressions $a_j \h^i$:  
\begin{equation}\label{eq: t positive ah}  
\{ t_{ki} : i = 1, \dots, d_i \}  
\leftrightarrow  
\{ a_j\h^{-i} : j > k, \, i = 0, \dots, c_j - 1 \}.  
\end{equation}  

At this stage, it may seem unclear why this identification is made only for $k>0$ and not for the variables $t_{0i}$. The reason lies in the geometric significance of the $\xi_0$ bundle, which plays a special role and influences our algebraic definitions accordingly. In particular, we will consider certain ($\Wt$-) functions that are not symmetric in the $t_{0i} $ variables. As a result, extending the identification in~\eqref{eq: t positive ah} to the case $k=0$ requires additional care—see Definition~\ref{def:610}.  

\begin{figure}
\begin{equation*}
  \begin{tikzpicture}[baseline=7,scale=.6]
\begin{scope}[yshift=0cm]
\draw [thick,red] (0.5,0) --(1.5,2); 
\draw[thick] (1,1)--(2.5,1) node [below] {$1$} -- (19,1);
\draw [thick,red](3.5,0) --(4.5,2);  
\draw [thick](4.5,1)--(5.5,1) node [below] {$3$} -- (6.5,1);
\draw [thick,red](6.5,0) -- (7.5,2);  
\draw [thick](7.5,1) --(8.5,1) node [below] {$4$} -- (9.5,1); 
\draw[thick,red] (9.5,0) -- (10.5,2);  
\draw[thick] (10.5,1) --(11.5,1) node [below] {$5$} -- (12.5,1); 
\draw [thick,blue](13.5,0) -- (12.5,2);   
\draw [thick](13.5,1) --(14.5,1) node [below] {$3$} -- (15.5,1);
\draw[thick,blue] (16.5,0) -- (15.5,2);  
\draw [thick](16.5,1) --(17.5,1) node [below] {$1$} -- (18.5,1);  
\draw [thick,blue](19.5,0) -- (18.5,2);  
\node at (.5,-1) {$\Zb_4$};
\node at (3.5,-1) {$\Zb_3$};
\node at (6.5,-1) {$\Zb_2$};
\node at (9.5,-1) {$\Zb_1$};
\node at (13.5,-1) {$\Ab_1$};
\node at (16.5,-1) {$\Ab_2$};
\node at (19.5,-1) {$\Ab_3$.};
\draw [dashed, black](7.5,2.2) to [out=70,in=200]  (10,6)  to [out=20,in=150] (16,4.4) to [out=-30,in=140] (18.5,2.2);
\draw [dashed, black](10.5,2.2) to [out=25,in=155] (12.5,2.2);
\draw [dashed, black](4.5,2.2) to [out=25,in=155] (12.5,2.2);
\draw [dashed, black](4.5,2.2) to [out=45,in=135] (15.5,2.2);
\draw [dashed, black](1.5,2.2) to [out=45,in=135] (15.5,2.2);
\node [blue] at (11.4,2) {$a_1\!/\!\h$};
\node [blue] at (11.4,3) {$a_1$};
\node [blue] at (11.4,5) {$a_2$};
\node [blue] at (11.4,4) {$a_2\!/\!\h$};
\node [blue] at (11.4,6) {$a_3$};
\node [blue] at (8.8,2.7) {$a_1\!/\!\h$};
\node [blue] at (9.6,5.3) {$a_2\!/\!\h$};
\node [blue] at (9.25,4) {$a_2\!/\!\h^2$};
\node [blue] at (9,6.1) {$a_3\!/\!\h$};
\node [blue] at (6.1,2.3) {$a_1\!/\!\h^2$};
\node [blue] at (6.6,5.3) {$a_2\!/\!\h^2$};
\node [blue] at (6.25,4) {$a_2\!/\!\h^3$};
\node [blue] at (3.2,4.1) {$a_2\!/\!\h^3$};
\node [blue] at (14.4,3.7) {$a_2$};
\node [blue] at (14.1,2.6) {$a_2\!/\!\h$};
\node [blue] at (14.5,5.9) {$a_3$};
\node [blue] at (17,3) {$a_3$};
\end{scope}
\end{tikzpicture}
\end{equation*}
\caption{The combinatorics of the restriction map to a torus fixed point.}
\label{fig:fixed point restriction}
\end{figure}
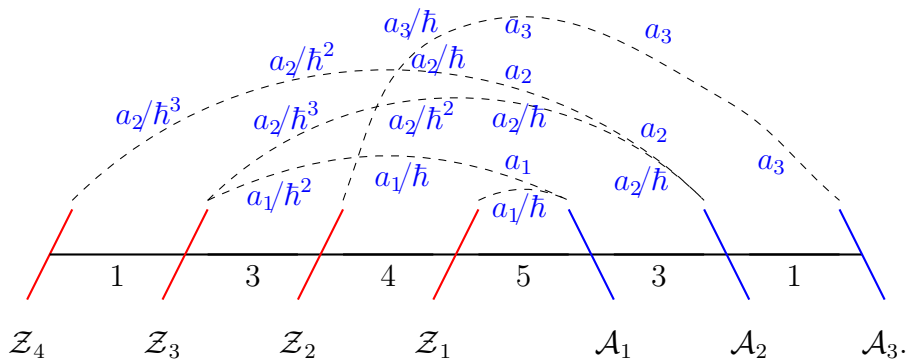

\section{Stable envelopes}

In this section we give a brief review of the three flavors of {\em stable envelopes}: cohomological, K-theoretic, and elliptic. 


\subsection{Preliminaries}
\label{subsec: preliminars}

Let $X$ be a smooth quasi-projective symplectic variety with an action of a torus $\Tt$ and let $A\subset \Tt$. 
The set of weights of the normal bundle $N_{X/X^{A}}$ of $X^{A}$ in $X$ is a finite subset and their dual hyperplanes cut the real Lie algebra 
\[
\text{Lie}(A)=\text{Cochar}(A) \otimes_\Z \R
\]
in finitely many chambers. A choice of a chamber ${\mC}$ allows us to distinguish between attractive, $A$-fixed, and repelling weights. Namely, we say that a weight $\chi\in \text{Char}(A)$ is attracting (resp. repelling) if $\lim_{t\to 0} \chi\circ \sigma (t)=0$ (resp. if $\lim_{t\to 0} \chi\circ \sigma (t)=\infty$) for any $\sigma\in \mC$. The trivial character is called the $A$-fixed character. Accordingly, the normal bundle $N_{X/X^{A}}\in K_{T}(X^{A})$ decomposes  
\begin{equation}
    \label{eq: decomposition normal bundle}
    N_{X/X^{A}}=N_{X/X^{A}}^+ \oplus N_{X/X^{A}}^-.
\end{equation}

Let us fix a chamber $\mC$. Given a subvariety $Y \subset X^{A}$, we define its attracting set $\Att{C}(Y)$ as: 
\[
\Att{C}(F)=\{x\in X\; |\; \lim_{t\to0} \sigma(t)\cdot x=f, \text{ for any $\sigma \in \mC$}\}\subset X.
\]
Notice that there is a canonical (generally non-continuous) map $\Att{C}(Y) \to Y$ obtained by taking the limit of the $\mC$-action induced by any $\sigma\in \mC$. Moreover, if $Y = F \subseteq X^{A}$ is a connected component, then $\Att{C}(F) \to F$ is an affine bundle of rank $\text{codim}_X(F)/2$.


Attracting sets $\Att{C}(F)$ are generally not closed in $X$ because their closures intersect other $\At$-fixed components. Since $X$ is quasi-projective Maulik and Okounkov \cite[\S{}3.2.3]{maulik2012quantum}  show that the choice of a chamber determines a partial ordering on the set of fixed components of $X^{\At}$ defined as the transitive closure of the relation
\[
\overline{\Att{C}(F)}\cap F'\neq 0 \implies F\geq F'.
\]

Consider the $T$-invariant subspaces
\begin{equation*}
    \label{equation spaces in the definition of stable envelopes}
    \Att{C}^{\leq}(F):= \bigcup_{F'\leq F} F\times \Att{C}(F'),
    \qquad 
    X^{\geq F}=X\setminus \bigcup_{F'< F} \Att{C}(F')
\end{equation*}
Notice that the inclusion $j: F\times \Att{C}(F)\hookrightarrow F\times X^{\geq F}$ is well defined and closed. 


\begin{example} \label{ex:poset}
    Assume that $X$ is bow variety and $A$ is the torus $\At$ defined in \S~\ref{sec: torus action}. According to Theorem \ref{thm:torus fixed points} the set $X^{\At}$ is finite. The tangent weights for the $\At$ action on $X$ are of the form $a_i/a_j$, with $0 \leq i,j \leq n$ \cite[\S~4]{rimanyi2020bow}. Hence, a permutation of the $a_i$ variables determines a chamber (but different permutations may give the same chamber). By slight abuse of notation we will write 
\[
\Chamb_{\sigma}=\lbrace a_{\sigma(1)}< \dots <a_{\sigma(n)}\rbrace
\]
for the chamber that a permutation $\sigma\in S_n$ determines. 

The partial order induced by $\Chamb_{\sigma}$ can be described combinatorially, in the language of BCT's, see \S~\ref{sec:fixed points}. In \cite{BFR23} this partial order is identified with the so-called combinatorial secondary Bruhat order on 01-matrices. For example, $\Ch(\DD)$ for the  brane diagram \eqref{eq: example brane diagram} is a variety of dimension 6, with 12 torus fixed points.  The Hasse diagram of the partial order on its fixed points (associated to $\Chamb_{id}$) is 
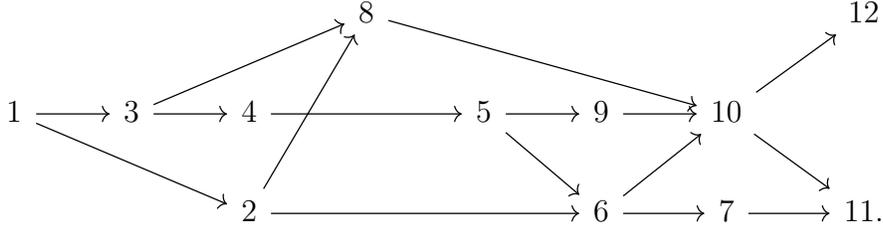
\[
\begin{tikzcd}
  &   &   & 8\ar[rrrd] &   &   &    & 12 \\
1\ar[rrd]\ar[r] & 3\ar[r]\ar[urr] & 4\ar[rr] &   & 5\ar[rd]\ar[r] & 9\ar[r] & 10 \ar[ru]\ar[rd] &  \\
  &   & 2\ar[ruu]\ar[rrr] &   &   & 6\ar[r]\ar[ru] & 7\ar[r]  & 11.
\end{tikzcd}
\]
The fixed points shown in \eqref{eq:tiediagram} are called 6 and 8 here, and \eqref{eq:tiediagram in Intro} is called 5 here.
\end{example}

\subsection{Cohomological Stable envelopes}

\begin{theorem}\cite[ Prop.3.5.1]{maulik2012quantum}
\label{thm: definition stab}
Let $\mC$ be a chamber for the action of $A$ on $X$. For every connected component $F\subseteq X^{A}$, there exists a unique $T$-invariant Lagrangian cycle 
\[
\Stab_{\Chamb}(F)\in \BoHo_{\Tt} (F\times X),
\]
proper over $X$, such that

\begin{enumerate}
    \item[(i)]  The restriction of $\Stab_{\Chamb}(F)$ on $F\times X^{\geq F}$ is equal to $\left[ \Att{C}(F)\right]$.
    \item[(ii)] The restriction of $\Stab_{\Chamb}(F)$ to $F\times X\setminus \Att{C}^\leq(F)$ is zero.
    \item[(iii)] For every $F'< F$, the restriction of $\Sh{L}_{\mC}$ to $F\times F'$ has $A$-degree less than $\codim(F')/2$.
\end{enumerate}
\end{theorem}
In (iii), the $A$-degree of a class in $H^*_{\Tt}(X^{A})$ is defined as the cohomological degree in the equivariant variables associated with the torus $A$. This can be explicitly computed via a splitting
\begin{equation}
    \label{eq: cohomological A-degree}
    H^*_{\Tt}(X^{A})\cong H^*_{\Tt/A}(X^{A})\otimes H^*_{A}(\ast)
\end{equation}
by determining the cohomological degree in the second term of the tensor product\footnote{Although the isomorphism above is not canonical, different isomorphisms produce the same filtration by cohomological degree in the $A$-equivariant parameters.}.

Following Maulik and Okounkov, the (cohomological) stable envelope is defined as the $H_{\Tt}$-linear map
\[
\Stab_{\mC}: H^*_{\Tt}(X^{A})\to H^*_{\Tt}(X), \qquad \gamma\mapsto (p_2)_* (\Sh{L}_{\mC}\cap (p_1)^*\gamma)
\]
defined by the cycle $\Sh{L}_{\mC}$ via convolution. For details about convolution in Borel--Moore homology, see \cite[\S~2.7]{Chriss_book}

The next lemma is a key result for the application of stable envelopes in the geometric representation theory of Yangians. For us, it will also crucial to relate the geometry of stable envelopes to the combinatorics of shuffle algebras.  

\begin{lemma}[{\cite[Lem.3.6.1]{maulik2012quantum}}]
\label{lma:triangle lemma}
Let $\mathfrak{C}$ be a chamber for the action of $\At$ on $X$ and let $\mathfrak{C}'\subset \mathfrak{C}$ be a face, of some lower dimension. Let $\At'\subset \At$ be the torus whose Lie algebra is the span of $\mathfrak{C}'$ in the Lie algebra of $\At$ and let $\mathfrak{C}/\mathfrak{C}'$ be the projection of $\mathfrak{C}$ on the Lie algebra of $\At/\At'$. 
Then we have 
\[
\Stab_{\mathfrak{C}'}\circ\Stab_{\mathfrak{C}/\mathfrak{C}'}  =\Stab_{\mathfrak{C}}.
\]
\end{lemma}

\subsection{Cohomological Stable envelopes for bow varieties}
\label{sec: coh stab bow}
For bow varieties $\Ch(\DD)$ we choose $A=\At$, $\Chamb=\Chamb_\sigma$, and apply the definition of the last section. The resulting stable envelopes are called $\Stab^H_{\Chamb}(f)$.

For example, consider $\Ch=\Ch(\ttt{\fs 1\fs n\bs n-1\bs \ldots \bs 2\bs 1\bs})$, that is, $\TsP^{n-1}$. The torus $\Cs_{\h}$ rescales the cotangent direction and $\At$ is identified with the maximal torus of $\GL(n)$ acting naturally on $\C^n$ and hence on $\TsP^{n-1}$.

    Let $f_k=\One_{2k}\cup \bigcup_{l\not=k}\One_{1l} \in X^{\At}$, that is, the fixed point whose tie diagram has $n$ ties: one connecting $\Zb_2$ with $\Ab_k$, as the rest connect $\Zb_1$ with $\Ab_l$ for all $l\not=k$.
    
    Let $t=t_{-11}$; this is the Chern class of the tautological bundle over $\TsP^{n-1}$. We saw in \S~\ref{sec:fixed point restrictions} that the restriction map to $f_k$ is $t \mapsto a_k-\h$. Then it is routine to check that for the chamber $\Chamb=\Chamb_{id}$
	\begin{equation}
	    \label{eq; coh stab proj}
        \Stab_{\Chamb}(f_k) =\prod_{i=1}^{k-1} (a_i- t)\cdot 
	\prod_{i=k+1}^n  (t-a_i+\h)
	\end{equation}
    satisfies the axioms of Theorem \ref{thm: definition stab}.

\subsection{K-theoretic Stable envelopes}
The definition of stable envelopes in K theory mimics the cohomological one. However, the notion of degree of a class in $H_T^*(X^{\At})$ is replaced by a polytope condition. Namely, given a class $f(A) \in K_T(X^{A})=K_{\Tt/A}(X^{A})\otimes K_{A}(\pt)$, we define
\[
\deg_A(f)=\text{Newton Polytope}(f)\in \Char(A)\otimes_{\Z}\R,
\]
i.e. the convex hull of nonzero coefficients of $f(a)$.

In addition, we require the existence of a polarization, i.e. a class $\alpha\in K_{\Tt}(X)$ such that, after restriction to $K_A(X)$, satisfies
\[
(\alpha+\alpha^\vee)|_{F}=TX|_{F}
\]
for all connected components $F\subset X^{A}$.
\begin{theorem} \cite{okounkov2020inductiveI}
Let $\mC$ be a chamber for the action of $A$ on $X$ and let $s\in \Pic(X)\otimes_{\Z}\R$. For every connected component $F\subset X^{A}$, there exists a class
\[
\uStab^s_{\Chamb}(F)\in K_{\Tt} (F\times X)[\hbar^{1/2}],
\]
proper over $X$, such that the following conditions hold. 
    \begin{itemize}
    \item[(i)] The restriction of $\uStab_{\Chamb}(F)$ on $F\times X^{\geq F}$ is equal to 
    \[
     \left(\frac{\det(N_{X/F}^-)}{\det (\alpha)|_F}\right)^{1/2}\otimes \Sh{O}_{\Att{C}(F)}.
    \]
    \item[(ii)] The restriction of $\uStab_{\Chamb}(F)$ to $F\times X\setminus \Att{C}^\leq(F)$ is zero.
    \item[(iii)] For every $F'< F$, the restriction of $\uStab_{\Chamb}$ to $F\times F'$ satisfies
    \[
    \deg_A(\uStab_{\Chamb}(F)|_{F\times F'})\subset \deg_A\left(\Lambda^\bullet  (\alpha)^\vee
|_{F'}\right) + \weight_A (s|_{F'})- \weight_A (s|_{F}).
    \]
    \end{itemize}
Further, they are unique if
\begin{equation}
    \label{eq: genericity slope}
    \weight_A (s|_{F'})- \weight_A (s|_{F})\not\in \Char(A)\subset \Char(A)\otimes_{\Z}\R
\end{equation}
for every pair $F'< F$.
\end{theorem}
We call the fractional line bundle $s\in \Pic(X)\otimes_{\Z}\R$ slope. A slope is generic for $X$ if it satisfies \eqref{eq: genericity slope}. Throughout this article we only consider generic slopes. 

To match the analogous formulas in singular and elliptic cohomology, it will be convenient to define a normalized K-theoretic stable envelope $\Stab^{K,s}_{\Chamb}(f)$ as follows: given a fixed point $F\in X^{\At}$, we require 
\begin{equation}
    \label{eq: normalized K theoretic stab}
    \det(\alpha)^{-1/2}\otimes \Stab^{K,s}_{\Chamb}(f)\otimes \det(\alpha_f)^{1/2}= \uStab^{K,s}_{\Chamb}(f),
\end{equation}
where $\alpha_f$ is the $A$-fixed part of $\alpha|_{F}$.

\subsection{K-theoretic stable envelopes for bow varieties} 
\label{sec: K stab bow}
For bow varieties $X=\Ch(\DD)$ we choose $A=\At$, $\Chamb=\Chamb_\sigma$, $s$ generic, the polarization $\alpha$ as in \eqref{class alpha for general bow variety}, and apply the definition of the last section. The resulting stable envelopes are called $\uStab^{K,s}_{\Chamb}(f)$.

For the example in \S~\ref{sec: coh stab bow}, the K-theoretic stable envelope 
for the standard chamber $\Chamb=\Chamb_{id}$ is 
\begin{align*}
    \uStab^{K,s}_{\Chamb}(f_k) =(-1)^{k-1}\frac{1}{(\sqrt{\hbar})^{k-1}} \prod_{i=1}^{k-1}\left(1-\left(\frac{t}{a_i}\right)^{-1}\right)\cdot 
	\left(\frac{t\hbar}{a_k}\right)^{\lfloor s\rfloor}
	\cdot 
	\prod_{i=k+1}^n \left( 1-\left(\frac{t\hbar }{a_i}\right)^{-1}\right)
\end{align*}
while its normalized version is 
\begin{align}
\label{eq: Kstab proj}
    \Stab^{K,s}_{\Chamb}(f_k) =\prod_{i=1}^{k-1}\ag{a_i}{t}\cdot 
	\left(\frac{t\hbar}{a_k}\right)^{\lfloor s\rfloor+1/2}
	\cdot 
	\prod_{i=k+1}^n \ag{t\h}{a_i}
\end{align}
where $\hat a(x)=(x^{1/2}-x^{-1/2})$, as before. Note that $\hat a(x)=\lim_{q\to 0} \vartheta(x)$. 

\subsection{Elliptic stable envelopes}
\label{sec: E stab}

Moving further from K-theory to elliptic cohomology, we get to the theory of elliptic stable envelopes. Although encoding the same geometric data, the elliptic stable envelope display two unique features:
\begin{itemize}

\item From a combinatorial point of view, the stable envelopes are transcendental functions in a number of variables. More precisely, they arise as products of theta functions $\th(t)$.
This feature reflects the fact that the stable envelopes are effectively sections of certain line bundles over products of elliptic curves. 

\item The stable envelopes depend on a set of additional parameters, called K\"ahler (or dynamical) parameters $z\in \Pic(X)\otimes_{\Z}\C$. The elliptic stable envelope depends continuously on these parameters, which in the K-theoretic limit recover the K-theoretic slope parameters $s\in \Pic(X)\otimes_{\Z}\R$.
More explicitly, one has 
\[
z=q^{-s}
\]
Formally, these parameters arise as equivariant parameters of an additional torus $Z=\Pic(X)\otimes_{\Z}\Cs$ acting trivially on $X$. 

\end{itemize}

By definition, the elliptic stable envelopes \cite{aganagic2016elliptic, okounkov2020inductiveI} of $X$ are distinguished sections of certain line bundles on $E_{T\times Z}(X^{A}\times X)$ depending on the choice of a chamber $\Chamb$ and of an attractive line bundle $\Sh{L}_X$. The latter is a line bundle on the scheme $E_T(X)$, which refines the choice of polarization in the K-theoretic setting. By definition, it satisfies
\begin{equation}
    \label{defining equation attractive line bundle}
    \deg_A(i_F^*\Sh{L}_X)=\deg_A(\th(N^-_{F/X})),
\end{equation}
for all connected components $F\subseteq X^{A}$. Here $i_F: F\hookrightarrow X$ is the inclusion. The notion of $A$-degree here is the elliptic analog of the cohomological and K-theoretic $A$-degree induced above. For the precise definition, we refer to \cite[\S~B.3]{okounkov2020inductiveI}. When $F$ is a fixed point, as in the case of bow varieties, then $i_F^*\Sh{L}_X$ is a line bundle on the abelian variety $E_{T\times Z}(\pt)$; its $A$-degree is then obtained by pullinkg back to $E_A(\pt)$ and considering its class in the Neron-Severi group $NS(E_A(\pt))\cong \Z^{\dim(A)}$. 

Set 
$ 
\Sh{L}_{A,F}=i_F^*\Sh{L}_X\otimes \th(-N_{X/F}^-).
$
The elliptic stable envelope of $F$ is a section of the sheaf
\begin{equation}
\label{sheaf definition of stable envelopes}
    (\Sh{L}_{A,F})^{\triangledown}\boxtimes \Sh{L}_X(\infty\Delta)
\end{equation}
on $E_{T\times Z}(F\times X)$. The notation $\infty\Delta$ indicates that the section of $(\Sh{L}_{A,F})^{\triangledown}\boxtimes \Sh{L}_X$ can have singular behavior on the resonant locus $\Delta\subset E_{\Tt/A}(\pt)$, see \cite[\S2.3]{okounkov2020inductiveI}. Equivalently, the stable envelopes are regular sections of the sheaf $\iota_{*}\iota^*\left((\Sh{L}_{A,F})^{\triangledown}\boxtimes \Sh{L}_X\right)$, where $\iota: E_{T\times Z/A}\setminus \Delta\to E_{T\times Z/A}$ is the inclusion.

\begin{theorem}\cite{okounkov2020inductiveI}
\label{thm: definition elliptic stab}
Let $\mC$ be a chamber for the action of $A$ on $X$ and let $\Sh{L}_X$ be an attractive line bundle. For every connected component $F\subset X^{A}$, there exists a unique section
\[
\Stab_{\Chamb}(F)\in \Gamma\left((\Sh{L}_{A,F})^{\triangledown}\boxtimes \Sh{L}_X(\infty \Delta)\right)
\]
proper over $X$, such that

\begin{enumerate}
    \item[(i)] The diagonal axiom: the restriction of $\Stab_{\Chamb}(F)$ on $F\times X^{\geq F}$ is equal to $\left[ \Att{C}(F)\right]$.
    \item[(ii)] The support axiom: the restriction of $\Stab_{\Chamb}(F)$ to $F\times X\setminus \Att{C}^\leq(F)$ is zero.
\end{enumerate}
\end{theorem}

Notice that, compared to its cohomological analogue, cf. Theorem \ref{thm: definition stab}, the elliptic stable envelope seems to be uniquely determined by one fewer condition. In fact, the degree condition for the cohomological stable envelope is encoded in the choice of attracting line bundle $\Sh{L}_X$.

Being proper over $X$, the elliptic stable envelope $\Stab_{\Chamb}(F)$ induces a morphism 
\begin{equation}
    \label{elliptic stable enveliopes map}
    \Stab_{\Chamb}(F):(p_A)_{\ast}(\Sh{L}_{A,F})\to p_{\ast}\Sh{L}_X(\infty \Delta)
\end{equation}
which is commonly denoted in the same way.

\begin{remark}
\label{remark: nondegenerate line bundle}
    Clearly, the content of Theorem \ref{thm: definition elliptic stab} is empty if $\infty\Delta=E_{\Tt/A}$. As a consequence, it is essential to bound the resonant locus $\infty\Delta$. The optimal bound is achieved when the attractive line bundle $\Sh{L}_X$ is non-degenerate, i.e. when the complement of its resonant locus $\Delta$ is dense. In most situations, including the case of bow varieties discussed below, a non-degenerate attractive line bundle $\Sh{L}_X$ exists only if the K\"ahler torus $Z$ is non-trivial. Hence, the K\"ahler torus naturally emerges in elliptic cohomology as a key ingredient to guarantee existence and well-behavior of the elliptic stable envelopes.
\end{remark}


\subsection{Eliptic stable envelopes for bow varieties}
\label{sec:Stabs for bows}

We now specialize the theory of elliptic stable envelopes to bow varieties.
Let $X$ be a bow variety with $m$ NS5 branes (ordered from right to left), $n$ D5 branes, and the torus $\Tt=\At\times \Cs_{\h}$ acting on it. We introduce the K\"ahler torus $Z=(\Cs)^{m}$ and let $(z_1,\dots z_m)$ be a set of coordinates spanning $Z$. We also set $T=\Tt\times Z=\At\times \Cs_{\h}\times Z$. It is beneficial to think about the subtorus $\At\times Z$ as the ambient space of two tuples of parameters $(a_1, \dots a_n)$ and $(z_1,\dots z_m)$ that are attached to the similarly named fivebranes.

Set $\Base:=E_{\Tt\times Z}(\pt)$, and recall the definition of the polarization $\alpha=\hbar (TX|_{\h=0})^{\vee}$ from \S~\ref{sec:tautological bundles}.  Define
\begin{equation}
    \label{eq: line bundles stab bow}
    \Sh{U}:=\bigotimes_{k=1}^{m-1} \Sh{U}\left(\xi_{-k},\frac{z_{k+1}}{z_{k}}\h^{m-k-\ch(\Zb_{k+1})}\right),
\qquad
    \Sh{L}_X=\th(\alpha)\otimes \Sh{U}.
\end{equation}

\begin{theorem}\cite{BR23}
    The line bundle $\Sh{L}_X$ is attractive for the action of $\At$ on $X$. Therefore, stable envelopes for bow varieties exist for any subtorus $A\subseteq \At$ and are unique.  Moreover the line bundle $\Sh{L}_X$ is non-degenerate in the sense of Remark \ref{remark: nondegenerate line bundle}. Furthermore, $\Delta$ is contained in the complement of the union of codimension-one abelian varieties of the form $\{z_i/z_j\hbar^{\alpha_{ij}}=1\}$ for certain $\alpha_{ij}\in \Z$.
\end{theorem}

The resulting elliptic stable envelopes for bow varieties will be called $\Stab^{E}_{\Chamb}(f)$. We do not indicate the choice of $\alpha$ because we fix that choice for the whole paper. 

\subsubsection{Cotangent bundle of a projective space}
\label{sec:Stab,TPn}

Continuing the example from \S~\ref{sec: coh stab bow} and~\ref{sec: K stab bow}, consider the fixed point $f_k=\One_{2k}\cup \bigcup_{l\not=k}\One_{1l} \in X^{\At}$ in $\Ch=\Ch(\ttt{\fs 1\fs n\bs n-1\bs \ldots \bs 2\bs 1\bs})=\TsP^{n-1}$, and as before, set $t=t_{-11}$. Then for the chamber $\Chamb=\Chamb_{id}$ we have
\begin{equation}\label{eq:Stab TP^n1}
	\Stab^{E}_{\Chamb}(f_k) =\prod_{i=1}^{k-1} \th\left(\frac{a_i}{t}\right)\cdot 
	\frac{\th\left(\frac{t}{a_k}\frac{z_2}{z_1}\h^{k-1}\right)}{\th\left(\frac{z_2}{z_1}\h^{k-2}\right)}
	\cdot 
	\prod_{i=k+1}^n \th\left( \frac{t}{a_i}\h\right).
	\end{equation}

Since elliptic stable envelopes are the most abstract of the three flavors, let us give a detailed verification of this formula. We need to verify that (i) the expression is a section of the correct line bundle, and (ii) that its principal restriction is the expected Euler class.
\begin{itemize}
    \item[(i)] This calculation is done in detail in \cite[Prop.~2.5]{BR23}, with the only change in that paper the roles of $z_1$ and $z_2$ are swapped. 
    \item[(ii)] The restriction map to the fixed point $f_k$ is the substitution $t\mapsto a_k\h^{-1}$, see \S~\ref{sec:fixed point restrictions}. At that substitution the formula above maps to 
    \[ 
    \prod_{i=1}^{k-1}\th\left(\frac{a_i}{a_k}\h\right) \cdot 
    \prod_{i=k+1}^{n}\th\left(\frac{a_k}{a_i}\right).
    \]
    The weights of the normal bundle of $f_k$ in $\TsP^{n-1}$ (written multiplicatively) are $a_k/a_i$ and $a_i/a_k\cdot\h$ for $i\not=k$. Hence what we got is indeed the elliptic Euler class of the $\Chamb$-negative part. 
\end{itemize}

\subsection{D5 resolutions and stable envelopes}

Let $X=X(\DD)$ be an arbitrary bow variety and fix a brane $\Ab\in \DD$. We assume that $c_k:=c(\Ab_k)>1$. Let $\wt X$ be a resolution of $\At$ in the sense of \S~\ref{subsec: D5 res}. In this section, we investigate the relation between the stable envelopes of $X$ and those of a resolution $\widetilde X$. Let $j: X\hookrightarrow \widetilde X$ be the embedding from Theorem \ref{theorem: embedding resolution D5 branes}. It is equivariant along the morphism $\varphi: \Tt\to \widetilde \Tt$ defined in equation \eqref{group homomorphism A-resolution}. We denote by $N_j$ be the associated $\Tt$-equivariant normal bundle. It can be shown that $N_j$ is topologically trivial and hence, given a chamber $\Chamb$, it decomposes in attractive, repelling, and fixed directions:
\[
N_j=N_j^+ + N_j^0 + N_j^-\in K_{\Tt}(\pt).
\]

The following result relates the stable envelopes of $X$ and $\wt X$. 
\begin{theorem}[Fusion of elliptic stable envelopes \cite{BR23}]
\label{Fusion of D5 branes for separated brane diagrams}
Fix some $f\in X^{\At}$ and a chamber $\Chamb$ for the $\At$-action on $X$. Let $F$ be the unique $\At$-fixed component of $\widetilde X$ containing $f$ and $\wt{\Chamb}_{\sigma}$ be any chamber for $\widetilde \At$ restricting to $\Chamb_{\sigma}$ on $\At$. Then 
\[
\varphi^{\oast}\left(\Stab^{*, F}_{\wt \Chamb_{\sigma}/\Chamb_{\sigma}}(\tilde f){(z\h^{- \gamma(f)})}\Big|_{\tilde f_\sharp}\right)e^*(N^-_j)  \Stab^{*, X}_{\Chamb_{\sigma}}(f)= j^{\oast} \varphi^{\oast}\Stab^{*, \wt X}_{\wt \Chamb_{\sigma}}(\tilde f) \qquad \forall \tilde f\in F^{\wt \At}.
\]
Here, $\gamma(f)$ is a multi-index whose entries $\gamma(f)_i$ denote the number of ties of $f$ connecting the $i$-th NS5 brane to the D5 branes left to the resolved D5 brane.
\end{theorem}
\begin{proof}
    The proof of the elliptic case is given in \cite[\S~6.3]{BR23} using the axiomatic characterization of the elliptic stable envelopes. The statements in K-theory and cohomology can be proved analogously.
\end{proof}

The appropriate choice chamber $\wt{\Chamb}_{\sigma}$ makes the previous formula becomes especially easy to apply:

\begin{corollary}
\label{corollary nicest one-term D5 resolution}
Assuming that $\wt{\Chamb_{\sigma}}/\Chamb_{\sigma}=\lbrace a'>a''\rbrace$ and choosing $\tilde f=\tilde f_\sharp$, we get 
    \begin{equation*}
    \prod_{s=1}^{c_k''} e(\h^s)e^*(N^-_j)  \Stab^{X}_{\Chamb_{\sigma}}(f)= j^{\oast}\varphi^{\oast}\Stab^{\wt X}_{\wt \Chamb_{\sigma}}(\tilde f_\sharp).
\end{equation*}
\end{corollary}

\
\begin{remark}
    As discussed in \cite[\S~6.6]{BR23}, Theorem \eqref{Fusion of D5 branes for separated brane diagrams} recovers the known fusion formula for R-matrices in cohomology and K-theory and generalizes them to elliptic cohomology.  For instance, if  $X$ has $m$ NS5 branes and two D5 branes with (local) charges $c_1$ and $c_2$, then the cohomology $H^*_{\Tt}(X)$ is a weight subspace of the quantum group representation 
    \begin{equation}
        \label{eq: qg rep}
    \Lambda^{c_1}\C^{m}(a_1)\otimes \Lambda^{c_2}\C^{m}(a_2).
    \end{equation}
    The fixed point basis corresponds to the standard coordinate basis and the actual weight is determined by the NS5 charges of $X$.
    
Now consider the D5 resolution of the first brane of X with charge decomposition $c_1=c'_1+c''_1$. The cohomology $H^*_{\widetilde{\Tt}}(\widetilde X^{\widetilde \At})$ is a weight subspace of  \begin{equation}
    \label{eq: qg rep 2}
    \Lambda^{c'_1}\C^{m}(a'_1)\otimes \Lambda^{c''_1}\C^{m}(a''_1) \otimes \Lambda^{c_2}\C^{m}(a_2).
\end{equation}
For generic $a_1'$ and $a_2''$ this representation is irreducible, but for $a_1'=a_1''\h^{-c''}$ i.e. on the divisor prescribed by $\varphi^{\oast}$, is reducible with irreducible quotient given by \eqref{eq: qg rep}. Moreover, the quotient map is identified with the bullback $j^{\oast}:H^*_{\Tt}(\widetilde X)\to H^*_{\Tt}(X^{\At})$. 
Theorem \ref{Fusion of D5 branes for separated brane diagrams} and \ref{corollary nicest one-term D5 resolution} express exactly this phenomenon from the point of view of stable envelopes, and imply that the (cohomological, K-theoretic, or elliptic) R-matrix for  \eqref{eq: qg rep} can be reconstructed from \eqref{eq: qg rep 2} via an explicit formula, known in the literature as fusion of R-matrices, see. \cite[Prop 6.19]{BR23}
\end{remark}

Iterating Corollary \ref{corollary nicest one-term D5 resolution}, we can relate the stable envelopes of a bow variety $X$ with those of partial flag varieties. Specifically, let $\wt X$ be the maximal resolution of $X$ in the sense of Corollary \ref{corollary any bow can be embedded in a flag variety}. We still denote by $j: X\to \wt X$, $\varphi: \Tt\to \wt\Tt$, and $\tilde f_{\sharp}\in \wt X^A$ the canonical embeddings and the distinguished resolution of a given fixed point $f\in X^A$, respectively. Notice that the $k$-th D5 brane $\Ab_k$ in $X$ is resolved by $c_k$ D5 branes $\Ab^{(1)}_k, \Ab^{(2)}_k, \dots, \Ab^{(c_k)}_k$ for every $k=1\dots,n$. Correspondingly, the equivariant parameter $a_k$ attached to $\Ab_k$ in $X$ is resolved by a tuple of equivariant parameters $(a_k^{(1)}, a_k^{(2)},\dots , a_{k}^{(c_k)})$. The pullback $\varphi^{\oast}$ is the specialization map setting $a_k^{(l)}\mapsto a_k\hbar^{l-1}$ for all $l=1,\dots, n$ and $l=1,\dots, c_k$.

By iteration of the previous lemma we get:

\begin{corollary}
\label{corollary nicest one-term D5 resolution full}
Let $\Chamb_\sigma$ be a chamber for the $\At$-action on $X$ and consider the chamber 
\[
{\Chamb}_{\wt \sigma}= \{a_{\sigma(1)}^{(c_{\sigma(1)})}<a_{\sigma(1)}^{(c_{\sigma(1)}-1)}<\ldots<a_{\sigma(1)}^{(1)}< \ldots<a_{\sigma(n)}^{(c_{\sigma(n)})}< a_{\sigma(n)}^{(c_{\sigma(n)}-1)}<\ldots<a_{\sigma(1)}^{(1)}
\}
\]
for the $\wt \At$-action on $\wt X$. We have
\begin{equation*}
    \label{nicest one-term D5 resolution}
    \prod_{k>0}\prod_{j=1}^{c_k-1}\prod_{i=1}^j e^*(\h^i)e^*(N^-_j)  \Stab^{X}_{\Chamb}(f)= j^{\oast}\varphi^{\oast}\Stab^{\wt X}_{\Chamb_{\wt \sigma}}(\tilde f_\sharp).
\end{equation*}
\end{corollary}

\begin{remark}
\label{rem: specialization maximal resolution}
    Notice that the pullback $\varphi^{\oast}$ simply enforces the change of variables
    \[
    (a^{(1)}_{k}, a^{(2)}_{k},\ldots,a^{(c_k)}_{k})=(a_{k}\h^{-c_k+1}, a_{k}\h^{-c_k+2},\ldots,a_{k})
    \]
    for all $k=1,\dots, n$.
\end{remark}

\subsection{Limit procedures}
Cohomological, K-theoretic, and elliptic stable envelopes form a hierarchy of progressively finer objects. 
Consequently, stable envelopes in different cohomology theories can be related by certain limit procedures. For concreteness, we discuss these limit procedures in the case of a bow variety $X$.

\begin{proposition}[{\cite[\S~4.5]{aganagic2016elliptic}}]

\label{prop: limit ell to K stab}
Let $X$ be a bow variety with $m$ NS5 branes and let $s=(s_{-1},\dots s_{m-1})\in \NN^{m}$ be generic. Then 
    \[
    \lim_{q\to 0} \left(\alpha\right)^{-1/2} \Stab_{\Chamb}^{E} \Big|_{z_{i+1}/z_{i}=q^{-s_i}}(F)\left(\alpha_F\right)^{1/2}=\uStab_{\Chamb}^{K,s} (F).
    \]
Here, $\alpha_F$ is the polarization of the fixed component $F$ induced by $\alpha$. Explicitly, it is given by the $\At$-fixed part of $\alpha|_F$. Therefore, with the normalization \eqref{eq: normalized K theoretic stab} we have
    \[
    \lim_{q\to 0}  \Stab_{\Chamb}^{E} \Big|_{z_i/z_{i+1}=q^{-s_i}}(F)=\Stab_{\Chamb}^{K,s} (F).
    \]

\end{proposition}

\begin{proposition}
\label{prop: Kthy to coh lim}
    The following formula holds
    \[
    \text{smallest cohomological degree of }\Chern\left(\Stab_{\Chamb}^{K,s} (F)\right)  =\Stab_{\Chamb}^{H} (F).
    \]
\end{proposition}
\begin{example}
Fix $s\in \R\setminus \Z$. Noticing that  
\begin{align}
\begin{split}
\label{eq: limit theta}
    & \lim_{q\to 0} \; \th(x)=\hata(x)=x^{1/2}-x^{-1/2},\\
        & \lim_{q\to 0} \; \frac{\th(xq^{-s})}{\th(q^{-s})}=x^{\lfloor s\rfloor+1/2},
\end{split}
\end{align}
we obtain
\begin{align*}
    \lim_{q\to 0}\Stab^{E}_{\Chamb}(f_k)\Big|_{z_1/z_2=q^{-s}}
    & = \lim_{q\to 0}\left(\prod_{i=1}^{k-1} \th\left(\frac{a_i}{t}\right)\cdot 
	\frac{\th\left(\frac{t}{a_k}\frac{z_1}{z_2}\h^{k-1}\right)}{\th\left(\frac{z_1}{z_2}\h^{k-2}\right)}
	\cdot 
	\prod_{i=k+1}^n \th\left( \frac{t}{a_i}\h\right)\right)
    \\
    &=\prod_{i=1}^k\hata\left(\frac{a_i}{t}\right)\cdot 
	\left(\frac{t}{a_k}\hbar\right)^{\lfloor s\rfloor+1/2}
	\cdot 
	\prod_{i=k+1}^n \hata\left( \frac{t}{a_i}\h\right)
    \\
    &=\Stab^{K,s}_{\Chamb}(f_k)
\end{align*}
where in the last line we used \eqref{eq: Kstab proj}. For the K-theoretic to cohomological limit as in Proposition~\ref{prop: Kthy to coh lim}, notice that
\begin{itemize}
    \item applying the Chern character amounts to replacing a given Chern root $x$ with $\exp(x)$;
    \item picking the smallest cohomological degree amounts to picking the leading term in the Taylor expansion of $\exp$.
\end{itemize} 
Hence, the factors of the form $\ch(\hata(x))=(e^{x/2}-e^{-x/2})$ contribute to $x$ while the factors of the form $\ch(x)=e^{x}$ contribute to the unit $1$.
Therefore, we obtain that 
\begin{align*}
    &\text{smallest cohomological degree of }\Chern\left(\Stab_{\Chamb}^{K,s} (F)\right) 
    \\
    &=\prod_{i=1}^{k-1} (a_i- t)\cdot 
	\prod_{i=k+1}^n  (t-a_i+\h)
    \\
    &=\Stab^H_{\Chamb}(f_k).
\end{align*}
    In the last line we used \eqref{eq; coh stab proj}.
\end{example}

\section{The shuffle algebra}
\label{sec:shuffle algebra}


Consider a left dimension vector $d=(\ldots, d_{-2},d_{-1}, d_0) \in \N^{\Z_{\leq 0}}$, cf. \S~\ref{sec:brane diagrams}. We still assume that $d_{-i}=0$ for $i\gg 0$.
For any $d\in \N^{\Z_{\leq 0}}$, we define vectors spaces
\[
\Ht_d^{H}, \qquad \Ht_d^{K}, \qquad \Ht_d^{E}
\]
of functions in the variables 
\[
\begin{array}{ll}
t_{ki}  \text{ for } k\leq 0 \text{ and } 1\leq i \leq d_{-k}, & \text{(topological variables)}
\\
\h & \text{(deformation variable)}
\end{array}
\]
that are symmetric in the $t_{ki}$ variables for the same $k<0$ value. Note that we do not require symmetry for the $t_{0i}$ variables. 
In the elliptic version $\Ht^E_d$ the functions also depend on 
\[
\begin{array}{ll}
z_1, z_2, z_3, \ldots \phantom{\hskip 3.1 true cm} & \text{(dynamical variables).}
\end{array}
\]
A star $*$ in the superscript instead of $H, K, E$ indicates that we talk about all three versions at the same time.
To define a multiplication on $\oplus_d \Ht_d^*$  we need some definitions.

\begin{definition} \ \newline
$\bullet$ For left dimension vectors $d', d''$ with $d=d'+d''$ consider the variables  
\begin{equation}\label{tptpp}
t'_k=\{t'_{k1},t'_{k2},\ldots, t'_{k,d'_k}\},
\ \ 
t''_k=\{t''_{k1},t''_{k2},\ldots, t''_{k,d''_k}\},
\ \ \text{for } k\leq 0.
\end{equation}
A {\em shuffle} of these variables is a choice of dividing the variables $\{t_{k1},$ $\ldots,$ $t_{k,d_k}\}$ into a $d'_k$ and a $d''_{k}$ element set, and identifying these sets with $t'_k$ and $t''_k$ respectively, {\em for every $k<0$}. For $k=0$ we identify  
\[
\begin{array}{llll}
    t'_{0i} & \text{ with } & t_{01}, t_{02}, \ldots, t_{0d'_0} & \text{in this oder, and }
\\
    t''_{0i} & \text{ with } & t_{0d'_{0}+1}, t_{0d'_{0}+2}, \ldots, t_{0d_0} & \text{in this order.}
\end{array}
\]
\newline
$\bullet$ 
Define the ``kernel'' function
\[
\phi^{*}_{d',d''}= 
\prod_{k<0} 
\frac{ e_{\h}^*(t'_k,t''_k) e_{\h}^*(t''_{k+1},t'_k) e^*(t''_k,t'_{k+1})}{e^*(t''_k,t'_k)},
\]
where 
\[
\begin{array}{ll}
    e^E_{\h}(A,B)=\prod_{ \substack{a\in A \\ b\in B}} \th\left(\frac{\h b}{a}\right)
    & e^E(A,B)=e^E_1(A,B) 
    \\[8pt]
    e^K_{\h}(A,B)=\prod_{ \substack{a\in A \\ b\in B}}\ag{\h b}{a} & e^K(A,B)=e^K_1(A,B)
    \\ [8pt] 
    e^H_{\h}(A,B)=\prod_{\substack{a\in A \\ b\in B}} (b-a+\h) 
    & e^H(A,B)=e^H_0(A,B).
\end{array}
\]
\newline
$\bullet$
Define the $z$-shift operator 
\[
Z_d[f]=f|_{z_k=\h^{-c_k} z_k} \qquad \text{(recall that $c_k=d_{-k+1}-d_{-k}$ is the charge of $\Zb_{k}$).}
\]
\end{definition}

\begin{definition}
\label{def: shuffle algebra}
Define the bilinear multiplication $\star: \Ht_{d'}^* \otimes \Ht_{d''}^* \to 
\Ht_{d'+d''}^*$ as follows. 
For $f'\in \Ht_{d'}^*, f''\in\Ht_{d''}^*$ the product $f' \star  f''\in \Ht_d^*$ ($d=d'+d''$) is 
\[
\Shuffle_{k<0}
\Bigl(
f'(t')  \cdot
Z_{d'}\left[ f''(t'')\right]\cdot
\phi^{*}_{d',d''}
\Bigl),
\]
where $\Shuffle_{k<0}$ means taking the $\prod_{k<0} \binom{d_k}{d'_k}$-term sum for all shuffles. The z-shift operator $Z_{d'}[-]$ only plays a role in the elliptic flavor.
\end{definition}

The obtained $d$-graded algebra 
\begin{equation}
    \label{eq: copmbinatorial shuffle algebra}
    \left(\Ht^*,\star\right)=\left(\bigoplus_{d} \Ht^*_d,\star\right)
\end{equation}
while not commutative, is associative. This can be either seen via a direct computation or by geometric means, as reviewed in the next section.  

\begin{example}
Let $d'=d''=(\ldots,0,1,d_0=1)$. For the elliptic multiplication we have
\begin{multline*}
f'(t_{-11}, t_{01}) 
\star
f''(t_{-11}, t_{01})
= 
f'(t_{-11},t_{01})
Z[f''(t_{-12},t_{02})]
\frac{\thh\left(\frac{t_{-12}}{t_{-11}}\h\right)\thh\left(\frac{t_{-11}}{t_{02}} \h\right)\thh\left(\frac{t_{01}}{t_{-12}}\right)}{\thh\left(\frac{t_{-11}}{t_{-12}}\right)}
\\
+[t_{-11}\leftrightarrow t_{-12}],
\end{multline*}
where $Z[-]$ denotes the substitution $z_2\mapsto z_2h^{-1}$.    
\end{example}

\section{Geometry of the shuffle algebra} 

In this section we recall the geometric nature of the shuffle algebra, that is, we discuss how it emerges as the cohomological (resp. K-theoretic and elliptic) Hall algebra of the quiver~$A_\infty$.
\subsection{Parabolic induction}
Fix a quiver $Q=(Q_0,Q_1)$ with set of vertices $Q_0$ and source and target maps $s,t: Q_1\to Q_0$. The double of $Q$ is by definition the quiver $\ol Q=(\ol Q_0, Q_1)$ such that $\ol Q_0=Q_0$ and $\ol Q_1=Q_1\sqcup  (Q_1)^*$, where $(Q_1)^*$ is a copy of the set $Q_1$ but with arrows of the opposite orientation. On other words, for every $a\in Q_1$, we have $a^*\in (Q_1)^*$ such that $s(a^*)=t(a)$ and $t(a^*)=s(a)$. 

Fix dimension and framing vectors $(\dd, \ww)\in \NN^I\times \NN^{Q_0}$. Let 
\[
\Rep_Q(\dd,\ww)=\bigoplus_{a\in Q_1} \Hom(\C^{\dd_{s(a)}}\C^{\dd_{t(a)}})\bigoplus_{i\in Q_0}\Hom(\C^{\ww_{i}}, C^{\dd_i})
\]
be the space space of $(\dd, \ww)$-dimensional framed representations. It admits an actions of $\GL(\ww)=\prod_i \GL(\ww_i)$ and $\GL(\dd)=\prod_i \GL(\dd_i)$. Notice that $\Rep_{\ol Q}(\dd,\ww)=T^*\Rep_Q(\dd,\ww)$. Hence, if the quiver is doubled or tripled, we can consider the additional action of a one dimensional torus $\Cs_{\h}$ rescaling the cotangent directions. We denote by $T_{\ww}$ the standard maximal torus of $\GL(\ww)$ and set $\Tt_{\ww}=T_{\ww}\times \Cs_{\h}$.

Consider a pair of decompositions $\dd=\dd'+\dd''$ and $\ww=\ww'+\ww''$ and $I$-graded subspaces $\C^{\dd'}\subset \C^{\dd}$  and $\C^{\ww'}\subset \C^{\ww}$. Let $ Z\subset \Rep_{Q}(\dd, \ww)$ be the subspace of $\Rep_{\ol Q}(\dd, \ww)$ consisting of those representations of $Q$ that preserve the subspaces $\C^{\dd'}\subset \C^{\dd}$ and  $\C^{\ww'}\subset \C^{\ww}$.  Alternatively, $Z$ can be characterized as follows. Consider the (unique) cocharacter $\lambda: \Cs\to \GL(\dd)\times \GL(\ww)$ acting trivially on the subspaces $\C^{\dd'}\subset \C^{\dd}$ and $\C^{\ww'}\subset \C^{\ww}$ with weight one on their complements. Consider the action of $\Cs$ on $ \Rep_{\ol Q}(\dd, \ww)$ induced by $\lambda$. Then $Z$ is the subspace of $\Rep_{\ol Q}(\dd, \ww)$ with non-negative weights.

Consider the composition $\lambda': \Cs\to \GL(\dd)\times \GL(\ww)\to\GL(\dd)$. Let $L\cong \GL(\dd')\times \GL(\dd'')$ be the Levi subgroup centralizing $\lambda'$, and let $P\subset \GL(\dd)$ be the parabolic subgroup containing $\lambda'$ such that $\lambda'$ acts on $\text{Lie}(P)$ with non-negative weights. 
By construction, there are canonical maps 
\begin{equation*}
    \begin{tikzcd}
    \Rep_{\ol Q}(\dd',\ww=')\times \Rep_{\ol Q}(\dd'',\ww'') &  Z\arrow[l] \arrow[r, hook]  & \Rep_{\ol Q}(\dd,\ww)
    \end{tikzcd}
\end{equation*}
Here, the two leftward pointing maps are obtained by restricting a representation to the subspaces $\C^{\dd'},\; \C^{\ww'}$ and by taking quotients, respectively. Set $\FM_{\ol Q}(\dd,\ww):=[\Rep_{\ol Q}(\dd,\ww)/\GL(\dd)]$. Passing to quotients by $L\cong \GL(\dd')\times \GL(\dd'')$, $P$ and $\GL(\dd)$, we get a correspondence
\begin{equation}
\label{eq: basic correspondence coha}
    {\FM_{\ol Q}(\dd',\ww')}\times {\FM_{\ol Q}(\dd'',\ww'')} \xleftarrow{q} [Z/P] \xrightarrow{p}\FM_{\ol Q}(\dd,\ww),
\end{equation}
which is often referred to as parabolic induction. 
It is standard to check that $q$ is smooth and $p$ is proper. In addition, the following composition is also proper:
\begin{equation}
    \label{fundamental correspondence stacks}
    [Z/P] \xrightarrow{q\times p} {\FM_{\ol Q}(\dd',\ww')}\times {\FM_{\ol Q}(\dd'',\ww'')} \times \FM_{\ol Q}(\dd,\ww).
\end{equation}
\subsection{Cohomology}

The proper morphism \eqref{fundamental correspondence stacks} defines a class in the Borel-Moore homology
\begin{equation}
    \label{eq: fund class corr cohomology}
    \Corr_{\dd',\dd''}^{\ww',\ww''} \in H^{\BM}_{\Tt_{\ww}}({\FM_{\ol Q}(\dd',\ww')}\times {\FM_{\ol Q}(\dd'',\ww'')} \times \FM_{\ol Q}(\dd,\ww))
\end{equation}
It is a standard fact (see eg. \cite[\S~2]{kontsevich2011cohomological}) that these correspondences are associative with respect to the convolution product in Borel-Moore homology, i.e. we have 
\begin{equation}
\label{eq: convolution shufle algebra}
    \Corr_{\dd'+\dd'',\dd'''}^{\ww'+\ww'',\ww'''}\circ \left(\Corr_{\dd',\dd''}^{\ww',\ww''}\boxtimes 1^{\dd'''}_{\ww'''}\right)=\Corr_{\dd',\dd''+\dd'''}^{\ww',\ww''+\ww'''}\circ \left(1^{\dd'}_{\ww'}\boxtimes \Corr_{\dd'',\dd'''}^{\ww'',\ww'''} \right)
\end{equation}
as classes in 
\[
H^{\BM}_{\Tt_{\ww}}({\FM_{\ol Q}(\dd',\ww')}\times {\FM_{\ol Q}(\dd'',\ww'')}\times {\FM_{\ol Q}(\dd''',\ww''')}  \times \FM_{\ol Q}(\dd,\ww)).
\]
Here, $1^{\dd}_{\ww}\in H^{\BM}_{T_{\ww}}({\FM_{\ol Q}(\dd,\ww)} \times \FM_{\ol Q}(\dd,\ww)) $ denotes the fundamental class of the diagonal.

From now on, we denote the $\C^\times_{\hbar}$-equivariant cohomology of the point by $\Bbbk^H$. Hence, $\Bbbk^H$ is the polynomial ring $\C[{\hbar}]$. Let now $(Q, I)$ be a quiver. We now define a $\N^I\times \N^I$-graded $\Bbbk^H$-module
\[
\Ht^H_Q=\bigoplus_{\dd,\ww\in \N^I} \Ht^H(\dd,\ww) \qquad \Ht^H(\dd,\ww):=H_{T_{\ww}}(\FM_{\ol Q}(\dd,\ww))=H(\FM_{\ol Q}(\dd,\ww)/\Tt_{\ww}).
\]
As a $\Bbbk^H$-module, $\Ht(\dd,\ww)$ is given by
\[
\Bbbk^H[\liet_{\ww}\times \liet_{\dd}]^{W_{G_{\dd}}},
\]
where $\liet_{\ww}$ and $\liet_{\dd}$ are the Lie algebras of the maximal tori of $\GL(\ww)$ and $\GL(\dd)$, respectively. We will often drop the reference to the quiver and write $\Ht^H$ in place of $\Ht^H_Q$
. 
We equip $\Ht^H$ with a multiplication $\N^I\times \N^I$-graded $\Bbbk^H$-algebra structure by defining multiplication maps on the $\N^I\times \N^I$-graded components graded component by setting
\begin{equation}
    \label{eq: mult coha}
    \begin{tikzcd}
\vmult:\Ht^H(\dd',\ww')\otimes_{\Bbbk^H}\Ht^H(\dd'',\ww'')\arrow[rr, "{p_{\oast}q^{\oast}}\cdot e^H(\hbar \lien)"] && \Ht^H(\dd'+\dd'',\ww'+\ww'').
\end{tikzcd}
\end{equation}
Here, the maps $q^{\oast}$ and $p_{\oast}$ are induced by the correspondence \eqref{fundamental correspondence stacks} and $\lien$ is the tautological bundle associated to the Lie algebra of nilpotent matrices in the parabolic group $P$. Equivalently, we set 
\begin{equation}
    \label{eq: mult coha v2}
    {\vmult}(\alpha\otimes \beta)=  (p_3)_{\oast} \left(\Corr_{\dd',\dd''}^{\ww',\ww''}\cup p_{12}^{\oast}(  e(\hbar \lien)\cup\alpha\otimes \beta)\right)
\end{equation}
where 
\[
\begin{tikzcd}
    & {\FM_{\ol Q}(\dd',\ww')}\times {\FM_{\ol Q}(\dd'',\ww'')} \times \FM_{\ol Q}(\dd,\ww) \arrow[dl, swap, "p_{12}"]\arrow[dr, "p_{3}"]& \\
    {\FM_{\ol Q}(\dd',\ww')}\times {\FM_{\ol Q}(\dd'',\ww'')} & & \FM_{\ol Q}(\dd,\ww) 
\end{tikzcd}
\]
are the canonical projections. We call the associative unital $\N^I\times \N^I$-graded algebra $(\Ht^H,\mathsf{m})$ the framed shuffle algebra of quiver $(Q, I)$. The adjective ``framed'' serves to stress that we are only quotienting the framing vertices by the torus $A_{\ww}$ rather than $\GL(\ww)$. This affects $\Ht^H_Q$ both as a vector space and as an algebra.

\begin{remark}
    By dropping the class $e(\hbar \lien)$ in the definition of $\vmult$, one gets another well defined algebra structure on the vector space $\Ht^H_Q$. However,the twist by the Euler class $e(\hbar \lien)$ emerges naturally in connection with he geometry of the moment map $\mu_{\dd,\ww}: \Rep_Q(\dd,\ww)=T^*\Rep_Q(\dd,\ww)\to \lieg_{\dd}$ associated to the $\GL(\dd)$-action. Indeed, replacing the stacks of the form $\FM_{\ol Q}(\dd,\ww)=(T^*\Rep_Q(\dd,\ww))/\GL(\dd)$ with the cotangent bundles $T^*(\Rep_{Q}(\dd,\ww)/\GL(\dd)) \cong \mu_{\dd,\ww}^{-1}(0)/G(\dd)$ in \eqref{eq: basic correspondence coha}, one can construct \cite{SV, YZ} an algebra structure on
    \[
    \Ht_{\Pi_Q} = \bigoplus_{\dd,\ww\in \N^{Q_0}} H^{\BM}(\mu_{\dd,\ww}^{-1}(0)/(G(\dd)\times \Tt_{\ww}))
    \]
    which is known as the (framed) preprojective CoHA. Notice that there is a canonical morphism 
    \[
    \iota_{\oast}: H^{\BM}(\mu_{\dd,\ww}^{-1}(0)/G(\dd)\times \Tt_{\ww})\to H(\FM_{\ol Q}(\dd,\ww)/\Tt_{\ww}) = H(T^*\Rep_{Q}(\dd,\ww)/(\GL(\dd)\times \Tt_{\ww})).
    \]
    induced by the closed embedding $\mu_{\dd,\ww}^{-1}(0)\hookrightarrow T^*\Rep_{Q}(\dd,\ww)$. However, the direct sum over $\dd,\ww\in \N^{Q_0}$ induces a morphism of algebras $\Ht_{\Pi_Q}\to \Ht_{Q}$ only if the multiplication map $\vmult$ on the target is given by \eqref{eq: mult coha}, and hence with the twist by $e(\hbar \lien)$. 

\end{remark}

\subsection{K-theory}

In K-theory the construction is similar. Namely, we define an $\Bbbk^K=\C[\hbar, \hbar^{-1}]$-algebra structure on 
\[
\Ht^K_Q=\bigoplus_{\dd,\ww\in \N^I} \Ht^K(\dd,\ww) \qquad \Ht^K(\dd,\ww):=K_{T_{\ww}}(\FM_{\ol Q}(\dd,\ww))=K(\FM_{\ol Q}(\dd,\ww)/\Tt_{\ww}).
\]
by replacing the multiplication map in cohomology \eqref{eq: mult coha} with its K-theoretic analog. We call the resulting associative unital $\N^I\times \N^I$-graded algebra $(\Ht^H, \vmult)$ the K-theoretic framed shuffle algebra of quiver $(Q, I)$.

\begin{remark}
    Restricting to the subalgebra such that $\ww=0$, one recovers the famous K-theoretic shuffle algebra studied by Negut and his collaborators \cite{negutshufflerevisited, negut2022shuffle, negut2022shufflermatrices, Negut_shufflequantum}. 
\end{remark}

\subsection{Elliptic cohomology}

In this section we define an elliptic version of the framed shuffle algebras from the previous section.
The $T_{\ww}$-equivariant elliptic cohomology of $\FM(\dd,\ww)$ is the abelian variety
\begin{equation}
\label{eq: unextended eliptic base}
    \Ell_{T_{\ww}}(\FM(\dd,\ww))=\Ell_{T_{\ww}\times G_{\dd}}(\Rep_{\ol Q}(\dd,\ww))=E_{\hbar}\times \prod_{i\in I} E^{\ww_i}\times \prod_{j\in I} E^{(\dd_j)}.
\end{equation}
Consider decompositions $\dd=\dd'+\dd''$ and $\ww'=\ww'+\ww''$. They induce a canonical map $ \FM(\dd',\ww')\times \FM(\dd'',\ww'')\to \FM(\dd,\ww)$ and hence a morphism in cohomology
\begin{equation}
\label{eq: iota}
    \iota:\Ell_{T_{\ww'}}(\FM(\dd',\ww'))\times_{E_{\hbar}} \Ell_{T_{\ww''}}(\FM(\dd'',\ww''))\to \Ell_{T_{\ww}}(\FM(\dd,\ww))
\end{equation}
Let $\Sh{L}$ be a line bundle on $\Ell_{T_{\ww}}(\FM(\dd,\ww))$. Let $\FN$ be the virtual normal bundle of  \eqref{fundamental correspondence stacks}. Pushing forward along the latter morphism, we get a morphism 
\begin{equation}
    \label{eq: elliptic mult untwisted}
    (p\times q)_*\Sh{O}_{\Ell_{T_{\ww}}([Z/P])}\xrightarrow{(p\times q)_{\oast}} \left(\iota^* \Sh{L}^{-1}\otimes \th(\FN)\right)\boxtimes \Sh{L},
\end{equation}
We denote by $\prescript{E}{}{\Corr_{\dd',\dd''}^{\ww',\ww''}}$  the image of the global section $1$ in the domain. It is the elliptic analog of the fudamental class \eqref{eq: fund class corr cohomology}. Notice that it depends on the choice of $\Sh{L}$ and induces, via convolution, a morphism
\begin{equation}
    \label{eq: Eha mult sheaf}
    \iota_*\left(\iota^* \Sh{L}\otimes \th(-\FN)\right)\to  \Sh{L},
\end{equation}
In analogy with \eqref{eq: convolution shufle algebra}, we have
\begin{equation}
\label{eq: convolution shufle algebra ell}
    \prescript{E}{}{\Corr_{\dd'+\dd'',\dd'''}^{\ww'+\ww'',\ww'''}}\circ \left(\prescript{E}{}{\Corr_{\dd',\dd''}^{\ww',\ww''}}\boxtimes 1^{\dd'''}_{\ww'''}\right)=\prescript{E}{}{\Corr_{\dd',\dd''+\dd'''}^{\ww',\ww''+\ww'''}}\circ \left(1^{\dd'}_{\ww'}\boxtimes \prescript{E}{}{\Corr_{\dd'',\dd'''}^{\ww'',\ww'''}} \right)
\end{equation}
To construct the shuffle algebra, we proceed as follows. Set
\[
\Ht^{E,0}_Q=\bigoplus_{\dd,\ww\in \N^I} \Ht^{E,0}_Q(\dd,\ww)\qquad \Ht^{E,0}_Q(\dd,\ww)=\Gamma(\Sh{O}_{\Kthy_{T_{\ww}}(\FM(\dd,\ww))})
\]
Pulling back \eqref{eq: Eha mult sheaf} by the Chern character morphism
\[
\ch : \Kthy_{T_{\ww}}(\FM(\dd,\ww))\to  \Ell_{T_{\ww}}(\FM(\dd,\ww))
\]
and taking global sections, we get a map 
\[
\ol \vmult^0: \Ht^{E,0}_Q(\dd'',\ww'')\otimes \Ht^{E,0}_Q(\dd'',\ww'')\to \Ht^{E,0}_Q(\dd,\ww)
\]
In analogy with \eqref{eq: mult coha}, we set $\vmult^0:=\ol {\vmult}^0\circ \th(\hbar {\lien})$. 
The associativity constrain \eqref{eq: convolution shufle algebra ell} implies that the pair $(\Ht^{E,0},\mathsf{m}^0)$ is an associative unital $\N^I\times \N^I$-graded algebra over $\Bbbk=\Gamma(\Sh{O}_{\Kthy_{\Cs_{\hbar}}})$. Notice that, by now, the algebra $(\Ht^{E,0},\mathsf{m})$ does not depend on the dynamical parameters.

\subsection{Introducing the dynamical parameters}

As in the theory of elliptic stable envelopes, the dynamical parameters can be formally introduced considering the  K\"ahler torus $Z=(\Cs)^{I}$, with coordinates $z=\{z_i\}_{i\in I}$, and replacing the space \eqref{eq: unextended eliptic base} with 
\begin{equation*}
        \Ell_{T_{\ww}\times Z}(\FM(\dd,\ww))=E_{\hbar}\times E^{I}\times \prod_{i\in I} E^{\ww_i}\times \prod_{j\in I} E^{(\dd_j)}.
\end{equation*}
Notice that although the action of $Z$ on $\FM(\dd,\ww)$ is trivial, we can twist the algebra structure introduced in the previous section by automorphisms of $E^I$. Specifically, we set 
\[
\Ht^{E}_Q=\bigoplus_{\dd,\ww\in \N^I} \Ht^{E}_Q(\dd,\ww)\qquad \Ht^{E,0}_Q(\dd,\ww)=\Gamma(\Sh{O}_{\Kthy_{T_{\ww}\times Z}(\FM(\dd,\ww))})
\]
and define a map
\[
\vmult: \Ht^{E}_Q(\dd'',\ww'')\otimes \Ht^{E}_Q(\dd'',\ww'')\to \Ht^{E}_Q(\dd,\ww)
\]
as the composition of $\vmult^0\circ (\id \times \tau')$, where tau is the shift operator 
\begin{equation}
\label{eq: dynamical shift operator}
    \tau' (f(z))= f(z-\hbar \mu(\dd',\ww')).
\end{equation}
Here, $\mu(\dd',\ww')=\ww'-C_Q\dd'$ is the weight associated with the Cartan matrix of the qiver $C_Q$. It is straightforward that  $(\Ht^{E},\mathsf{m})$ is still an associative unital $\N^I\times \N^I$-graded algebra. 
We remark that, although the introduction of the dynamical parameter may seem arbitrary at first, it will crucial to relate the elliptic shuffle algebra to the elliptic stable envelopes (cf. Prop. \ref{prop: compatibility elliptic stab coha}).

\subsection{The case \texorpdfstring{$Q=A_{\infty}$}{Am}} 
Assume now that $Q$ is the linear quiver $A_{\infty}$. We label the vertices by negative integers. Therefore, a dimension vector $\dd$ for $A_{\infty}$ is given by a tuple of non-negative integers $\dd=(\dots, d_{-2}, d_{-1})$. We require all dimension vectors to have finite support, i.e. that $\dd_{-i}=0$ for $i\gg 0$. For a given non-negative integer $d_0$, we denote by $k\delta_{-1}$ the framing vector $(0,\dots, 0, d_0)$.

Consider the sub-algebra of $(\Ht_{A_\infty} ,\; \vmult)$ given by
\begin{equation}
\label{eq: geometric shuffle algebra}
    \left(\bigoplus_{\dd, d_0 } \Ht^*_{A_\infty}(\dd, d_0\delta_{-1}),\; \vmult\right)
\end{equation}

Notice that $\Ht^*_{A_\infty}(\dd, d_0\delta_{-1})$ is the cohomology (resp. K-theory, elliptic cohomology) of the stack $T^*\Rep_Q(\dd,d_0\delta_{-1})/(\GL(\dd)\times \Tt_{k\delta_{-1}})$, which is homotopic to $B\GL(\dd)\times BT_{k\delta_{-1}}\times B\Cs_{\hbar}$. Therefore, we can identify $\Ht^*_{A_\infty}(\dd, d_0\delta_{-1})$ with the space of functions $\Ht_{d}$ defined in \S~\ref{sec:shuffle algebra}. In particular, the variables $t_{ki}$ for $k<0$ and are identified with the Chern roots of $B\GL(\dd)$, the variables $t_{0i}$ with the Chern roots of $ BT_{k\delta_{-1}}$ and $\hbar$ with the Chern root of $B\Cs_{\hbar}$. Applying the localization formula in cohomology (resp. K-theory, elliptic cohomology) to compute the multiplication map $\vmult$, we get

\begin{proposition}
    Under this identification, the cohomological and elliptic versions of the algebra \eqref{eq: geometric shuffle algebra} coincide with \eqref{eq: copmbinatorial shuffle algebra}. The K-theory versions coincide up to replacing the $\hat A$-genus $\hat a(x)=(x^{1/2}-x^{-1/2})$ entering in the definition of \eqref{eq: copmbinatorial shuffle algebra} with $(1-x^{-1})$.
\end{proposition}

\begin{remark}
    The discrepancy in the two K-theory version is just a matter of convenience. In fact, despite not being defined in integral K-theory, the K-theory version of the algebra \eqref{eq: copmbinatorial shuffle algebra} can be directly recovered from its lift in elliptic cohomology by taking the limit $\lim_{q\to 0}\th(x)=\hata(x)$. On the stable envelope side, this choice corresponds to the normalization \eqref{eq: normalized K theoretic stab}.
\end{remark}

\section{Shuffle formula for stable envelopes of bow varieties}
\label{sec:shuffle_formula_for Stab}

\subsection{The \texorpdfstring{$\Wt$}{W-tilde} function}
\label{sec:Wt}
First we associate a function $\Wt_{\Chamb}(f)(t_{ki},z_i,\h)$ to a fixed point $f\in \Ch(\DD)$. It will depend on the $t_{ki}$ variables for $k\leq 0$ only, and it will be symmetric in $t_{ki}$ for the same $k$ {\em only for} $k<0$. It will not be symmetric in the $t_{0i}$ variables, so the $t_{0i}$ variables require extra attention. In the upper index we will put $H$, $K$, or $Ell$ according to which flavor we consider.

\subsubsection{One tie functions}
\label{sec:one tie functions}

For a tie $\One_{kl}$ we define the basic functions 
\begin{multline*}
\Wt^E(\One_{kl})=
\prod_{i=1}^{k-1} \frac{\th({t_{-i1}}/{t_{-i+1,1}} \cdot {z_k}/{z_i}){\th(\h)}}{\th( {z_k}/{z_i} \cdot \h^{-1})},
\\
\Wt^{K,s}K(\One_{kl})=
\prod_{i=1}^{k-1} \hata(\h)
\left( \frac{t_{-i1}\h}{t_{-i+1,1}}\right)^{m_{ik}+1/2},
\qquad
\Wt^H(\One_{kl})={\h^{k-1}}.
\end{multline*}
Here, $m_{ik}=\sum_{j=i}^{k-1} s_i$ and the slope parameters $s_i\in \R$ are generic.
Notice that these functions do not depend on $l$, only on $k$.

\subsubsection{The recursion}
For a chamber $\Chamb$ we need a preferred order of the possible ties $\One_{kl}$.

\begin{definition}
    For a permutation $\sigma\in S_n$ we define a total order on the set $\{\One_{kl}\}$ of possible ties of $\DD$ by
$\One_{kl}<_{\sigma}\One_{k'l'}$ if either $\sigma^{-1}(l)<\sigma^{-1}(l')$, or $l=l'$ and $k>k'$.
\end{definition}

For the ties in Figure~\ref{fig:fixed point restriction} and $\sigma=(2,3,1)$ the order is
$\One_{42} <_{\sigma} \One_{32} <_{\sigma}\One_{23} <_{\sigma}\One_{31} <_{\sigma}\One_{11}$.

\begin{definition}
\label{def: W tilde}
    For the fixed point $f=\cup_{kl}\One_{kl}$ on a bow variety, and chamber $\Chamb_{\sigma}$ define
    \[
    \Wt_{\Chamb_{\sigma}}^*(f)=\prod_{<_\sigma} \Wt^*(\One_{kl}),
    \]
    where the product means $\star$-multiplication, in increasing $<_\sigma$-order.  
\end{definition}

Notably, the dependence of the right-hand side on the chamber $\Chamb=\Chamb_{\sigma}$ arises solely through the ordering $<_{\sigma}$---the individual factors themselves remain unchanged. This phenomenon has been observed in various contexts where geometric or algebraic objects, such as quiver polynomials, Chern-Schwartz-MacPherson classes, and stable envelopes for quivers, are expressed as ordered products of elementary building blocks within different shuffle or cohomological Hall algebras \cite{RRcoha,RRcsmcoha,botta2021shuffle}.

\subsubsection{Examples for \texorpdfstring{$\Wt$}{W-tilde} functions}
\begin{example}\label{ex:TP1_Wtilde}
Consider $r=c=(1,1)$ and the corresponding $X=\TsP^1$. Let $\Chamb_1=\Chamb_{id}=(a_1<a_2)$, and $\Chamb_2=(a_2<a_1)$. The two fixed points are $f_1=\One_{12}\cup \One_{21}$ (ties intersect) and $f_2=\One_{11} \cup \One_{22}$ (ties do not intersect).
We have

\begin{align*}
\Wt^E_{\mC_1}(f_1)=
\Wt^E_{\mC_2}(f_2) & 
= \frac{\th(t_{-11}/t_{01} \cdot z_2/z_1){\th(\h)}}{\th(z_2/z_1\cdot \h^{-1})} 
\cdot 1
\cdot
\th(t_{-11}/t_{02}\h),
\\
\Wt^E_{\mC_2}(f_1) 
= \Wt^E_{\mC_1}(f_2)
& = 
1 \cdot Z\left[
\frac{\th(t_{-11}/t_{02} \cdot z_2/z_1){\th(\h)}}{\th(z_2/z_1\cdot \h^{-1})} \right]
\cdot
\th(t_{01}/t_{-11})
\\
& = 
1 \cdot 
\frac{\th(t_{-11}/t_{02} \cdot z_2/z_1 \h){\th(\h)}}{\th(z_2/z_1)}
\cdot
\th(t_{01}/t_{-11})
\end{align*}
where $Z$ is the operation $z_1\mapsto z_1\h^{-1}$.
In cohomology we have 
\begin{align*}
\Wt^H_{\mC_1}(f_1)=
\Wt^H_{\mC_2}(f_2) & 
= \h \cdot 1
\cdot
(t_{-11}-t_{02}+\h),
\\
\Wt^H_{\mC_2}(f_1) 
= \Wt^H_{\mC_1}(f_2)
& = 
1 \cdot \h 
\cdot
(t_{01}-t_{-11}).
\end{align*}
\end{example}

\begin{example}
\label{ex:0101,11} Consider $X(r=(0101),c=(11))$. Denote
\begin{align*}
A:= & \frac{
\th({t_{-31}}/{t_{-21}}\cdot {z_4}/{z_3})
\th({t_{-21}}/{t_{-11}}\cdot {z_4}/{z_2})
\th({t_{-11}}/{t_{01}}\cdot {z_4}/{z_1})
}
{
\th({z_4}/{z_3}\cdot h^{-1})
\th({z_4}/{z_2}\cdot h^{-1})
\th({z_4}/{z_1}\cdot h^{-1})
}\th(\h)^3
\\
B:= & \frac{
\th(t_{-11}/t_{01}\cdot z_2/z_1)
}{
\th(z_2/z_1\cdot \h^{-1})
}\th(\h),
\end{align*}
and ---only for this example---let the shuffle operation be $\Shuffle(F)=F+F(t_{-11}\leftrightarrow t
_{-12})$. 
Then we have
\[
\Wt^E_{\mC}(\One_{42})=A, \qquad\qquad \Wt^E_{\mC}(\One_{21})=B
\]
for arbitrary $\mC$. For $\mC_1=(a_1<a_2)$  we have 
\[
\Wt^E_{\mC_1}(\One_{42}\cup \One_{21})=
\Shuffle\left( B \cdot Z[T[A]] 
\frac{
\th(t_{-12}/t_{-11}\h)
\th(t_{-11}/t_{-21})\th(t_{01}/t_{-12})\th(t_{-11}/t_{02}\h)}
{
\th(t_{-11}/t_{-12})}
\right).
\]
Here the Z-shift operator $Z[\ ]$ is the substitution $z_2\mapsto z_2\h^{-1}$, and T-shift operator $T[\ ]$ is the substitutions $t_{-11}\mapsto t_{-12} $, $t_{01}\mapsto t_{02}$. 
For $\mC_2=(a_2<a_1)$ we have 
\[
\Wt^E_{\mC_2}(\One_{42}\cup \One_{21})=\\
\Shuffle
\left( A \cdot Z[T[B]] 
\frac{
\th(t_{-12}/t_{-11}\h)
\th(t_{-21}/t_{-12}\h)\th(t_{-11}/t_{02}\h)\th(t_{01}/t_{-12})}
{\th(t_{-11}/t_{-12})}
\right).
\]
Here $Z=\id$ and $T[\ ]$ is the substitutions $t_{-11}\mapsto t_{-12}$, $t_{01}\mapsto t_{02}$. In cohomology we obtain
\begin{align*}
\Wt^H_{\mC_1}(\One_{42}\cup \One_{21})= &
\Shuffle\left( 
\frac{(t_{-12}-t_{-11}+\h)(t_{-11}-t_{-21})(t_{01}-t_{-12})(t_{-11}-t_{02}+\h)}{(t_{-11}-t_{-12})}
\right),
\\
\Wt^H_{\mC_2}(\One_{42}\cup \One_{21})= &
\Shuffle\left( 
\frac{(t_{-12}-t_{-11}+\h)(t_{-21}-t_{-12}+\h)(t_{-11}-t_{02}+\h)(t_{01}-t_{-12})}{(t_{-11}-t_{-12})}
\right).
\end{align*}
\end{example}

\begin{example}
\label{ex:0101,2}
  Consider $\Ch(r=(0101),c=(2))$ and we keep the notation of Example~\ref{ex:0101,11}. Since $n=1$, the chamber is irrelevant, and for the only fixed point $f=\One_{21}\cup\One_{41}$ we obtain
  \[ 
  \Wt^*(f)=\Wt^*_{\Chamb_2}(\One_{42} \cup \One_{21})
  \quad
  \text{ from Example \ref{ex:0101,11}}.
  \]    
\end{example}

\subsection{The \texorpdfstring{$\W$}{W} function}
\label{subsec: the W function}
In this section we define the $\W_{\Chamb}(f)(t_{ki},a_i,z_i,\h)$ function that is obtained from the  
$\Wt_{\Chamb}(f)(t_{ki},z_i,\h)$ function by a substitution and a multiplication by certain factors all depending only on the charge vectors $c, r$ and the chamber $\Chamb$. In particular these slight `modifications' are formally the same for all fixed points on the same $\Ch(\DD)$ variety. 

\begin{definition}
    For D5 charge vector $c$ define 
    \[
    \varepsilon^*_c= e^*\left(\bigoplus_{k>0}\bigoplus_{j=1}^{c_k-1}\bigoplus_{i=1}^j \h^i\right)^{-1}.
    \]
    Explicitly, we have
    \[ \varepsilon_c^E=\prod_{k>0}\prod_{j=1}^{c_k-1}\prod_{i=1}^j \th(\h^i)^{-1}, 
    \qquad
    \varepsilon^K_c=\prod_{k>0}\prod_{j=1}^{c_k-1}\prod_{i=1}^j {\hata(\h^i)^{-1}},
    \qquad
    \varepsilon^H_c=\prod_{k>0}\prod_{j=1}^{c_k-1}\prod_{i=1}^j (i\h)^{-1}. 
    \]
\end{definition}

Recall the $\Tt$-equivariant $K$-theory class $T_{\D5}(d^+)$ from \S~\ref{sec:tautological bundles}. When expressed in terms of the variables $a_i$ and $\h$, using  
\begin{equation}\label{eq:tplus subs again}
\xi_k|_{f} = \bigoplus_{j>k} \bigoplus_{i=0}^{c_j-1} a_j\h^{-i},
\end{equation}  
cf. \eqref{eq:tplus subs}, we denote this bundle by $T'_{\D5}(c)$.  Each monomial in $T'_{\D5}(c)$ takes the form $a_i/a_j \cdot \h^s$. The $\Chamb_{\sigma}$-negative part of this Laurent polynomial consists of the terms satisfying $\sigma(i) < \sigma(j)$.  

\begin{definition}  
For a D5 charge vector $c$ and a chamber $\Chamb$, define $\eu^*_{c,\sigma}$ as the Euler class of the $\Chamb$-negative part of $T'_{\D5}(c)$, i.e.,  
\[
\eu^*_{c,\sigma} = e_-(T'_{\D5}(c)).
\]  
\end{definition}  

\begin{example}
 Let $c=(2,1,3)$, that is, $d^+=(6,4,3,0)$. The substitution~\eqref{eq:tplus subs again} is 
 \begin{multline*}
 t_{0,1\ldots 6}=\{a_1,a_1\h^{-1},a_2,a_3,a_3\h^{-1},a_3\h^{-2}\},\ 
 t_{1,1\ldots 4}=\{a_2,a_3,a_3\h^{-1},a_3\h^{-2}\},
 \\
 t_{2,1\ldots3}=\{a_3,a_3\h^{-1},a_3\h^{-2}\}.
 \end{multline*}
Performing this substitution in $T_{\D5}(d^+)$ yields
\[ 
T'_{\D5}(c)=
-\frac{a_1}{a_2}-\frac{a_2}{a_1}\h 
-\frac{a_1}{a_3}-\frac{a_3}{a_1}\h
-\frac{a_3}{a_1}-\frac{a_1}{a_3}\h
-\frac{a_3}{a_1}\h-\frac{a_1}{a_3}\h^2
-h^2-3\h-3-\h^{-1},
\]
and hence, for example, for $\sigma=id$ we obtain
\[
\eu^E_{c,id}=\frac{1}
{
\th\left( \frac{a_1}{a_2} \right) 
\th\left( \frac{a_1}{a_3} \right) 
\th\left( \frac{a_1}{a_3}\h \right)
\th\left( \frac{a_1}{a_3}\h^2 \right)
}. 
\]
\end{example}

\begin{example}
    Another example we leave to the reader to verify is
    \[
    \eu^H_{(3,3),id}=\frac{1}{(a_1-a_2-\h)(a_1-a_2+2\h)}.
    \]
\end{example}

Recall that the bow variety $X$ is a quotient of the torus $G=\prod_{\Xb} \GL(W_{\Xb})$. Let $G_{\NS5}\subset G$ be the subgroup given by 
\[
G_{\NS5}=\prod_{\Zb} \GL(W_{\Zb^+}).
\]
In other words, we only consider the general linear groups $\GL(W_{\Xb})$ associated with a D3 brane, which is left to a NS5 brane. Let $\fg_{\NS5}$ be its Lie algebra and consider 
\[
\tau_{r}^*=e^*(\hbar \fg_{\NS5})^{-1}.
\]
Notice that the set of dimensions $d_k$ for $k<0$ is determined by the charge vector $r$, hence the notation.
Explicitly, we have 
\[
\tau_{r}^E=\prod_{k<0}\prod_{i,j=1}^{d_k}\th\left(\frac{\hbar t_{ki}}{t_{kj}}\right)^{-1}
\quad 
\tau_{r}^K=\prod_{k<0}\prod_{i,j=1}^{d_k}{\hata\left(\frac{\hbar t_{ki}}{t_{kj}}\right)^{-1}}
\quad 
\tau_{r}^H=\prod_{k<0}\prod_{i,j=1}^{d_k}\left(\hbar + t_{ki}-t_{kj}\right)^{-1}.
\]

Formula \eqref{eq:tplus subs} names the set of $a_j\h^s$ expressions to which the set $\{t_{0i}\}$ specializes at every fixed point of $\Ch(\DD)$. That was sufficient for us when we substituted in $T_{\D5}(d^+)$ which is symmetric in the $t_{0i}$ variables. However, $\Wt$ functions are not, and hence a specific order of those expressions is necessary. 
\begin{definition} \label{def:610}
Let $c$ be a D5 charge vector, and $\sigma\in S_n$ a permutation. Consider the ordered list of monomials
\[
\underbrace{a_{\sigma(1)}, a_{\sigma(1)}\h^{-1},\ldots,a_{\sigma(1)}\h^{-c_\sigma(1)+1}}_{c_{\sigma(1)}}, \ldots,
\underbrace{a_{\sigma(n)}, a_{\sigma(n)}\h^{-1},\ldots,a_{\sigma(n)}\h^{-c_\sigma(n)+1}}_{c_{\sigma(n)}}.
\]
Substituting these expression in the $t_{0i}$ variables {\em in this order} will be denoted by $|(c,\sigma)$.
\end{definition}

\begin{definition}
    For a torus fixed point $f\in \Ch(r,c)^{\At}$ and a chamber $\Chamb=\Chamb_{\sigma}$ we define the function $\W$ as
    \[
    \W_{\Chamb}^*(f)=\varepsilon^*_{c} \cdot \tau^*_{r}  \cdot \eu^*_{c,\sigma} \cdot \Wt^*_{\Chamb}(f)|_{(c,\sigma)}.
    \]
\end{definition}

\begin{lemma}
   The classes $\varepsilon^*_{c}$ and $\eu^*_{c,\sigma}$ are trivial (i.e. equal to one) iff $c=\ul 1$.
\end{lemma}
\begin{proof}
    By definition $\varepsilon^*_{c}=1$ iff $c=\ul 1$. The analogous statement for $\eu^*_{c,\sigma}$ follows from a straightforward computation with \eqref{eq: D5 part tagent} and \eqref{eq:tplus subs again}.
\end{proof}
\begin{corollary}
\label{cor: main thm partial flag}
    If $X$ is the cotangent bundle of a partial flag variety, i.e. if $c=\ul 1$, then 
    \[
    \W_{\Chamb}^*(f)= \tau^*_{r}\cdot \Wt^*_{\Chamb}(f)|_{(c,\sigma)}.
    \]
\end{corollary}

\subsubsection{Examples for $\W$ functions}

\begin{example}
    Continuing Example~\ref{ex:TP1_Wtilde} for the fixed points of $\TsP^1$ we obtain
\begin{align*}
\W^E_{\mC_1}(f_1)
 & 
= \frac{\th(t_{-11}/a_1 \cdot z_2/z_1)}{\th(z_2/z_1\cdot \h^{-1})} 
\cdot \th(t_{-11}/a_{2} \cdot \h),
\\
\W^E_{\mC_1}(f_2) 
& 
=  \frac{\th(t_{-11}/a_2 \cdot z_2/z_1 \cdot \h)}{\th(z_2/z_1)}
\cdot
\th(a_1/t_{-11}),
\\
\W^E_{\mC_2}(f_1)
 & 
= \frac{\th(t_{-11}/a_2 \cdot z_2/z_1 \cdot \h)}{\th(z_2/z_1)} 
\cdot \th(a_2/t_{-11}),
\\
\W^E_{\mC_2}(f_2) 
& 
=  \frac{\th(t_{-11}/a_2 \cdot z_2/z_1)}{\th(z_2/z_1\cdot \h^{-1})}
\cdot
\th(t_{-11}/a_1\cdot \h).
\end{align*}
The cohomological versions are
\[
\W^H_{\Chamb_1}(f_1)=t_{-11}-a_2+\h,
\quad
\W^H_{\Chamb_1}(f_2)=a_1-t_{-11},
\quad
\W^H_{\Chamb_2}(f_1)=a_2-t_{-11},
\quad
\W^H_{\Chamb_2}(f_2)=t_{-11}-a_1+\h.
\]
\end{example}

More generally, for the bow variety representing $\TsP^{n-1}$ in \S~\ref{sec:Stab,TPn} we obtain that the $\W^E_{\Chamb_{id}}(f_k)$ function is exactly the right hand side of~\eqref{eq:Stab TP^n1}, thereby suggesting the validity of Theorem~\ref{thm:main} stated below.

\begin{example} \label{ex:cannotthinkofagoodname}
Continuing Examples~\ref{ex:0101,11} and~\ref{ex:0101,2} we obtain
\[
\W^*_{\Chamb_1}(\One_{42}\cup \One_{21})=\frac{1}{\th(\h)^4}\Wt^*_{\Chamb_1}(\One_{42}\cup \One_{21})|_{t_{01}=a_1,t_{02}=a_2},
\]
\[
\W^*_{\Chamb_2}(\One_{42}\cup \One_{21})=\frac{1}{\th(\h)^4}\Wt^*_{\Chamb_2}(\One_{42}\cup \One_{21})|_{t_{01}=a_2,t_{02}=a_1},
\]
\[
\W^*(\One_{41}\cup \One_{21})=\frac{1}{\th(\h)^5}\cdot \Wt^*(\One_{41}\cup \One_{21})|_{t_{01}=a_1,t_{02}=a_1/\h}.
\]
(In the last line, in case of cohomology, $t_{02}=a_1/\h$ should be replaced with $t_{02}=a_1-\h$.)
\end{example}

\subsection{\texorpdfstring{$\W$}{W} functions express stable envelopes}
\label{sec: main section}
We are ready to state our main theorem. 
\begin{theorem} \label{thm:main}
    For a torus fixed point $f \in \Ch(\DD)^{\At}$ and chamber $\Chamb$, the function $\W^*_{\Chamb}(f)$ represents the stable envelope $\Stab^*_{\Chamb}(f)$.
\end{theorem}

There are two subtleties to discuss about what the word `represents' means in the theorem.

Recall that $\Stab^*_{\Chamb}(f)$ is a cohomology/K theory/elliptic cohomology class that satisfies normalization and support, as well as an extra condition. The extra condition in cohomology and K theory is about the degree of fixed point restrictions, and in elliptic cohomology it is about being a section of the correct line bundle over a product of elliptic curves.

The first subtlety to discuss is how these functions represent classes at all. For example, in cohomology, a {\em polynomial} in $t_{ki}, a_i, \h$ clearly represents a cohomology class. However, $\W$ is a {\em rational function}, even after simplification. The fact is that all of the torus fixed point restrictions of $\W$ are polynomials, and we mean this tuple of fixed point restrictions to be the cohomology class. Here is the smallest illustration: consider $\Ch(r=(1,1),c=(2))$ which is in fact a one-point space $\{f\}$. Our name for $f$ is $\One_{11}\cup\One_{21}$. We obtain $W^H_{\Chamb}(f)=(t_{-11}-t_{02}+\h)/\h$, and the restriction map to $f$ is given by $t_{-11}\mapsto a_1-\h$, $t_{02}\mapsto a_1-\h$. Thus $W^H_{\Chamb}(f)$ restricted to $f$ is 1, which is the cohomological stable envelope of a one-point space. In more complicated examples the denominators also contain $t_{ki}$ variables, but the phenomenon is that same, cf. \cite[\S5.2]{RTV}, \cite[\S7.7]{rimanyi2020bow}.

The second subtlety concerns the restriction maps in general. Let us illustrate it with the example
\begin{multline*}
\W^H(\One_{41}\cup \One_{21})=
\frac{(a_1-t_{-12})(t_{-11}-a_1+2\h)(t_{-21}-t_{-12}+\h)}{(t_{-11}-t_{-12}+\h)(t_{-11}-t_{-12})\h}+\\
\frac{(a_1-t_{-11})(t_{-12}-a_1+2\h)(t_{-21}-t_{-11}+\h)}{(t_{-12}-t_{-11}+\h)(t_{-12}-t_{-11})\h},
\end{multline*}
cf. Example~\ref{ex:cannotthinkofagoodname} (the ambient bow variety is a one-point space, so we expect $\Stab^H=1$). The restriction map is given by 
\[
t_{-21}=a_1-2\h, \qquad t_{-11}=a_1-\h, \qquad t_{-12}=a_1-2\h.
\]
However, this substitution is undefined, because even after bringing the two fractions to common denominator and canceling all the common factors, the substitution gives $0/0$. However, term-by-term substitution makes sense: the substitution in the first term results 1. The substitution in the second term, while undefined, is of the type $0^2/0$. Our definition of substitution is term-by-term, and any time a term has higher order vanishing in the numerator than in the denominator, that term's contribution is 0. Hence, the substitution of the example above results $1+0=1$, as expected.

\subsection{Wheel conditions for W functions}
For a dimension vector $d$, let $\Yt_d^{H}\subseteq \Ht_d^{H}$ be the subspace of functions $f\in \Ht_d^{H}$ satisfying the following ``wheel conditions'':
\begin{align*}
    & f|_{t_{-k,a}=t_{-k+1,b}=t_{-k,c}+\h} =0\\
    & f|_{t_{-k,a}=t_{-k-1,b}= t_{-k,c}-\hbar} =0
\end{align*}
for all $k<0$ and $a,b,c$ with $a\neq c$. Notice that the wheel conditions are trivial if $d_i\leq 1$ for all $i< 0$. 

Similarly, we let $\Yt_d^{K}\subseteq \Ht_d^{K}$ (resp. $\Yt_d^{E}\subseteq \Ht_d^{E}$) be the subspace of functions $f\in \Ht_d^{K}$ (resp. $f\in \Ht_d^{E}$) satisfying the following of wheel conditions:
\begin{align*}
    & f|_{t_{-k,a}=t_{-k+1,b}=\hbar t_{-k,c}} =0\\
    & f|_{t_{-k,a}=t_{-k-1,b}= \hbar^{-1}t_{-k,c}} =0
\end{align*}
for all $k<0$ and $a,b,c$ with $a\neq c$.
Notice that the K-theoretic and elliptic wheel conditions are simply the multiplicative version of the cohomological wheel conditions.

The next proposition theorem is due to Negut \cite[Prop. 2.10]{negut2022shuffle}. Although the statement and its proof in \emph{loc.cit} are given in K-theory, they directly transfer to singular and elliptic cohomology.
\begin{proposition}
\label{prop: wheel conditions ans shuffle product}
    Set $\Yt^{\ast}:=\oplus_{d} \Yt_d^{\ast}$. The shuffle product $\star: \Ht_{d'}^* \otimes \Ht_{d''}^* \to 
\Ht_{d'+d''}^*$ preserves the wheel conditions. Hence, $\Yt^{\ast}$ is a subalgebra of $\Ht^{\ast}$.
\end{proposition}
The $\W$ weight functions for bow varieties are inductively constructed with the shuffle product starting from one-tie weight functions (cf. \ref{sec:one tie functions}). Since the latter trivially satisfy the wheel conditions, Proposition \ref{prop: wheel conditions ans shuffle product} implies the following theorem.

\begin{theorem}
    Cohomological, K-theoretic and elliptic weight functions of bow varieties satisfy the wheel conditions.
\end{theorem}
\begin{remark}
    From the geometric perspective, the wheel conditions encode the fact that a bow variety is a symplectic quotient, i.e. it arises as the quotient of the zero locus of 
    \[
    \mu_<: \MM\to \bigoplus_{i<0} \NN_{\Xb_i}.
    \]
    which is the moment map for the action of the group $\prod_{i<0} \GL(d_i)$ on $\MM$, cf. \S\ref{sec:def of bow variety}.
    In fact, each wheel condition corresponds to a potentially non-zero entry of the moment map $\mu_<$, see  \cite[Prop. 2.10]{negut2022shuffle} for more details.
\end{remark}

\section{Proof of the main theorem}
\label{sec:proof}

\subsection{The case of a partial flag variety}

Let $X=X(\DD)$ be a bow variety with $n$ D5 branes and $m$ NS5 branes. Assume that all D5 charges of $\DD$ are equal to one, that is, $c_i=1$ for all $i=1,\dots, n$. By Proposition \ref{prop: charge one bows}, it follows that $X(\DD)$ is isomorphic to the cotangent bundle of a partial flag variety, i.e. to a type $A_m$ quiver variety. In this section, we prove the theorem in this specific setting.

\begin{proposition}
\label{prop: main theorem partial flag}
    Theorem \ref{thm:main} holds whenever $c_i=1$ for all $i=1,\dots, n$.
\end{proposition}

In fact, the theorem in this special case is already known \cite{Rimanyi_2019full, botta2021shuffle, botta2023framed}. Yet, we briefly outline a conceptual proof here. For simplicity, we first focus on the cohomological case and then discuss what changes in K-theory and elliptic cohomology.

\begin{proof}[Proof in cohomology]
    Consider the quiver variety description of $X$. Then $X=X(\dd,\ww)$ can be written as 
\[
X=\mu_{\dd, \ww}^{-1}(0)^{ss}/\GL(\dd),
\]
where $Q=A_m$, $\ww=(0,\dots, 0, n)$, and $\mu_{\dd, \ww}^{-1}: \Rep_{\overline Q}(\dd, \ww)\to \gl(\dd)$ is the moment map. As a consequence, we have open and closed embeddings
\begin{equation}
\label{eq: quiver rep stable vs unstable}
    X(\dd, \ww)=\mu_{\dd, \ww}^{-1}(0)^{ss}/\GL(\dd)\hookrightarrow [\mu_{\dd, \ww}^{-1}(0)/\GL(\dd)]\hookrightarrow [\Rep_{\overline Q}(\dd, \ww)/\GL(\dd)]=: \FM(\dd,\ww)
\end{equation}
To produce an explicit tautological representative of the elliptic stable envelope, one can exploit the \emph{non-abelian stable envelope} map \cite{Aganagic:2017gsx, okounkov2020inductiveII, botta2023framed}. This map provides a canonical extension of a cohomology class on $X(\dd,\ww)=\mu_{\dd, \ww}^{-1}(0)^{ss}/\GL(\dd)$ to a class on $ \FM(\dd, \ww)$ (with support on $\mu_{\dd, \ww}^{-1}(0)/\GL(\dd)$). Formally, the nonabelian stable envelope is a canonical morphism 
\[
\Psi_{\dd,\ww}: H_{T_{\ww}}(X(\dd,\ww))\to  H_{T_{\ww}}(\FM(\dd,\ww))
\]
satisfying $\iota^{\oast}\circ \Psi_{\dd,\ww}=e(\hbar \gl(\dd))\cdot$, where the map $i$ is the composition \eqref{eq: quiver rep stable vs unstable}. Notice that, in the current notation, the class $e(\hbar \gl(\dd))$ is nothing but the class $\varepsilon_r^H$ introduced in \S~\ref{subsec: the W function}.

For us, the relevance of the non-abelian stable envelope lies in the following proposition.
\begin{proposition}[\cite{botta2023framed, BD}]
\label{prop: NS stab and CoHA}
 Consider a decomposition $\ww=\ww'+\ww''$ and the corresponding two dimensional torus $\Cs_{a_1}\times \Cs_{a_2}\subset A_{\ww}$ acting with $a_1$ on $\C^{\ww'}$ and $a_2$ on $\C^{\ww''}$. This torus admits two chambers $\{a'<a''\}$ and $\{a''<a'\}$.
    Then the diagram 
    \[
    \begin{tikzcd}
        H_{T_{\ww'}}(X(\dd',\ww'))\otimes_{\Bbbk^H} H_{T_{\ww''}}(X(\dd'',\ww''))\arrow[rr, "\Stab_{\{a_1<a_2\}}"]\arrow[d, swap, "\Psi_{\dd',\ww'}\otimes \Psi_{\dd',\ww'}"] & &H_{T_{\ww}}(X(\dd,\ww))\arrow[d, "\Psi_{\dd,\ww}"] \\
         H_{T_{\ww'}}(\FM(\dd',\ww'))\otimes_{\Bbbk^H} H_{T_{\ww''}}(\FM(\dd'',\ww''))\arrow[rr, "\vmult"] & & H_{T_{\ww}}(\FM(\dd,\ww))
    \end{tikzcd}
    \]
    is commutative. The analog statement for $\Stab_{\{a>a''\}}$ is obtained by replacing $\vmult$ with $\vmult^{\text{op}}$.
\end{proposition}

We are now ready to produce tautological presentations of the stable envelopes $\Stab_{\Chamb}(f)$ for the maximal torus $T_{\ww}$ acting on $X(\dd,\ww)$. With respect to the torus action, fixed points of the fixed points are of the form
\[
f=X(\dd^{(1)}, \delta_m)\times \dots \times X(\dd^{(n)}, \delta_m),
\]
where\footnote{Notice that all $X(\dd^{(k)}, \delta_n$ are zero dimensional.}
\[
\delta_m=(0, \dots, 0, 1)
\qquad 
\dd^{(k)}=(0,\dots, 0, 1, \dots, 1).
\]
In the language of \S~\ref{sec:fixed points}, we can equivalently write $f=\One_{i_11}\cup \One_{i_2 2}\cup\dots \cup \One_{i_n n}$ where $i_k-1$ is the number of $1$s in $\dd^{(k)}=(0,\dots, 0, 1, \dots, 1)$.

Consider the composition 
\[
H_{T_{\ww}}(X(\dd,\ww)^{T_{\ww}})\xrightarrow{\Stab_{\Chamb}} H_{T_{\ww}}(X(\dd,\ww)^{T_{\ww}})\xrightarrow{\Psi_{\dd,\ww}} H_{T_{\ww}}(\FM(\dd,\ww)).
\]
and let $Y_{\Chamb}(f)\in H_{T_{\ww}}(\FM(\dd,\ww))$ be the image of the the fundamental class of $f$. Combining Proposition \ref{prop: NS stab and CoHA}, associativity of stable envelopes \ref{lma:triangle lemma}, and associativity of the CoHA multiplication, it follows that 
\[
Y_{\Chamb}(f)= \Psi_{\dd^{(\sigma(1))}, \delta_m}(1)\star \dots \star\Psi_{\dd^{(\sigma(n))}, \delta_m}(1) .
\]
Therefore, to deduce Proposition \ref{prop: main theorem partial flag}, it suffices to show that 
\[
\Psi_{\dd^{(k)}, \delta_m}(1)=\hbar^{i_k-1}=\Wt(\One_{i_k, k}).
\]
But this easily follows from the definition of $\Psi$, cf. \cite{okounkov2020inductiveII, botta2023framed}.
\end{proof}

\begin{proof}[Proof in K-theory]
    In K-theory, the proof is essentially same. In particular, the K-theoretic analog of Proposition \ref{prop: NS stab and CoHA} replaces cohomological stable envelopes with K-theoretic stable envelopes and the cohomological shuffle algebra with the K-theoretic shuffle algebra. Notice that although Proposition \ref{prop: NS stab and CoHA} is proved in \cite{botta2023framed, BD} in cohomology, the argument of \cite{botta2023framed} transfers to K-theory without modifications. The only difference in K-theory is that the K-theoretic weight functions $\Wt^K(\One_{i_k, k})$ are non-constant rational functions, cf. \ref{sec:Wt}. Rather than showing that 
    \[
    \Psi^K_{\dd^{(k)}, \delta_m}(1)=\Wt^K(\One_{i_k, k}),
    \]
    we directly move to the elliptic setting and show the analog statement there. The equation above will then follow by degenerating from elliptic cohomology to K-theory.
\end{proof}

\noindent\emph{Proof in elliptic cohomology.}
The elliptic proof is also similar but deserves further attention because of the shifts of the dynamical parameters. We first review the definition of elliptic nonabelian stable envelope following \cite{okounkov2020inductiveII}. Consider the morphisms
\[
\begin{tikzcd}
    \Ell_{T_{\ww}\times Z}(X(\dd,\ww))\arrow[rr, "\hbox{\j}"]\arrow[dr, swap, "p_1"] & & \Ell_{T_{\ww}\times Z}(\FM(\dd,\ww))\arrow[dl, "p_2"]\\
     & \Ell_{T_{\ww}\times Z}(\pt)&
\end{tikzcd}
\]
induced by \eqref{eq: quiver rep stable vs unstable}. The line bundle $\Sh{L}_{X(\dd,\ww)}$ on $\Ell_{T_{\ww}}(X(\dd,\ww))$—see \eqref{eq: line bundles stab bow} for the definition— naturally extends to $\Ell_{T_{\ww}}(\FM(\dd,\ww))$, in the sense that there exists a line bundle $\Sh{L}^0_{\FM(\dd,\ww)}$ such that $\hbox{\j}^*\Sh{L}_{\FM(\dd,\ww)}=\Sh{L}_{X(\dd,\ww)}$. In fact, this extension can be explicitly described since $\Sh{L}_{X(\dd,\ww)}$ is defined in terms of tautological line bundles on $X$ that are pulled back from $\FM(\dd,\ww)$ along \eqref{eq: quiver rep stable vs unstable}. 

Set $\Sh{L}_{\FM(\dd,\ww)}=\Sh{L}^0_{\FM(\dd,\ww)}\otimes \th(\hbar\gl(\dd))$. The elliptic non-abelian stable envelope is a canonical morphism 
\[
\Psi_{\dd,\ww}: (p_1)_*\Sh{L}^{\infty}_{X(\dd,\ww)}\xrightarrow{\Psi_{\dd,\ww}} (p_2)_*(\Sh{L}^{\infty}_{\FM(\dd,\ww)})
\]
satisfying $i^{\oast} \circ \Psi_{\dd,\ww}=\th(\hbar \gl(\dd))$. The notation $\Sh{L}^{\infty}$ is a shortcut for $\Sh{L}(\infty\Delta)$, cf. \S~\ref{sec: E stab}. 

To state the analog of Proposition \ref{prop: NS stab and CoHA}, it is convenient to define the following twisted tensor product: given two quasicoherent sheaves $\Sh{F'}$ (resp. $\Sh{F''}$) on $\Ell_{T_{\ww'}}$ (resp $\Ell_{T_{\ww';}}$), we define 
\[
\Sh{F'}\odot \Sh{F''}:= \left(\Sh{F'}\boxtimes (\tau')^* \Sh{F''}\right)\otimes \Sh{G}_{\dd',\dd'', \ww',\ww''}
\]
where
\begin{itemize}
    \item the exterior tensor product is on 
    \begin{equation*}
        \Ell_{T_{\ww'}\times Z}\times_{\Ell_{{\Cs_{\hbar}}\times Z}}\Ell_{T_{\ww''}\times Z}= \Ell_{T_{\ww}\times Z}
    \end{equation*}
    \item $\Sh{G}_{\dd,\dd'',\ww'\ww''}$ is a line bundle on $\Ell_{T_{\ww}\times Z}$, uniquely determined by Prop. \ref{prop: compatibility elliptic stab coha} below. 
    \item $\tau'$ is the shift operator introduced in \eqref{eq: dynamical shift operator}.
\end{itemize}

With this notation, the analog of Proposition \ref{prop: NS stab and CoHA} is the following:
\begin{proposition}
\label{prop: compatibility elliptic stab coha}
Consider the same set-up of Proposition \ref{prop: NS stab and CoHA}. Let $\vmult^0$ be the morphism from equation \eqref{eq: Eha mult sheaf}. Then there exists a unique line line bundle $\Sh{G}_{\dd,\dd'',\ww'\ww''}$ making the maps in the diagram
    \[
    \begin{tikzcd}
    (p'_1)_*\Sh{L}^{\infty}_{X(\dd',\ww')}\odot (p''_1)_*\Sh{L}^{\infty}_{X(\dd'',\ww'')} \arrow[d, swap, "\Psi_{\dd',\ww'}\otimes\Psi_{\dd'',\ww''}"] \arrow[rr, "\Stab_{\{a_1<a_2\}}"]& & (p_1)_*\Sh{L}^{\infty}_{X(\dd',\ww')}\arrow[d, "\Phi_{\dd,\ww}"]
    \\
    (p'_2)_*\Sh{L}^{\infty}_{\FM(\dd',\ww')}\odot (p''_2)_*\Sh{L}^{\infty}_{\FM(\dd'',\ww'')}\arrow[rr, "\vmult^0"] &  &(p_2)_*\Sh{L}^{\infty}_{\FM(\dd,\ww)}
    \end{tikzcd}
    \]
    well defined. Moreover, for this choice of $\Sh{G}_{\dd,\dd'',\ww'\ww''}$ the diagram is commutative.  The analog statement for $\Stab_{\{a>a''\}}$ is obtained by replacing $\vmult^0$ with $(\vmult^0)^{\text{op}}$.
\end{proposition}

\begin{proof}
    The proof of this proposition is a straightforward adaptation of the cohomological argument in \cite{botta2023framed}. In \emph{loc. cit.} the proof is first reduced to the abelian case via abelianization of cohomological stable envelopes \cite{Shenfeld}, and then proved for hypertoric varieties via an explicit computation. Carrying out the same argument via elliptic abelianization \cite{aganagic2016elliptic}, the proposition follows. The only extra piece of complexity in the elliptic setting is due to the well-defniteness of the horizontal morphism in the diagrams. For instance, one needs to show that the line bundles in the top horizontal morphism match those in \eqref{elliptic stable enveliopes map}. For partial flag varieties (and more generally for bow varieties), this is checked in \cite[\S~5.12]{BR23}. For general quiver varieties, this is checked in the appendix of \cite{botta2021shuffle}.
\end{proof}

\subsection{The general case}

In this section we deduce the proof of Theorem \ref{thm:main} from the special case of the cotangent bundle of a partial flag variety. In view of Proposition \ref{prop: main theorem partial flag}, Lemma \ref{lemma: D5 res bundles pullback}, and Theorem \ref{thm:main} is reduced to the following:
\begin{proposition}
    Assume that Theorem \ref{thm:main} holds whenever $c_i=1$ for all $i=1,\dots, m$. Then it holds for an arbitrary bow variety $X$. 
\end{proposition}

\begin{proof}

Let $\wt X$ be the maximal resolution of $X$ in the sense of Corollary \ref{corollary any bow can be embedded in a flag variety}.  Notice that $X$ and $\wt X$ share the same NS5 charge vector $r$. On the other hand, the charge vector of $\wt X$ is the unit vector $\ul 1=(1,1,\dots)$. Then $\wt X$ is the cotangent bundle of the full flag variety and the proposition holds by the assumption for $\wt X$. Fix an arbitrary chamber $\Chamb_\sigma$ for $X$ and a fixed point $f=\cup_{kl} \One_{kl}$. Let $\tilde f_{\sharp}\in \wt X^{\wt \At}$ be the distinguished resolution of $f$. It $\tilde f_{\sharp}$ is of the form
\begin{multline*}
    \underbrace{\One_{k_1 1}\cup\dots\cup \One_{k_{c_1},c_1}}_{c_1-\text{ties}}\cup \underbrace{\One_{k_{c_1+1} c_1+1}\cup\dots\cup \One_{k_{c_1+c_2},c_1+c_2}}_{c_2-\text{ties}}\cup \dots
    \\
    \cup\underbrace{\One_{k_{c_1+\dots c_{n-1}+1}, c_1+\dots+c_{n-1}+1}\cup\dots\cup \One_{k_{c_1+\dots+c_n},c_1+\dots+c_n}}_{c_n-\text{ties}},
\end{multline*}
where each bracket corresponds to those ties of $\tilde f_{\sharp}$ resolving the ties $f$ connected to a given D5 brane. By Corollary \ref{corollary nicest one-term D5 resolution full} and Proposition \ref{prop: main theorem partial flag}, we can write 
\begin{equation}
   \label{eq: proof main thm0} \underbrace{\prod_{k>0}\prod_{j=1}^{c_k-1}\prod_{i=1}^j e^*(\h^i)}_{=(\varepsilon_c^*)^{-1}} e^*(N^-_j) \Stab^{X}_{\Chamb}(f)
    = j^{\oast}\varphi^{\oast}\Stab^{\wt X}_{\Chamb_{\wt \sigma}}(\tilde f_\sharp).
\end{equation}
On the other hand, by the inductive hypothesis and Corollary \ref{cor: main thm partial flag}, we have 
\begin{equation}
    \label{eq: proof main thm1}
    \Stab^{\wt X}_{\wt \Chamb_{\sigma}}(\tilde f_\sharp)\simeq W_{\wt \Chamb_{\sigma}}(\tilde f_\sharp)= \tau_{r}^* \cdot \Wt_{\Chamb_{\wt \sigma}}(\tilde f_\sharp)|_{(\ul 1, \wt \sigma)}
\end{equation}
where the symbol $\simeq$ indicates that the left hand side is represented by the right hand side in the sense of \S~\ref{sec: main section}.
Moreover, because of the specific choice of chamber $ \Chamb_{\wt \sigma}$ prescribed by Corollary~\ref{corollary nicest one-term D5 resolution full} we have 
\begin{multline*}
    \Wt_{\Chamb_{\wt \sigma}}(\tilde f_\sharp)=\prod_{i=0}^{c_{\sigma(1)}-1} \Wt(\One_{k_{c_{\sigma(1)}-i},c_{\sigma_1}-i})\ast \prod_{i=0}^{c_{\sigma(2)}-1} 
    \Wt(\One_{k_{c_{\sigma(1)}+c_{\sigma(2)}-i},c_{\sigma(1)}+c_{\sigma(2)}-i})\ast\dots
    \\
    \ast \prod_{i=0}^{c_n-1} \Wt(\One_{k_{c_{\sigma(1)}+\dots+ c_{\sigma(n)}-i}, c_{\sigma(1)}+\dots+c_{\sigma(n)}-i}).
\end{multline*}
But from Definition \ref{def: W tilde} and the equation above it follows that  
\begin{equation}
    \label{eq: proof main thm2}
    \Wt_{\Chamb_{\wt \sigma}}(\tilde f_\sharp)=\Wt_{\Chamb_{\sigma}}(f).
\end{equation}
Moreover, from Definition \ref{def:610} and Remark \ref{rem: specialization maximal resolution}, it follows that for any function $f(t_{0i})$ in the variables $\{t_{0i}\}_{i=1,\dots, d_0}$ we have 
\begin{equation}
    \label{eq: proof main thm3}
    \varphi^{\oast}\left(f(t_{0i})|_{(\ul 1, \tilde \sigma)}\right)=f(t_{0i})|_{(c,  \sigma)}.
\end{equation}
Therefore, combining \eqref{eq: proof main thm0}, \eqref{eq: proof main thm1}, \eqref{eq: proof main thm2} , and \eqref{eq: proof main thm3}, we get 
\[
\Stab^{X}_{\Chamb}(f)\simeq \varepsilon^*_c \cdot \tau^*_r \cdot (e^*(N_j^-))^{-1} \cdot \Wt_{\Chamb_{\sigma}}(f)|_{(c, \sigma)}.
\]
Thus, to complete the proof, it suffices to show that $(e^*(N_j^-))^{-1} =\eu^*_{c,\sigma}$. By \eqref{eq:tangent bundle}, Lemma \ref{lma: D5 part tangent is trivial for flags} and the fact that $j^*(T_{\NS5} \wt X(d^-))=T_{\NS5} X(d^-)$ we get 
\[
N_j= (T_{\NS5} X(d^-)+0)-(T_{\NS5} X(d^-)+T_{\D5} X(d^+))=-T_{\D5} X(d^+),
\]
so the sought after equality follows by taking the repelling part of the equation above. The proof follows.

\end{proof}

\section{Application: \texorpdfstring{$\th$}{theta} function identities from 3d mirror symmetry}
\label{sec:mirror}

Recall that torus fixed points on bow varieties can be encoded by tie diagrams or BCTs. 

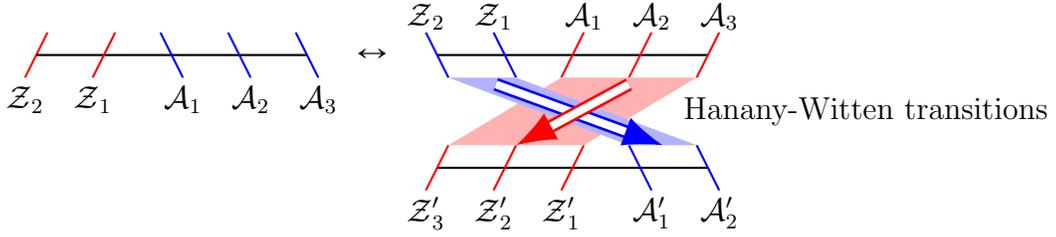
\begin{figure}
\[
\begin{tikzpicture}[scale=.3,baseline=6]
\draw [thick,red] (0.5,0) --(1.5,2); 
\draw[thick] (1,1)-- (13,1);
\draw [thick,red](3.5,0) --(4.5,2);  
\draw [thick,blue](7.5,0) -- (6.5,2);  
\draw[thick,blue] (10.5,0) -- (9.5,2);  
\draw [thick,blue](13.5,0) -- (12.5,2);   
\node at (.5,-1) {$\Zb_2$};
\node at (3.5,-1) {$\Zb_1$};
\node at (7.5,-1) {$\Ab_1$};
\node at (10.5,-1) {$\Ab_2$};
\node at (13.5,-1) {$\Ab_3$};
\end{tikzpicture}
\leftrightarrow
\begin{tikzpicture}[scale=.3,baseline=6]
\draw [thick,blue] (1.5,0) --(0.5,2); 
\draw[thick] (1,1)-- (13,1);
\draw [thick,blue](4.5,0) --(3.5,2);  
\draw [thick,red](6.5,0) -- (7.5,2);  
\draw[thick,red] (9.5,0) -- (10.5,2);  
\draw [thick,red](12.5,0) -- (13.5,2);   
\node at (.5,2.7) {$\Zb_2$};
\node at (3.5,2.7) {$\Zb_1$};
\node at (7.5,2.7) {$\Ab_1$};
\node at (10.5,2.7) {$\Ab_2$};
\node at (13.5,2.7) {$\Ab_3$};

\draw[thick] (1,-4)-- (13,-4);
\draw [thick,red] (0.5,-5) --(1.5,-3); 
\draw [thick,red](3.5,-5) --(4.5,-3);  
\draw [thick,red](6.5,-5) -- (7.5,-3);  
\draw[thick,blue] (10.5,-5) -- (9.5,-3);  
\draw [thick,blue](13.5,-5) -- (12.5,-3);   
\node at (.5,-6) {$\Zb'_3$};
\node at (3.5,-6) {$\Zb'_2$};
\node at (6.5,-6) {$\Zb'_1$};
\node at (10.5,-6) {$\Ab'_1$};
\node at (13.5,-6) {$\Ab'_2$};

\fill [fill=red!30]
     (6.5,0) -- (12.5,0) -- (7.5,-3) -- (1.5,-3) -- (6.5,0) ;
\fill [fill=blue!30]
     (1.5,0) -- (4.5,0) -- (12.5,-3) -- (9.5,-3) -- (1.5,0) ;
\draw [blue,line width=1pt, double distance=3pt,
            arrows = {-Latex[length=0pt 3 0]}] (3.6,-.3) -- (11,-3);
\draw [red,line width=1pt, double distance=3pt,
             arrows = {-Latex[length=0pt 3 0]}] (9.5,-.3) -- (4.5,-3);
\node at (20,-1.5) {Hanany-Witten transitions};
\end{tikzpicture}
\]
\caption{The Hanany-Witten interpretation of Definition~\ref{def:mirror fixed point}.}
\label{fig:HW}
\end{figure} 

\begin{definition} \label{def:mirror fixed point}
    Let $f$ be a fixed point on a bow variety whose associated BCT is the matrix $T\in \{0,1\}^{m\times n}$. Define the mirror matrix $T^!\in \{0,1\}^{n\times m}$ by 
    \[
    T^!_{i,j}=1-T_{n+1-j,m+1-i}.
    \] 
    That is, $T^!$ is obtained by reflecting $T$ across the NE-SW diagonal and negating (0 $\leftrightarrow$ 1) the entries. Let the associated torus fixed point (on a different bow variety) be called $f^!$. 
\end{definition}

The involution $f\leftrightarrow f^!$ is a bijection between the torus fixed points of $\Ch(r,c)$ and $\Ch(r^!,c^!)$, where
\[
r^!_i=m-c_{n+1-i}, \qquad c^!_i=n-r_{m+1-i}.
\]
These bow varieties are also called mirror duals of each other: $\Ch(r,c)^!=\Ch(r^!,c^!)$. Mirror dual bow varieties are, in general, of different dimensions, they have different rank torus actions, but there is a combinatorial bijection between their torus fixed points. 

The convention on the indices in the definition above is best explained by tie diagrams and Hanany-Witten transitions \cite{HW}, see Figure~\ref{fig:HW}. The mirror of a fixed point encoded by a tie diagram should be obtained by reflecting the diagram, both fivebranes and ties, across a horizontal axis. In this way NS5 branes turn to D5 branes, and vice versa. According to our convention, the NS5 branes need to be on the left, so we need to pass all the NS5 branes through all the D5 branes. Moving different kinds of fivebranes through each other is called {\em Hanany-Witten transition} (HW). Its main feature is that a pair of fivebranes is connected by a tie before a HW transition if and only if they are {\em not} connected after it. As \cite{HW} puts it: \texttt{``When an NS fivebrane passes through a D fivebrane a threebrane connecting them is created''}. This explains the 0-1 swap in the definition. For more details see \cite{Nakajima_Takayama, rimanyi2020bow, BR23}.

\begin{example}
The mirror dual of the 6-dimensional $\Ch(r=(1,1,2,1),c=(2,2,1))$ is the 12-dimensional $\Ch(r=(3,2,2),c=(2,1,2,2))$. A pair of corresponding fixed points is:
\[
\left(
\begin{tikzpicture}[baseline=7,scale=.2]
\begin{scope}[yshift=0cm]
\draw [thick,red] (0.5,0) --(1.5,2); 
\draw[thick] (1,1)--(2.5,1) -- (19,1);
\draw [thick,red](3.5,0) --(4.5,2);  
\draw [thick](4.5,1)--(5.5,1) -- (6.5,1);
\draw [thick,red](6.5,0) -- (7.5,2);  
\draw [thick](7.5,1) --(8.5,1) -- (9.5,1); 
\draw[thick,red] (9.5,0) -- (10.5,2);  
\draw[thick] (10.5,1) --(11.5,1) -- (12.5,1); 
\draw [thick,blue](13.5,0) -- (12.5,2);   
\draw [thick](13.5,1) --(14.5,1) -- (15.5,1);
\draw[thick,blue] (16.5,0) -- (15.5,2);  
\draw [thick](16.5,1) --(17.5,1)  -- (18.5,1);  
\draw [thick,blue](19.5,0) -- (18.5,2);  
\draw [dashed, black](7.5,2.2) to [out=25,in=155] (18.5,2.2);
\draw [dashed, black](10.5,2.2) to [out=25,in=155] (15.5,2.2);
\draw [dashed, black](4.5,2.2) to [out=35,in=145] (12.5,2.2);
\draw [dashed, black](4.5,2.2) to [out=45,in=135] (15.5,2.2);
\draw [dashed, black](1.5,2.2) to [out=45,in=135] (12.5,2.2);
\end{scope}
\end{tikzpicture}
\right)^!=
\begin{pmatrix}
    0 & 1 & 0 \\
    0 & 0 & 1 \\
    1 & 1 & 0 \\
    1 & 0 & 0 
\end{pmatrix}^!=
\begin{pmatrix}
    1 & 1 & 0 & 1 \\
    1 & 0 & 1 & 0\\
    0 & 0 & 1 & 1 
\end{pmatrix}
=
\begin{tikzpicture}[baseline=7,scale=.2]
\begin{scope}[yshift=0cm]
\draw [thick,red] (0.5,0) --(1.5,2); 
\draw[thick] (1,1)--(2.5,1) -- (19,1);
\draw [thick,red](3.5,0) --(4.5,2);  
\draw [thick](4.5,1)--(5.5,1) -- (6.5,1);
\draw [thick,red](6.5,0) -- (7.5,2);  
\draw [thick](7.5,1) --(8.5,1) -- (9.5,1); 
\draw[thick,blue] (10.5,0) -- (9.5,2);  
\draw[thick] (10.5,1) --(11.5,1) -- (12.5,1); 
\draw [thick,blue](13.5,0) -- (12.5,2);   
\draw [thick](13.5,1) --(14.5,1) -- (15.5,1);
\draw[thick,blue] (16.5,0) -- (15.5,2);  
\draw [thick](16.5,1) --(17.5,1) -- (18.5,1);  
\draw [thick,blue](19.5,0) -- (18.5,2);  
\draw [dashed, black](7.5,2.2) to [out=39,in=155] (18.5,2.2);
\draw [dashed, black](7.5,2.2) to [out=32,in=155] (12.5,2.2);
\draw [dashed, black](7.5,2.2) to [out=25,in=155] (9.5,2.2);
\draw [dashed, black](1.5,2.2) to [out=35,in=135] (15.5,2.2);
\draw [dashed, black](1.5,2.2) to [out=45,in=145] (18.5,2.2);
\draw [dashed, black](4.5,2.2) to [out=35,in=145] (9.5,2.2);
\draw [dashed, black](4.5,2.2) to [out=45,in=145] (15.5,2.2);
\end{scope}
\end{tikzpicture}.
\]
\end{example}

\begin{theorem}[3d mirror symmetry for elliptic stable envelopes \cite{BR23}] 
\label{thm:mirror for Wt}
Let $f,g$ be torus fixed points on $\Ch(\DD)$ with $m$ NS5 branes and $n$ D5 branes. Then we have \begin{equation}\label{eq:mirror symmetry for Wt}
    \frac{\Wt_{\id}(f)|_{g}}{\Wt_{\id}(g)|_{g}} \xlongleftrightarrow[\text{  $a$-$z$ swap  }]{\text{}} 
    (-1)^{\#f + \#g} \cdot
    \frac{\Wt_{\id}(g^!)|_{f^!}}{\Wt_{\id}(f^!)|_{f^!}},
  \end{equation}    
  where `$a$-$z$ swap' means that the two sides are equal after the identification of variables
  \[
  a_i\leftrightarrow z_{n+1-i}, \qquad\qquad z_i \leftrightarrow a_{m+1-i}.
  \]  
We wrote $(-)|_g$ for the restriction map $\phi_g$ to the torus fixed point $g$, and $\#f$ is the number of crossings of ties in the tie diagram of $f$.
\end{theorem}

It is an equivalent statement that the same equality holds for the $W$ functions, that is, for the stable envelopes. The difference between $\W$ and $\Wt$ cancels out in the fraction. 

Both sides of~\eqref{eq:mirror symmetry for Wt} are explicit formulas that can be programmed to a computer---the $\Wt$ functions are explicit in \S~\ref{sec:Wt} and the restriction maps in Theorem~\ref{thm:RestrictionMaps}. Hence Theorem~\ref{thm:mirror for Wt} is an equality between two explicit elliptic functions for any pair of 01-matrices of the same size and the same row and column sums.

\begin{remark}
   If $g$ is not below $f$ in the poset of fixed points for the chamber $\Chamb_{id}$, then the same holds for $f^!$ and $g^!$, and both sides of~\eqref{eq:mirror symmetry for Wt} are trivially 0. This shows that the identity is trivial $0=0$ either for the pair $(f,g)$ or for the pair $(g,f)$. 
\end{remark}

Consider $\Ch(r,c)$ with $r=(1,1,2,1)$, $c=(2,2,1)$. In Example~\ref{ex:poset} we presented the poset of the 12 torus fixed points for $\Chamb_{id}$. Using the notation of that poset graph here are the BCTs of three of the fixed points:
\[
f_3=
\begin{pmatrix}
    0 & 1 & 0 \\
    0 & 0 & 1 \\
    1 & 1 & 0 \\
    1 & 0 & 0 
\end{pmatrix}
\quad
f_9=
\begin{pmatrix}
    1 & 0 & 0 \\
    0 & 1 & 0 \\
    0 & 1 & 1 \\
    1 & 0 & 0 
\end{pmatrix}
\quad
f_{12}=
\begin{pmatrix}
    1 & 0 & 0 \\
    1 & 0 & 0 \\
    0 & 1 & 1 \\
    0 & 1 & 0 
\end{pmatrix}.
\]
Applying Theorem~\ref{thm:mirror for Wt} for $f=f_{12}$ and $g=f_{9}$ the identity we obtain (arranged to one side) is 
\begin{multline} \label{eq:f12f9}
\th(a_3/a_1)\th(\h^{-1})\th(a_1/a_2\ z_3/z_2)\th(a_2/a_3\ z_3/z_2\ \h^{-1})+
\\
\th(a_1/a_2)\th(z_2/z_3\ \h)\th(a_2/a_3\ \h^{-1})\th(a_3/a_1\ z_2/z_3) +
\\
\th(a_2/a_3)\th(z_3/z_2)\th(a_3/a_1\ z_2/z_3\ \h)\th(a_1/a_2\ \h)=0.
\end{multline}
This is, in fact, a very well know identity for $\th$ functions: the so-called {\em Fay's trisecant identity} \cite{Fay}. Usually it is phrased this way: if $x_1x_2x_3=y_1y_2y_3=1$ then 
\begin{equation}\label{eq:fay}
\delta(x_1,y_2)\delta(x_2,y_1^{-1})+\delta(x_2,y_3)\delta(x_3,y_2^{-1})+\delta(x_3,y_1)\delta(x_1,y_2^{-1})=0.
\end{equation}
where $\delta(x,y)=\th(xy)/(\th(x)\th(y))$. After multiplying with the denominators we have
\[
\th(x_3)\th(y_3)\th(x_1y_2)\th(x_2y_1^{-1})+
\th(x_1)\th(y_1)\th(x_2y_3)\th(x_3y_2^{-1})+
\th(x_2)\th(y_2)\th(x_3y_1)\th(x_1y_3^{-1})=0,
\]
which is the same as~\eqref{eq:f12f9} with
\[
\begin{array}{lcccl}
    x_1=a_1/a_2 & &&& y_1=z_2/z_3 \h \\
    x_2=a_2/a_3 & &&& y_2=z_3/z_2 \\
    x_3=a_3/a_1 & &&& y_3=\h^{-1}.
\end{array}
\]
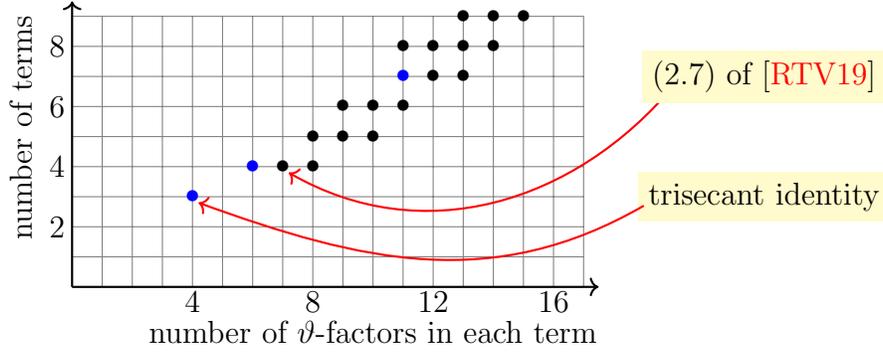
\begin{figure}
\[
\begin{tikzpicture}[scale=.4]
   \tkzInit[xmax=17,ymax=9,xmin=0,ymin=0]
   \tkzGrid
   \draw[thick,->] (0,0) to (17.5,0);
      \node at (10,-1.5)  {number of $\th$-factors in each term};
   \draw[thick,->] (0,0) to (0,9.5);
      \node[rotate=90] at (-1.7,5.3) {number of terms};
   \node[blue] at (4,3) {$\bullet$};
   \node[blue] at (6,4) {$\bullet$};
   \node at (7,4) {$\bullet$};
   \node at (8,4) {$\bullet$};
   \node at (8,5) {$\bullet$};
   \node at (9,5) {$\bullet$};
   \node at (9,6) {$\bullet$};
   \node at (10,5) {$\bullet$};
   \node at (10,6) {$\bullet$};
   \node at (11,6) {$\bullet$};
   \node[blue] at (11,7) {$\bullet$};
   \node at (11,8) {$\bullet$};
   \node at (12,7) {$\bullet$};
   \node at (12,8) {$\bullet$};
   \node at (13,7) {$\bullet$};
   \node at (13,8) {$\bullet$};
   \node at (14,8) {$\bullet$};
      \node at (13,9) {$\bullet$};
      \node at (14,9) {$\bullet$};
      \node at (15,9) {$\bullet$};
   \node at( 4,-.5) {$4$}; \node at( 8,-.5) {$8$}; \node at(12,-.5) {$12$}; \node at(16,-.5) {$16$};
   \node at (-.5,2) {$2$}; \node at (-.5,4) {$4$}; \node at (-.5,6) {$6$}; \node at (-.5,8) {$8$};
\draw (23,3) node [fill=yellow!25] {trisecant identity};
\draw[thick, red,->] (19,2.7) to [out=-150,in=-20] (4.2,2.8);
\draw[thick, red,->] (20,6.7) to [out=-130,in=-30] (7.2,3.8);
\draw (23,7) node [fill=yellow!25] {(2.7) of \cite{RTV}};
   \end{tikzpicture}
\]
\caption{The sizes of irreducible $\th$ function identities Theorem~\ref{thm:mirror for Wt} provides for small $r$ and $c$. Examples for those in blue are given below.} 
\label{fig:coordinates}
\end{figure}

The trisecant identity is the fundamental identity governing $\th$ functions, see eg. \cite{tatatheta}, \cite{raina}, \cite[\S5]{rimanyi2019hbardeformed} for various roles it plays in mathematics and physics. When expressed as a sum of products of $\th$ functions, it consists of three terms, each containing four factors. Beyond this well known 3-term-4-factor identity, Theorem~\ref{thm:mirror for Wt} generates numerous additional identities. In Figure~\ref{fig:coordinates}, we indicate the sizes of the identities discovered by our computer program for small values of $r$ and $c$. We display the {\em irreducible} identities only, omitting those that contain a proper subset of terms summing to zero.

For example, Theorem~\ref{thm:mirror for Wt} for $f=f_9$ and $g=f_3$ gives 
\begin{multline*} 
-\th(a_2/a_3)\th(a_4/a_3 \h)\th(z_2/z_3)\th(z_3/z_1 \h)\th(a_2/a_4 \ z_3/z_2\ \h )\th(a_2/a_3 \ z_2/z_1 \h)
\\
+\th(a_2/a_3)\th(a_2/a_4)\th(z_2/z_3)\th(z_1/z_2)\th(a_3/a_4\ z_3/z_2 )\th(a_2/a_3\  z_3/z_1\  \h^2)
\\
-\th(a_2/a_4)\th(a_2/a_3\ \h)\th(z_2/z_1\ \h)\th(z_3/z_2\ \h)\th(a_3/a_4\ z_3/z_2)\th(a_2/a_3\ z_3/z_1 \ \h)
\\
+\th(a_3/a_4)\th(a_2/a_3\ \h)\th(z_3/z_1\ \h)\th(z_3/z_2\ \h)\th(a_2/a_4\ z_3/z_2 )\th(a_2/a_3\ z_2/z_1 \ \h)
=0,
\end{multline*}
a 4-term-6-factor identity. For $f=f_{12}$ and $g=f_3$ Theorem~\ref{thm:mirror for Wt} gives a 7-term-11-factor identity:
\begin{multline*} 
-\no(\h)\no(a_{12})\no(a_{23})\no(a_{31}\h)\no(a_{42}\h)\no(z_{23})^2\no(z_{31}\h) \no(a_{23}z_{32}\h)\no(a_{14}z_{32}\h)\no(a_{23}z_{21}\h)\\
+\no(\h)\no(a_{12})\no(a_{23})^2\no(a_{14})\no(z_{31}\h)\no(z_{23})\no(z_{32}\h^2) \no(a_{13}z_{32})\no(a_{24}z_{32})\no(a_{23}z_{21}\h)\\
-\no(a_{12})\no(a_{23})^2\no(a_{42}\h)\no(a_{31}\h)\no(z_{12})\no(z_{23})^3\no(a_{23}z_{31}\h^2)\no(a_{14}z_{32}\h)\\
-\no(\h)^2\no(a_{23})\no(a_{14})\no(a_{12})\no(z_{23})\no(z_{12})\no(a_{23}z_{31}\h^2)\no(a_{32}z_{32}\h)\no(a_{24}z_{32})\no(a_{13}z_{32})\\
+\no(\h)\no(a_{23}\h)\no(a_{13})\no(a_{14})\no(a_{32}\h)\no(z_{32}\h)^2\no(z_{21}\h)\no(a_{23}z_{31}\h)\no(a_{12}z_{32})\no(a_{24}z_{32})\\
-\no(\h)\no(a_{12})\no(a_{23}\h)\no(a_{14})\no(a_{32}\h)\no(z_{31}\h)\no(z_{32}\h)^2\no(a_{13}z_{32})\no(a_{23}z_{21}\h)\no(a_{24}z_{32})\\
+\no(a_{31}\h)\no(a_{12}\h)\no(a_{23}\h)\no(a_{24})\no(a_{23})
\no(z_{21}\h)\no(z_{32}\h)^2\no(z_{23})
\no(a_{14}z_{32})\no(a_{23}z_{31}\h)=0,
\end{multline*}
where we used the shorthand notations $a_{ij}=a_i/a_j$, $z_{ij}=z_i/z_j$, as well as $(x)=\th(x)$.

\bibliographystyle{alpha}
\bibliography{sample}

\end{document}